\documentclass[11pt]{article}
\usepackage{amsmath, amsfonts, amsthm, amssymb, color}
\usepackage{graphicx}
\usepackage{float}
\usepackage{verbatim}

\allowdisplaybreaks

\numberwithin{equation}{section}

\newtheorem{lemma}{Lemma}[section]

\newtheorem{prop}[lemma]{Proposition}
\newtheorem{thm}[lemma]{Theorem}

\theoremstyle{definition}

\theoremstyle{remark}
\newtheorem{remark}[lemma]{Remark}

\def\R{\mathbb{R}}
\def\N{\mathbb{N}}
\def\t{\mathbf{t}}

\def\L{\mathcal{L}}
\def\O{\mathcal{O}}
\def\a{\mathbf{a}}
\def\b{\mathbf{b}}
\def\i{\mathbf{i}}
\def\j{\mathbf{j}}
\def\k{\mathbf{k}}

\def\p{\mathbf{p}}

\def\tt{\mathcal{T}}
\def\A{\mathcal{A}}
\def\I{\mathcal{I}}
\def\O{\mathcal{O}}
\def\S{\mathcal{S}}
\def\x{\mathbf{x}}
\def\p{\mathbf{p}}

\numberwithin{equation}{section} \numberwithin{table}{section}

\title{Quantitative recurrence and the shrinking target problem for overlapping iterated function systems}
\author{Simon Baker$^1$ and Henna Koivusalo$^2$\\\\
	\emph{$^{1}$Department of Mathematical Sciences,} \\ \emph{Loughborough University,} \\ \emph{Loughborough, LE11 3TU, UK} \\ Email: simonbaker412@gmail.com\\ \\
	\emph{$^2$School of Mathematics,} \\ \emph{University of Bristol,} \\
	\emph{Bristol, BS8 1UG, UK}\\ Email: henna.koivusalo@bristol.ac.uk\\ \\
}

\date{\today}
\begin{document}
	\maketitle

	\begin{abstract}
		In this paper we study quantitative recurrence and the shrinking target problem for dynamical systems coming from overlapping iterated function systems. Such iterated function systems have the important property that a point often has several distinct choices of forward orbit. As is demonstrated in this paper, this non-uniqueness leads to different behaviour to that observed in the traditional setting where every point has a unique forward orbit.
		
		We prove several almost sure results on the Lebesgue measure of the set of points satisfying a given recurrence rate, and on the Lebesgue measure of the set of points returning to a shrinking target infinitely often. In certain cases, when the Lebesgue measure is zero, we also obtain Hausdorff dimension bounds. One interesting aspect of our approach is that it allows us to handle targets that are not simply balls, but may have a more exotic geometry. 
		%	We analyse the Lebesgue measure of recurrent points and points returning to shrinking targets on generic (in the sense of Falconer) self-affine sets with positive Lebesgue measure. As a corollary we also obtain Hausdorff dimension bounds. \textbf{(Simon - More details here?)}
		
	\end{abstract}
	

	\section{Introduction}
	\subsection{Dynamical systems with choice} For many dynamical systems one may be interested in, from both a pure and applied perspective, one may encounter situations where it is natural for an element of the state space to have a choice for its forward trajectory. This phenomenon can be observed in the setting of random walks, and similarly in the setting of applied models for real world events that have some in-built randomness. This paper is motivated by the following general question: 
	%This phenomenon can be modelled in the following way: Let $X$ and $Y$ be two sets, and let $f_{1}:X\to X$ and $f_{2}:X\times Y\to Y$ be two maps. We then associate the map $T:X\times Y\to X\times Y$ given by $T(x,y)=(f_1(x),f_{2}(x,y)).$ The pair $(T,X\times Y)$ is known as a skew product. Loosely speaking, we should think that when there is a choice of forward trajectory for the element in the $y$-coordinate, then the $x$-coordinate tells us which choice to make. Importantly, different $x$ coordinates may lead to different forward orbits in the $y$ coordinate. 
	
	\begin{center}
		For such dynamical systems,	what extreme behaviours do we observe if instead of considering a single choice of forward trajectory, we consider all forward trajectories?
	\end{center}
	In this paper we study this question in the context of dynamical systems arising from iterated function systems. The extreme behaviours we are interested come from the shrinking target problem and quantitative recurrence. Before detailing our results we provide the relevant background from these areas and from Fractal Geometry. 
	
	\subsection{Fractal Geometry}
	Given a finite set of invertible $d\times d$ matrices $\{A_i\}_{i\in \I}$ satisfying $\|A_i\|<1$ for all $i\in \I$, and a finite set of vectors $\{t_i\}_{i\in \I}$ each belonging to $\R^d$, we can associate the set of contracting maps $\{S_i:\mathbb{R}^d\to\mathbb{R}^d\}_{i\in \I}$ where each $S_i$ is given by $S_{i}(x)=A_ix+t_i.$ We call $\{S_i\}_{i\in \I}$ an affine iterated function system or IFS for short. Importantly, one can associate to any IFS a unique, non-empty, compact set $X$ satisfying 
	\[
	X=\bigcup_{i\in \I}S_{i}(X). 
	\]	
	The set $X$ is called the self-affine set or invariant set associated to $\{S_i\}_{i\in \I}$. These $X$ often exhibit fractal like behaviour, and are a well studied family of fractal sets (see \cite{Fal}). 
	
	Given an IFS $\{S_i\}_{i\in \I}$, for each $i\in \I$ we let $T_{i}:\mathbb{R}^d\to\mathbb{R}^d$ be the the map given by $T_{i}(x)=S_{i}^{-1}(x)$. It follows from the definition of a self-affine set, that for any $x\in X$ there exists $i\in \I$ such that $T_{i}(x)\in X$. When $S_{i}(X)\cap S_{j}(X)=\emptyset$ for all $i\neq j$ this choice of $T_i$ is unique for all $x$. This means that under this assumption we have a well defined map $T:X\to X$ given by $T(x)=T_{i}(x)$ for $x\in S_i(X)$. When there exists distinct $i$ and $j$ satisfying $S_{i}(X)\cap S_{j}(X)\neq\emptyset,$ we say that the iterated function system is overlapping. In this case there exists $x$ for which we have a choice of $T_i$ satisfying $T_{i}(x)\in X$. In other words, there exists $x\in X$ for which we have a choice of forward trajectory. It is these overlapping iterated function systems that will be the main focus of this paper, and for which we will consider our motivating question.
	
	We finish this discussion on iterated function systems, and in particular on overlapping iterated functions systems, by emphasising that the study of these objects is currently a very active area of research. We refer the reader to the articles \cite{Hochman2,Hochman,Shmerkin,Shm,Var2,Var3} for recent advances in the study of overlapping iterated function systems and their associated self-affine sets.

	%	To see the connection between iterated function systems and the skew product dynamical systems mentioned above, we will now define a skew product which can be seen to generate all forward orbits in the special case where $\I=\{1,2\}$ and $S_{1}(X)\cap S_{2}(X)\neq \emptyset$. Let $\sigma:\{1,2\}^{\N}\to \{1,2\}^{\N}$ be the usual shift map given by $\sigma((i_n))=(i_{n+1})$. Let $T:\{1,2\}^{\N}\times X$ be defined as follows:
	%	\[ T((i_n),x) = \left\{ \begin{array}{ll}
		%		(\sigma((i_n)),T_{1}(x)) & \mbox{if $x \in S_{1}(X)\setminus S_{2}(X)$};\\
		%		(\sigma((i_n)),T_{i_1}(x)) & \mbox{if $x \in S_1(X)\cap S_{2}(X)$};\\
		%		(\sigma((i_n)),T_{2}(x)) & \mbox{if $x \in S_{2}(X)\setminus S_{1}(X)$}.\end{array} \right. \] 
	%It can be shown that all possible forward orbits of an element of $X$ can be generated by $T$. For larger alphabets $\I$ one can define an analogous skew product which can be seen to generate all forward orbits. 

	\subsection{Shrinking targets}	The general framework for shrinking target problems in $\mathbb{R}^d$ is the following: Let $T:X\to X$ be a continuous map defined on some Borel set $X\subset \mathbb{R}^d$. Given a sequence of points $\mathbf{x}=(x_n)_{n=1}^{\infty}\in X^{\N}$ and $(E_n)_{n=1}^{\infty}$ a sequence of Borel subsets of $\mathbb{R}^d,$  we associate the set 
	\[
	W(\mathbf{x},(E_n)):= \{x\in X:T^n(x)\in x_{n}+E_n \textrm{ for i.m. }n\in \mathbb{N}\}.
	\] 
	Here and throughout we use i.m. as a shorthand for infinitely many. Often $(E_n)$ is taken to be a nested sequence of sets containing the origin, and $\x$ is a constant sequence. As such the study of the sets $W(\mathbf{x},(E_n))$ is commonly known as the shrinking target problem. Typically one is interested in establishing measure-theoretic and dimension results for the sets  $W(\mathbf{x},(E_n))$. If one were to equip $X$ with a Borel probability measure $\mu$, one can try to determine $\mu(W(\mathbf{x},(E_n)))$. If $\sum_{n=1}^{\infty}\mu(\{x:T^n(x)\in x_{n}+E_n\})<\infty,$ then it is a simple consequence of the Borel Cantelli lemma that $\mu(W(\mathbf{x},(E_n)))=0.$ If $\sum_{n=1}^{\infty}\mu(\{x:T^n(x)\in x_{n}+E_n\})=\infty,$ then determining $\mu(W(\mathbf{x},(E_n)))$ is a much more challenging problem. Generally speaking, if the dynamical system $T:X\to X$ is mixing sufficiently quickly with respect to $\mu,$ then $\sum_{n=1}^{\infty}\mu(\{x:T^n(x)\in x_{n}+E_n\})=\infty$ implies $\mu(W(\mathbf{x},(E_n)))=1$. Numerous results exist verifying this principle, see for instance \cite{CheKle,Phi}. For more background on shrinking target problems we refer the reader to \cite{ BarTro,CheKle,HillVel,HillVel2,KoLiRa, KoiRam, LLVZW} and the references therein. 
	
	Our framework for studying shrinking targets for overlapping iterated function systems is the following: Given an IFS $\S=\{S_i\}_{i\in \I}$ with self-affine set $X$, a sequence $\x\in X^{\N}$, and a sequence of Borel sets $(E_n)$, we let 
	$$W(\S,\x,(E_n)):=\left\{x\in X:(T_{i_N}\circ \cdots \circ T_{i_1})(x)\in x_{N}+E_{N} \textrm{ for i.m. }(i_1,\ldots,i_N)\in \bigcup_{n=1}^{\infty}\I^n\right\}.$$ 
	%As we will see, the Lebesgue measure of $W(\S,\x,(E_n))$ will often be determined by the Lebesgue measure of the sets $(E_n)$. An important role is played by a parameter $\lambda(\A)$ that we now define. 
	
	Recall that to define an IFS we need a finite set of contracting matrices $\A=\{A_i\}_{i\in \I}$ and a finite set of translation vectors $\{t_i\}_{i\in \I}$. Each $S_i$ then satisfies $S_i(x)=A_{i}x+t_i$. As such given an IFS $\S=\{S_i\}_{i\in \I}$ with corresponding set of matrices $\A=\{A_i\}_{i\in \I},$ we can define 
	\begin{equation}	\label{eq:lambda}
		\lambda(\A)=\sum_{i\in \I}|Det(A_i)|.
	\end{equation}	
	We will often suppress the dependence of $\lambda(\A)$ on $\S$ from our notation. For the family of IFSs we will be interested in, the parameter $\lambda(\A)$ will determine whether $X$ typically has positive Lebesgue measure. In the context of shrinking targets, $\lambda(\A)$ will determine the fastest rate\footnote{On the exponential scale.} at which the Lebesgue measure of $E_n$ can converge to zero, yet we could still hope for $W(\S,\x,(E_n))$ to have positive Lebesgue measure (see Theorem \ref{Main thm}, Theorem \ref{exponential theorem}, and Theorem \ref{Convergent sum}). We will introduce some additional notation in the special case when $(E_n)$ is a sequence of balls centred at the the origin: Given an IFS $\S,$ $\mathbf{x}\in X^{\N},$ and $h:\mathbb{N}\to [0,\infty),$ we associate the set
	$$\left\{x\in X:(T_{i_N}\circ \cdots \circ T_{i_1})(x)\in B\left(x_{N}, \left(\frac{h(N)}{\lambda(\A)^N}\right)^{1/d}\right) \textrm{ for i.m. } (i_1\ldots i_N)\in \bigcup_{n=1}^{\infty}\I^n\right\}.$$
	We denote this set by $W(\S,\mathbf{x},h).$ In the special case when $\x=(y)_{n=1}^{\infty}$ is a constant sequence, we adopt the simpler notation $W(\S,y,(E_n))$ for $W(\S,\x,(E_n)),$ and $W(\S,y,h)$ for $W(\S,\mathbf{x},h).$

	%	Of particular interest will be the case when $(E_n)$ is a sequence of Borel sets centred at the origin. In that case 
	
	%When dealing will such sets we adopt the following notation: Given an IFS $S$, $y\in X$, and %$h:\mathbb{N}\to [0,\infty)$ let
	%$$W(\S,y,h):=\left\{x\in X:d((T_{i_N}\circ \cdots T_{i_1})(x),y)\leq \left(\frac{h(N)}{\lambda(\A)^N}\right)^{1/d} \textrm{ for i.m. } (i_1\ldots i_N)\in \cup_{n=1}^{\infty}\I^n\right\}.$$ 
	
	\subsection{Recurrence}	The notion of recurrence is fundamental in dynamical systems. If $T:X\to X$ is a measure preserving transformation of a probability space $(X,\mathcal{B},\mu)$, the famous Poincar\'{e} recurrence theorem states that for any set $A\in\mathcal{B}$ satisfying $\mu(A)>0$, we have that $\mu$-almost every $x\in A$ satisfies $T^{n}(x)\in A$ for infinitely many $n\in \N$ (see \cite{Wal} for a proof of this statement). Under relatively weak assumptions this measure-theoretic statement can be seen to imply a metric one. If $X$ is equipped with a metric $d$ so that $(X,d)$ is separable and $\mathcal{B}$ is the Borel $\sigma$-algebra, then Poincar\'{e}'s recurrence theorem implies that for $\mu$-almost every $x\in X$ we have $$\liminf_{n\to\infty} d(T^n(x),x)=0.$$ It is natural to wonder whether this metric statement can be further improved upon and replaced with something more quantitative. With this goal in mind the following framework is natural: Let $X$ and $T$ be as above. Given a function $\psi:\N\to [0,\infty)$, we define the set of points that return at the rate $\psi:$
	\[
	R(\psi):=\{x\in X:d(T^n(x),x)\leq \psi(n)\,\textrm{ for i.m. }n\in \N\}.
	\]
	Just as for shrinking target sets, we are typically interested in establishing measure-theoretic and dimension results for the sets $R(\psi)$. The first result in this direction is due to Boshernitzan \cite{Bos}. He proved that if $(X,d)$ is a separable metric space, $\mu$ is a Borel probability measure, and $T:X\to X$ is a measure preserving transformation, then if the $\alpha$-dimensional Hausdorff measure is $\sigma$-finite on $(X,d)$, for $\mu$-almost every $x\in X$ we have 
	\[
	\liminf_{n\to\infty} n^{1/\alpha}d(T^n(x),x)<\infty.
	\]
	Moreover, if we also assume that $\mathcal{H}^{\alpha}(X)=0$, then Boshernitzan proved that for $\mu$-almost every $x\in X$ we have $$\liminf_{n\to\infty} n^{1/\alpha}d(T^n(x),x)=0.$$ The $\mu$-almost everywhere rate of recurrence Boshernitzan's result provides only depends upon the metric properties of the set $X$. However, it is natural to expect that the measure $\mu$ would also influence the recurrence behaviour of a $\mu$-typical point. This issue was addressed in a paper of Barreira and Saussol \cite{BarSau}. They proved that if $T:X\to X$ is a Borel measurable map defined on $X\subset\mathbb{R}^d$,  and $\mu$ is a $T$-invariant probability measure, then for $\mu$-almost every $x\in X$ we have $$\liminf_{n\to\infty}n^{1/\alpha}d(T^n(x),x)=0\textrm{ for any }\alpha>\liminf_{r\to 0}\frac{\log \mu(B(x,r))}{\log r}.$$  More recently, a number of papers have appeared that bring the quantitative recurrence theory more closely in line with the shrinking target theory. In particular, it has been shown that $\mu(R(\psi))$ is related to the convergence/divergence properties of $\sum_{n=1}^{\infty}\psi(n)^{\delta}$ for some appropriate $\delta>0$  (see \cite{BakFar,BarTro,ChangWuWu,HLSW,KKP,KleZhe}). A common assumption appearing in each of these papers was a uniformity assumption on the local behaviour of the measure $\mu$. It turns out that this assumption is essential if one wants $\mu(R(\psi))$ to be related to the convergence/divergence properties of $\sum_{n=1}^{\infty}\psi(n)^{\delta}$. In a recent paper together with Allen and B\'ar\'any \cite{ABB}, the first author proved that if $T:X\to X$ is the left shift on a topologically mixing subshift of finite type and $\mu$ is the Gibbs measure for some H\"{o}lder continuous potential satisfying a weak non-degeneracy condition, then there exists $\psi$ for which $\mu(R(\psi))$ is not determined by the convergence/divergence properties of $\sum_{n=1}^{\infty}\psi(n)^{\delta}$ for any $\delta$.
	
	For iterated function systems we study the following recurrence sets. Given an IFS $\S$ and a sequence of Borel sets $(E_n)$, we let 
	$$R(\S,(E_n)):=\left\{x\in X:(T_{i_N}\circ \cdots \circ T_{i_1})(x)\in x+E_N \textrm{ for i.m. } (i_1\ldots i_N)\in \bigcup_{n=1}^{\infty}\I^n \right\}.$$
	We also introduce the following more specific framework when our recurrence neighbourhoods are balls. Given an IFS $\S$ and $h:\mathbb{N}\to [0,\infty),$ we let $$R(\S,h):=\left\{x\in X:(T_{i_N}\circ \cdots \circ T_{i_1})(x)\in B\left(x, \left(\frac{h(N)}{\lambda(\A)^N}\right)^{1/d}\right)\, \textrm{ for i.m. } (i_1\ldots i_N)\in \bigcup_{n=1}^{\infty}\I^n\right\}.$$ Just as for the shrinking target problem, the parameter $\lambda(\A)$ plays an important role in the study of quantitative recurrence. On the exponential scale it describes the best rate of recurrence we could hope to observe for a Lebesgue typical point (see Theorem \ref{Main thm} and Theorem \ref{exponential theorem}). %Given an IFS $\S$ and a sequence of sets $(E_n)$ we let 
	%$$R(\S,(E_n)):=\left\{x\in X:(T_{i_N}\circ \cdots \circ T_{i_1})(x)\in x+E_n\, \textrm{ for i.m. } (i_1\ldots i_N)\in \cup_{n=1}^{\infty}\I^n\right\}.$$ We will pay special attention to the case where $(E_n)$ is a sequence of balls centred at the origin. For this special case we have the following notation. 
	
	\section{Statements of results}\label{sec:statements}
	We begin this section by introducing some more notation relating to iterated function systems. Given a finite set of invertible $d\times d$ matrices $\A=\{A_i\}_{i\in \I}$ satisfying $\|A_i\|<1$ for all $i\in \I$, one can associate a parameterised family of iterated function systems in the following way. To each $\tt=(t_1,\ldots, t_{\I})\in\mathbb{R}^{\#\I\cdot d}$ we associate the IFS  $\mathcal{S}_{\tt}=\{S_{i,\tt}(x)=A_ix+t_i\}_{i\in \I}.$ We also let $T_{i,\tt}=S_{i,\tt}^{-1}$ for each $i\in \I$ and $\tt$. We denote the corresponding self-affine set by $X_{\tt}.$ To each $\tt\in \mathbb{R}^{\# \I\cdot d}$ we associate a surjective projection map $\pi_{\tt}:\I^{\N}\to X_{\tt}$ given by $$\pi_{\tt}(\i)=\lim_{n\to\infty}(S_{i_1,\tt}\circ \cdots \circ S_{i_{n},\tt})(0)=t_{i_1}+\sum_{n=1}^{\infty}(A_{i_1}\circ \cdots \circ A_{i_{n}})t_{i_{n+1}}.$$ When we equip $\I^{\N}$ with the product topology then $\pi_{\tt}$ is also continuous. 
	
	This parameterised family of IFSs was originally studied by Falconer \cite{Fal2}. Under the assumption that $\|A_i\|<1/3$ for all $i\in \I$, he showed that for Lebesgue almost every $\tt$ the Hausdorff dimension of $X_{\tt}$ is equal to the affinity dimension of $\A$. See \cite{Fal2} for the definition of the affinity dimension. Similarly, he proved that when the parameter $\lambda(\A)$ from \eqref{eq:lambda} is strictly greater than $1$, then $X_{\tt}$ has positive Lebesgue measure for Lebesgue almost every $\tt$. Falconer's results were extended by Solomyak in \cite{Sol2} to the case where the set of matrices satisfies the weaker assumption that  $\|A_i\|<1/2$ for all $i\in \I$. For a set of matrices with $\lambda(\A)>1$, so that $X_\tt$ typically has positive Lebesgue measure, 
	it is natural to then ask whether $X_\tt$ also has interior points. Not much is known about this, but very recently, sufficient conditions for $X_\tt$ to almost surely have non-empty interior were given in \cite{FenFen}.

	We are interested in proving results on the Lebesgue measure of shrinking target sets and quantitative recurrence sets for $\S_{\tt}$ that hold for Lebesgue almost every $\tt$. Proving statements that only hold for Lebesgue almost every $\tt$ might at first appear to be unnatural. Particularly when compared to the traditional shrinking target framework where every point in the domain has a unique forward orbit. However, Theorem \ref{Exact overlap} and the discussion after its proof demonstrate that such a restriction is entirely necessary. Often there exists a dense set of $\tt$ for which the corresponding shrinking target sets and quantitative recurrence sets exhibit different behaviour to that observed for a Lebesgue typical $\tt$. 
	
	%Given a finite set of matrices $\A=\{A_i\}_{i\in \I}$ we let $$\lambda(\A)=\sum_{i\in \I}Det(A_i).$$ For any function $\psi:\mathbb{N}\to[0,\infty)$ there exists $h:\mathbb{N}\to [0,\infty)$ such that $$\psi(n)=\left(\frac{h(n)}{\lambda(\A)^n}\right)^{1/d}\qquad \forall n\in\mathbb{N}.$$ As we will see  $(\frac{1}{\lambda(\A)^n})^{1/d},$ represents the best approximation one can expect to hold for a typical point in the self-affine set. 
	
	%The function $h$ should be thought of as providing a sub-exponential improvement over this rate of approximation. To prove positive results it is necessary to impose some additional conditions on the function $h$. These are detailed below. 
	
	To prove positive measure results for $W(\S_{\tt},\mathbf{x},h)$ and $R(\S_{\tt},h)$ it is necessary for us to impose the following additional assumptions on the function $h$. Given $B\subset \N$ we define the {\it upper density} of $B$ to be $$\overline{d}(B):=\limsup_{n\to\infty}\frac{\#\{1\leq j\leq n:j\in B\}}{n}.$$ Given $\epsilon>0$ we let $$H_{\epsilon}:=\left\{h:\N\to [0,\infty):\sum_{n\in B}h(n)=\infty, \forall B\subset \N\, \textrm{ such that }\, \overline{d}(B)>1-\epsilon\right\}.$$ We then define 
	\begin{equation}\label{eq:H}
		H=\bigcup_{\epsilon>0}H_{\epsilon}.
	\end{equation} 
	We emphasise that $h:\mathbb{N}\to [0,\infty)$ given by $h(n)=1/n$ is contained in $H$. The reader may find it instructive to keep this example in mind. We say that $h:\mathbb{N}\to [0,\infty)$ is {\it decaying regularly} if  $$\inf_{n\in\mathbb{N}}\frac{h(n+1)}{h(n)}>0.$$ 
	With all of this terminology established we can now state our results.
	\begin{thm}
		\label{Main thm}
		Let $\A=\{A_i\}_{i\in \I}$ be a finite set of invertible $d\times d$ matrices satisfying the following properties:
		\begin{itemize}
			\item There exists a positive diagonal matrix $A$ such that $A_i=A$ for all $i\in \I$.
			\item $\|A\|<1/2$.
			\item $\lambda(\A)>1$.
			%	\item There exists $\lambda_1,\mu>0$ such that $A=\begin{pmatrix}
				%		\lambda_1 & 0 & \cdots & 0 \\
				%		0 & \lambda_2 & \cdots & 0 \\
				%		0 &  \cdots & \cdots &\lambda_d 
				%	\end{pmatrix}.$
		\end{itemize} Then the following statements are true:
		\begin{enumerate}
			\item  For Lebesgue almost every $\tt\in \R^{\#\I \cdot d}$, for any $\x\in X_{\tt}^{\N}$ and $h\in H$ we have that $W(\mathcal{S}_{\tt},\x,h)$ has positive Lebesgue measure.
			\item  For Lebesgue almost every $\tt\in \R^{\#\I \cdot d}$, for any $x\in X_{\tt}$ and $h\in H$ that is decaying regularly and decreasing, we have that Lebesgue almost every $y\in X_{\tt}$ is contained in $W(\mathcal{S}_{\tt},x,h)$.
			%\item  For Lebesgue almost every $\tt$, for any $(i_{n,m})_{(n,m)\in \N\times \N}\in \I^{\N\times \N}$ and $h\in H$ we have that $W(\mathcal{S}_{\tt},(\pi_{\tt}((i_{n,m})_{m=1}^{\infty})_{n=1}^{\infty},h)$ has positive Lebesgue measure.
			%	\item For Lebesgue almost every $\tt$, for any $\i\in \I^{\N}$ and $h\in H$ that is decaying regularly and decreasing we have that Lebesgue almost every $x\in X_{\tt}$ is contained in $W(\mathcal{S}_{\tt},\pi_{\tt}(\i),h)$.
			\item Let $U$ be an open subset of $\mathbb{R}^{\# \I \cdot d}$ and $\i\in \I^{\N}$. Assume that $\pi_{\tt}(\i)\in int(X_{\tt})$ for Lebesgue almost every $\tt\in U$. Then for Lebesgue almost every $\tt\in U$, for any $h\in H$ that is decaying regularly and decreasing, we have that Lebesgue almost every $x\in X_{\tt}$ is contained in $$\left\{x: \exists \j\in\I^{\N} \textrm{ such that }(T_{j_{n},\tt}\circ \cdots \circ T_{j_1,\tt})(x)\in B\left(\pi_{\tt}(\i),\left(\frac{h(n)}{ \lambda(\A)^{n}}\right)^{1/d}\right)\textrm{ for i.m. } n\in \N\right\}.$$
			\item For Lebesgue almost every $\tt\in \R^{\#\I \cdot d}$, for any $h\in H$ we that $ R(\mathcal{S}_{\tt},h)$ has positive Lebesgue measure. 		
		\end{enumerate}
	\end{thm}
	We emphasise that in our second and third statements we are only considering targets centred at a single point. Moreover, the third statement only applies when $int(X_\tt)\neq \emptyset$ for Lebesgue almost every $\tt$ belonging to the open subset $U$. The significance of our third statement is that it gives sufficient conditions so that infinitely many targets are hit if we only consider the orbit arising from a single sequence $\j\in \I^{\N}$. 
	%Rephrasing this statement in the language of skew products, this statement gives sufficient conditions so that for a typical $y\in Y$ there exists $x\in X$ such that the second coordinate of $T^{n}(x,y)$ hits our sequence of targets infinitely often.

	For a general set of matrices $\A$ satisfying $\|A_i\|<1/2$ for all $i\in \I,$ and for a general sequence of Borel sets $(E_n)$, we are not able to prove the appropriate analogue of Theorem \ref{Main thm}. However, if we weaken our expectations and restrict to sets $(E_n)$ with Lebesgue measure satisfying $\L_{d}(E_n)=\lambda(\A)^{-n},$ then we are able to prove a positive result. Importantly, this result also holds for a more general set of matrices.  This is the content of Theorem \ref{exponential theorem} below. In the statement of this theorem we will make use of the following notation. Given $C>0$ we let $W(\S,\x,C)$ and $R(\S,C)$ denote the shrinking target set and recurrent set corresponding to the constant function $h(n)=C$. Before stating Theorem \ref{exponential theorem} we define a useful notion that will be used in the proofs of Theorem \ref{Main thm} and Theorem \ref{exponential theorem}. This notion will allow us to upgrade certain positive measure statements to full measure statements.
	
	We call a family of open sets $\mathcal{V}$ for which the following properties hold a {\it density basis} for a Borel set $X$:
	\begin{itemize}
		\item For all $x\in X$ there are arbitrarily small $V\in \mathcal{V}$ containing $x.$
		\item For any Borel set $A\subset X$ we have $$\lim_{V\to x, x\in V\in\mathcal{V}}\frac{\L_{d}(V\cap A)}{\L_{d}(V)}=\chi_{A}(x)\, \textrm{ for Lebesgue almost every }x\in\mathbb{R}^d.$$
	\end{itemize}
	By $V\to x$ we mean that $\textrm{diam}(V\cup \{x\})\to 0$. Let $X$ be the self-affine set of an iterated function system. Assume that $\L_{d}(X)>0$. We say that the self-affine set $X$ is {\it differentiation regular} if there exists a density basis $\mathcal{V}$ for $X$ and a constant $\eta>0,$ such that for every $x\in X$ there exists a sequence $\{V_j(x)\}$ in $\mathcal{V}$ for which the following properties are satisfied:
	\begin{itemize}
		\item $x\in V_{j}(x)$ for all $j.$
		\item $V_{j}(x)\to x$ as $j\to \infty$
		\item For each $V_{j}(x)$ there exists a word $(i_1,\ldots,i_n)$ such that $$(S_{i_1}\circ \cdots \circ S_{i_n})(X)\subset V_{j}(x),\,\, \L_{d}((S_{i_1}\circ \cdots \circ S_{i_n})(X))\geq \eta \L_{d}(V_{j}(x)).$$
	\end{itemize} In \cite{Shm} Shmerkin proved the following statement.
	
	\begin{lemma}
		\label{Shmerkin lemma}
		Let $\{S_i\}_{i\in \I}$ be an iterated function system with self-affine set $X$. Assume that $\L_{d}(X)>0$ and that one of the following properties are satisfied:
		\begin{enumerate}
			\item $d=2$ and all the matrices $\{A_i\}$ are equal.
			\item All the matrices $A_i$ are simultaneously diagonalisable. 
			\item There is a finite set $W\subset \R^d$ of at least $d$ linearly independent elements such that $A_i(W)\subset W$ for all $i\in \I$. 
		\end{enumerate}
		Then $X$ is differentiation regular.
	\end{lemma}
	Equipped with the notion of differentiation regular we can now state Theorem \ref{exponential theorem}.
	%the following holds: for every $x\in K$ there is a sequence $\{V_{j}(x)\}$ in $\mathcal{V}$ with $x\in V_{j}(x)$, $V_{j}(x)\to x$, such that if $V=V_{j}(x)$ for some $j$, there exists a finite word $u$ satisfying $$S_{u}(X)\subset V,\,\, \L(S_u(X))\geq \eta \L(V).$$ In \cite{Shm} Shmerkin proved the following.

	\begin{thm}
		\label{exponential theorem}
		Let $\A=\{A_i\}_{i\in \I}$ be a finite set of invertible $d\times d$ matrices satisfying the following properties:
		\begin{itemize}
			\item $\|A_i\|<1/2$ for all $i\in \I$.
			\item $\lambda(\A)>1$.
		\end{itemize}
		%Let $\psi:\mathbb{N}\to [0,\infty)$ be given by $$\psi(n)=\left(\frac{1}{\lambda(\A)^n}\right)^{1/d}.$$
		Then the following statements are true:
		\begin{enumerate}
			\item  Let $(i_{n,m})_{(n,m)\in \N\times \N}\in \I^{\N\times \N}$ be the sequence of centres for the shrinking targets, and $(E_n)$ be a sequence of Borel sets satisfying the following properties:
			\begin{itemize}
				\item There exists $Q>0$ such that $E_n\subset [-Q,Q]^d$ for all $n\in \N$.
				\item $\L_{d}(E_n)=\lambda(\A)^{-n}$ for all $n.$
				\item For all $s\in(0,1)$ and $n\in \N$ we have $s\cdot E_n\subset E_n$.
				\item There exists $C>0$ such that for each $n\in \N$ there exists $r_n>0$ satisfying  $\L_{d}\left([-r_n,r_n]^d\pm E_n\right)\leq C\L_{d}(E_n)$. 
			\end{itemize} Then for Lebesgue almost every $\tt\in \R^{\# \I\cdot d},$ we have
			$$\limsup_{n\to\infty}\L_{d}\left(\left\{x:\exists (i_1,\ldots,i_n)\in \I^n \textrm{ such that } (T_{i_n,\tt}\circ \cdots \circ T_{i_1,\tt})(x)\in \pi_{\tt}((i_{n,m})_{m=1}^{\infty})+E_n\right\}\right)>0.$$ Moreover, for Lebesgue almost every $\tt\in \R^{\# \I\cdot d},$ the set $W(\mathcal{S}_{\tt},(\pi_{\tt}((i_{n,m})_{m=1}^{\infty}))_{n=1}^{\infty},(E_n))$ has positive Lebesgue measure.
			\item Let $\i\in \I^{\N}$ and assume that $X_{\tt}$ is differentiation regular for Lebesgue almost every $\tt\in \R^{\# \I\cdot d}$. Then for Lebesgue almost every $\tt\in \R^{\# \I\cdot d},$ Lebesgue almost every $x\in X_{\tt}$ is contained in the set
			$$\bigcup_{C>0}W(S_{\tt},\pi_{\tt}(\i),C).$$
			%$$\{x:\exists C>0 \textrm{ such that }(T_{j_{N}}\circ \cdots \circ T_{j_1})(x)\in B(\pi_{\tt}(\i),C\psi(N))\textrm{ for i.m. }(j_{1},\ldots, j_{N})\in \cup_{N=1}^{\infty}\I^{N}\}.$$
			%$W((\mathcal{S}_{\tt},(\pi_{\tt}((i_{m})_{m=1}^{\infty}))_n,\psi)$.
			\item  Let $(E_n)$ be a sequence of Borel sets satisfying the following properties:
			\begin{itemize}
				\item There exists $Q>0$ such that $E_n\subset [-Q,Q]^d$ for all $n\in \N$.
				\item $\L_{d}(E_n)=\lambda(\A)^{-n}$ for all $n.$
				\item For all $s\in(0,1)$ and $n\in \N$ we have $s\cdot E_n\subset E_n$.
				\item There exists $C>0$ such that for each $n\in \N$ there exists $r_n>0$ satisfying  $\L_{d}\left([-r_n,r_n]^d\pm E_n\right)\leq C\L_{d}(E_n)$.
			\end{itemize}
			Then for Lebesgue almost every $\tt\in \R^{\# \I\cdot d},$ we have
			$$\limsup_{n\to\infty}\L_{d}\left(\left\{x:\exists (i_1,\ldots,i_n)\in \I^n \textrm{ such that } (T_{i_n,\tt}\circ \cdots \circ T_{i_1,\tt})(x)\in x+E_n\right\}\right)>0.$$ Moreover, for Lebesgue almost every $\tt\in \R^{\# \I\cdot d},$ the set $R(\mathcal{S}_{\tt},(E_n))$ has positive Lebesgue measure. 
		\end{enumerate}
	\end{thm}
	Lemma \ref{Shmerkin lemma} lists a number of conditions on the iterated function system under which $X_\tt$ is differentiation regular, and the statement 2. of the above Theorem holds. 	
	
	As previously mentioned, $\lambda(\A)^{-n}$ is the fast rate at which the Lebesgue measure of $E_{n}$ can converge to zero, yet we could hope for $W(S,\x,(E_n))$ to have positive Lebesgue measure. If one considers a faster rate, it is natural to wonder whether one can obtain Hausdorff dimension results instead. With this goal in mind we introduce the following framework. Given an IFS $S$, $\x\in X^{\N}$ and $s>1$, we let 
	$$\left\{x\in X: \left(T_{i_N}\circ \cdots \circ T_{i_1}\right)(x)\in B\left(x_{N},\left(\frac{1}{\lambda(\A)^{sN}}\right)^{1/d}\right) \textrm{ for i.m. }(i_1,\ldots,i_N)\in \bigcup_{n=1}^{\infty}\I^n\right\}.$$ We denote this set by $W_{s}(S,\x)$. When $\x$ is a constant sequence, i.e. $\x=(y)_{n=1}^{\infty},$ then we use $W_{s}(S,y)$ to denote $W_{s}(S,\x).$
	As a corollary of the proof techniques of Theorem \ref{exponential theorem}, and a Mass Transference Principle of Wang and Wu \cite{WanWu}, we will prove the following result on the almost sure dimension of $W_{s}(S,\x).$. 
	
	\begin{thm}\label{thm:dim}
		Let $\A=\{A_i\}_{i\in \I}$ be a finite set of invertible $d\times d$ matrices satisfying the following properties:
		\begin{itemize}
			\item There exists a positive diagonal matrix $A$
			such that $A_i=A$ for all $i\in \I$, with diagonal entries $\lambda_1, \ldots, \lambda_d$. 	
			\item $\|A\|<1/2$.
			\item $\lambda(\A)>1$.
			\item There exists an open set $U\subset \mathbb{R}^{\# \I\cdot d}$ such that $X_{\tt}$ has non-empty interior for Lebesgue almost every $\tt\in U$. 
			%\item $\#\I\det A^2>1$.\textbf{Simon - Maybe assume non-empty interior for Lebesgue almost every $\tt$ instead. Then mention this particular case. Note that then you have to include $\lambda(\A)>1$ somewhere too.}
		\end{itemize}
		Then for any $\j\in \I^{\N},$ for Lebesgue almost every $\tt\in U,$ for any ball $B$ centred at $X_{\tt}$ we have $$\dim_{H}\left(W_{s}(S_{\tt},\pi_{\tt}(\j))\cap B\right)\geq \min_{p\in P}\left\{\# K_{1} + \# K_{2}\left(1-\frac{s-1}{dp}\right)+\frac{\# K_{3}}{dp}+\sum_{i\in K_{3}}\frac{\log \lambda_i} {p\log\lambda(\A)^{-1}}\right\}$$
		where $$a_i=\frac{\log \lambda_i} {\log\lambda(\A)^{-1}}+\frac{1}{d},$$	
		$P=\{a_i,a_i+\frac{s-1}{d}:1\leq i\leq d\},$ and $$K_{1}=\{i:a_i\geq p\},\, K_{2}=\left\{i:a_i+\frac{s-1}{d}\leq p\right\},\, K_{3}=\{1,\ldots,d\}\setminus (K_{1}\cup K_{2}).$$ 
		%		Let $s>1$ be such that $s\neq d\frac{\log(\lambda_i/\lambda_k)}{\log \lambda(\A)^{-1}} +1$ (\textbf{Simon - For any $i,k$ or $i\neq k$?}) and 
		%\begin{align*}		
		%		W&(S_{\tt},\pi_{\tt}(\j),s)\\
		%&:=\left\{x\in X: \left(T_{i_N}\circ \cdots \circ T_{i_1}\right)(x)\in B\left(\pi_{\t}(\j),\left(\frac{1}{\lambda(\A)^{sN}}\right)^{1/d}\right) \textrm{ for i.m. }(i_1,\ldots,i_N)\in \bigcup_{n=1}^{\infty}\I^n\right\}. 
		%\end{align*}
		%To state the dimension lower bound we need the following notation: Let  $b_k=\log\lambda_k/\log\lambda(\A)^{-1}+\tfrac 1d$ for each $k=1, \dots, d$, and then let $\mathcal P=\{b_1, \dots, b_d, b_1+\tfrac {s-1}d, \dots, b_d+\tfrac {s-1}d\}$. Denote, for each $p\in \mathcal P$, 
		%\[
		%K_1=\{k\mid b_k\ge p\}, \quad K_2=\{k\mid b_k+\tfrac {s-1}d \le p\}\setminus K_1, \quad K_3=\{1, \dots, d\}\setminus (K_1\cup K_2). 
		%\]
		%Then 
		%\[
		%\dim W(S_\tt, \pi_\tt(\j), s)\ge \min_{p\in \mathcal P}\left\{\#K_1+ \#K_2+\frac{(\log \lambda(\A)^{-1})^{-1}\sum_{k\in K_3}\log \lambda_k}{p}+\frac{\#K_3+\#K_2(1-s)}{dp}\right\}. 
		%\]
		%\textbf{Simon - Maybe we take some of the notation out of theorem to shorten it.}
	\end{thm}	
	
	It follows from a result of Feng and Feng \cite{FenFen} that if we replace our assumption $\lambda(\A)>1$ with the stronger assumption $\#\I \cdot |Det(A)|^2>1,$ then our fourth assumption is immediately satisfied and we can in fact take $U=\mathbb{R}^{\# \I\cdot d}$.

	In Theorems \ref{Main thm} and \ref{exponential theorem} our statements for recurrence sets only establish positive Lebesgue measure. It is natural to expect that these recurrence sets in fact have full Lebesgue measure. Motivated by this shortcoming, we introduce a family of IFSs for which we can establish this stronger full measure statement. 
	
	Given $\lambda\in(1/2,1)$ we associate the IFS $S_{\lambda}=\{S_{0}(x)=\lambda x,S_{1}(x)=\lambda x +1\}.$ For each $S_{\lambda}$ the corresponding invariant set equals $[0,\frac{1}{1-\lambda}].$ For each $\lambda\in(1/2,1)$ we let $\mu_{\lambda}$ denote the law of the random variable $$\sum_{i=0}^{\infty}\epsilon_i \lambda^{i}$$ where each $\epsilon_i$ takes values $0$ and $1$ with equal probability. The probability measure $\mu_{\lambda}$ is known as the Bernoulli convolution corresponding to $\lambda$.  Determining the dimension of $\mu_{\lambda}$, and determining those $\lambda$ for which the corresponding Bernoulli convolution is absolutely continuous are two important problems that have attracted much attention. We refer the reader to \cite{Hochman2,Hochman,Shmerkin,Sol,Var2,Var3} for a more detailed survey of Bernoulli convolutions and for an overview of some recent results. 
	
	We call a number $\beta\in (1,2)$ a Garsia number if it is an algebraic integer with norm $\pm 2$, whose Galois conjugates are all of modulus strictly greater than $1$. This family of algebraic integers was first introduced in \cite{Gar}, where it was shown that if $\lambda$ is the reciprocal of a Garsia number then $\mu_{\lambda}$ is absolutely continuous with bounded density. Examples of Garsia numbers include $2^{1/n}$ for $n\geq 1$, and $1.08162\ldots$ the appropriate root of $x^6+x^5-x-2=0$. For more on Garsia numbers we refer the reader to the survey \cite{HarPan} by Hare and Panju. Our main result for this family of IFSs is the following.
	
	\begin{thm}
		\label{Garsia thm}
		Let $\lambda\in(1/2,1)$ be such that $\lambda^{-1}$ is a Garsia number. Then for any $h:\mathbb{N}\to [0,\infty)$ satisfying $\sum_{n=1}^{\infty}h(n)=\infty$, we have that Lebesgue almost every $x\in [0,\frac{1}{1-\lambda}]$ is contained in $R(S_{\lambda},h).$
	\end{thm}
	The corresponding shrinking target analogue of Theorem \ref{Garsia thm} was obtained in \cite{Bakapprox}. We emphasise that our method of proof for Theorem \ref{Garsia thm} is different to that given in \cite{Bakapprox}.
	
	Using the mass transference principle of Beresnevich and Velani \cite{BerVel}, we will use Theorem \ref{Garsia thm} to prove the following result which applies when $R(S_{\lambda},h)$ has zero Lebesgue measure.
	
	\begin{thm}
		\label{GarsiaMTP}
		Let $\lambda\in(1/2,1)$ be such that $\lambda^{-1}$ is a Garsia number and $h:\mathbb{N}\to [0,\infty)$ be bounded. Then for any ball $B$ contained in $[0,\frac{1}{1-\lambda}]$ we have 
		$$\sum_{n=1}^{\infty}2^{n(1-s)}h(n)^s<\infty\implies \mathcal{H}^{s}(B\cap R(S_{\lambda},h))=0$$ and $$\sum_{n=1}^{\infty}2^{n(1-s)}h(n)^s=\infty\implies \mathcal{H}^{s}(B\cap R(S_{\lambda},h))=\mathcal{H}^{s}(B).$$ 
	\end{thm}
	
	%\textbf{Simon - Mass transference principle corollary of Theorem \ref{Garsia thm}}
	\begin{remark}
		It is interesting to compare Theorems \ref{Main thm}, \ref{exponential theorem}, and \ref{Garsia thm} with existing results on the shrinking target problem and quantitative recurrence. The most significant difference between our work and the existing body of work is that we establish positive measure results when our target sets/recurrence neighbourhoods shrink to zero exponentially fast. This is perhaps not surprising as we consider multiple forward orbits. Nevertheless, it is a significant change to the case of a single orbit where for a shrinking target set or a recurrence set to have positive Lebesgue measure, the target sets/recurrence neighbourhoods must typically shrink to zero at a polynomial rate.
	\end{remark}
\begin{remark}
	In this paper we consider the Lebesgue measure of certain limsup sets defined using a self-affine iterated function system. There are many other natural measures one can consider supported on the self-affine set, e.g. self-affine measures. It is natural to ask what is the measure of these limsup sets with respect to these other measures. The techniques of this paper do not immediately apply in this context.
\end{remark}
	\begin{remark}
		We finish this introductory section by drawing a comparison between this paper and \cite{Bak}. In \cite{Bak} the first author studied the following family of limsup sets: Given $\S$ an IFS, $z\in X$, and $(B_n)$ a sequence of balls, we let
		$$Q(\S,z,(B_n)):=\left\{x\in X:x\in (S_{i_1}\circ \cdots \circ S_{i_N})(z)+B_N \textrm{ for i.m. }(i_1,\ldots,i_N)\in \bigcup_{n=1}^{\infty}\I^n\right\}.$$ In \cite{Bak} the first author studied the sets $Q(\S,z,(B_n))$ for the parameterised family of affine iterated function systems appearing in Theorems \ref{Main thm} and \ref{exponential theorem}. In the special case where each matrix defining the family is a similarity, i.e. satisfies $\|Ax-Ay\|=r\|x-y\|$ for all $x,y\in\mathbb{R}^d$ for some $r\in (0,1)$, then the methods used in \cite{Bak} can be used to prove weak versions of Theorems \ref{Main thm} and \ref{exponential theorem}. On a technical level, what makes the analysis of this paper particularly challenging when compared to \cite{Bak}, is that instead of working with limsup sets defined by a sequence of balls, the limsup sets we eventually work with are defined using ellipses or more exotic shapes. Because of this potentially more complicated geometry, the arguments from \cite{Bak} do not apply and new ideas are required. %\textbf{Simon - Be more definitive here. }
	\end{remark}

	The rest of the paper is structured as follows. In Section \ref{Preliminaries} we establish some notation and collect some technical results that we will use in the proofs of our theorems. In Section \ref{Basic results} we prove a number of straightforward theorems that demonstrate how a recurrence set or a shrinking target set can have zero Lebesgue measure. These theorems highlight some of the technical obstacles that need to be overcome to prove our results. In Section \ref{Section 5} we prove Theorem \ref{Main thm} and in Section \ref{Section 6} we prove Theorem \ref{exponential theorem}. In Section \ref{Section 7} we prove Theorem \ref{thm:dim} and in Section \ref{Garsia section} we prove Theorems \ref{Garsia thm} and \ref{GarsiaMTP}.
	
	\section{Preliminaries}
	\label{Preliminaries}
	In this section we introduce some notation and collect some technical results that we will use in the proofs of our main theorems. 
	
	\subsection{Notation}
	Suppose that an IFS $\{S_i\}_{i\in \I}$ is given. We let $\I^{*}=\bigcup_{n=1}^{\infty}\I^n$. For a word $\i=(i_1,\ldots,i_n)\in \I^*$ we let $S_{\i}=S_{i_1}\circ \cdots \circ S_{i_n}$ and $T_{\i}=T_{i_1}\circ \cdots \circ T_{i_n}$. Similarly, given a finite set of matrices $\{A_i\}_{i\in \I}$ and $\i=(i_1,\ldots,i_n)\in \I^*,$ we let $A_{\i}=A_{i_1}\circ \cdots \circ A_{i_n}$. We denote the length of a word $\i\in \I^{*}$ by $|\i|$. Given two distinct words $\i,\j\in \I^{n}$ we let $|\i\wedge \j|=\min\{k:i_{k}\neq j_k\}$, and let $\i\j$ denote their concatenation. We also let $\i^{\infty}$ denote the element of $\I^{\N}$ corresponding to the infinite concatenation of $\i$ with itself. For $\i=(i_1,\ldots,i_n)\in \I^{*}$ we let $\overline{\i}=(i_n,\ldots,i_1)$. We emphasise that $S_{\i}^{-1}=T_{\overline{\i}}$.
	
	Given two real valued functions $f,g:X\to \mathbb{R}$ defined on some set $X$, we write $f\ll g$ if there exists $C>0$ such that $|f(x)|\leq C\cdot g(x)$ for all $x\in X$. We will also on occasion write $f=\O(g)$ which will have the same meaning as $f\ll g$. We write $f\asymp g$ if $f\ll g$ and $g\ll f$. When we want to emphasise a dependence of the underlying constant on some other parameter, for instance $R$, we write $f\ll_{R} g$ or $f=\O_{R}(g).$

	We let $\L_{d}$ denote the Lebesgue measure on $d$-dimensional Euclidean space.
	\subsection{Technical results}
	The following standard lemma is due to Kochen and Stone \cite{KocSto}. For a proof of this lemma see either \cite[Lemma 2.3]{Har} or \cite[Lemma 5]{Spr}. 
	\begin{lemma}
		\label{Quasi-independence on average}
		Let $(X,\mathcal{B},\mu)$ be a finite measure space and $E_n\in \mathcal{B}$ be a sequence of sets such that $\sum_{n=1}^{\infty}\mu(E_n)=\infty$. Then $$\mu\left(\limsup_{n\to\infty} E_n\right)\geq \limsup_{Q\to\infty}\frac{\left(\sum_{n=1}^{Q}\mu(E_n)\right)^2}{\sum_{n,m=1}^{Q}\mu(E_n\cap E_m)}.$$
	\end{lemma}
	Given a Borel set $A\subset \mathbb{R}^d$, we say that $x\in \mathbb{R}^d$ is a density point for $E$ if $$\lim_{r\to 0}\frac{\L_{d}(B(x,r)\cap E)}{\L_{d}(B(x,r))}=1.$$ The following result, known as the Lebesgue density theorem, will play an important role in our analysis. For a proof see \cite{Mat}.
	
	\begin{thm}
		\label{Lebesgue density theorem}
		Let $E\subset\mathbb{R}^d$ be a Borel set. Then for Lebesgue almost every $x\in\mathbb{R}^d$ we have 
		$$\lim_{r\to 0}\frac{\L_{d}(B(x,r)\cap E)}{\L_{d}(B(x,r))}=\chi_{E}(x).$$ In particular, Lebesgue almost every $x\in E$ is a density point for $E$.
	\end{thm}
	The following lemma is one of Bonferroni's inequalities. We include a proof for the sake of completeness.
	
	\begin{lemma}
		\label{Bonferroni}
		Let $(X,\mathcal{B},\mu)$ be a measure space and $E_1,\ldots, E_n\in \mathcal{A}$. Then $$\mu\left(\bigcup_{k=1}^{n}E_k\right)\geq \sum_{k=1}^{n}\mu(E_k)-\sum_{1\leq k<k'\leq n}\mu(E_k\cap E_{k'}).$$
	\end{lemma}
	\begin{proof}
		We will show that for any $x\in X$ we have 
		\begin{equation}
			\label{Bonferroni function}
			\chi_{\cup_{k=1}^{n}E_k}(x)\geq \sum_{k=1}^{n}\chi_{E_k}(x)-\sum_{1\leq k<k'\leq n}\chi_{E_k\cap E_{k'}}(x).
		\end{equation} Using this inequality and integrating with respect to $\mu$ yields our result. If $x\notin \cup_{k=1}^{n}E_k$ then \eqref{Bonferroni function} holds trivially. Now suppose that $x\in \cup_{k=1}^{n}E_k$. Let $j\geq 1$ be such that $\sum_{k=1}^{n}\chi_{E_k}(x)=j.$ After relabelling our sets we may assume that $x\in E_1, x\in E_2,\ldots, x\in E_j$, and $x\notin \cup_{k=j+1}^{n}E_k$. Then we have $$\sum_{1\leq k<k'\leq n}\chi_{E_k\cap E_{k'}}(x)\geq \sum_{k=2}^{j}\chi_{E_1\cap E_k}(x)=j-1.$$ Using this inequality, we observe $$\chi_{\cup_{k=1}^{n}E_k}(x)=1=j-(j-1)\geq \sum_{k=1}^{n}\chi_{E_k}(x)-\sum_{1\leq k<k'\leq n}\chi_{E_k\cap E_{k'}}(x).$$ Therefore \eqref{Bonferroni function} holds and our proof is complete. 
	\end{proof}

	The following technical framework is adapted from \cite{Bak}. Let $\Omega$ denote a metric space equipped with a Borel measure $\eta$, and let $X$ denote some compact subset of $\mathbb{R}^d$. For each $n\in\mathbb{N},$ we assume that we are given a Borel set $E_n$ and a finite set of continuous functions $\{f_{l,n}:\Omega\to X\}_{l=1}^{R_n}$. We are interested in the distribution of the elements of  $\{f_{l,n}(\omega)+s\cdot E_{n}\}_{l=1}^{R_n}$ for a $\eta$-typical $\omega$ and for small $s>0$. We say that a set $Y\subset \mathbb{R}^d$ is {\it $(s,E_n)$-separated} if for all $ x,x'\in Y$ such that $x\neq x'$ we have $$(x+sE_n)\cap (x'+sE_n)=\emptyset. $$ Given $\omega\in \Omega,$ $s>0$, and $n\in\mathbb{N}$ we let $$T(\omega,s,n):=\max\{\# Y:Y\subset\{1,\ldots, R_n\} \textrm{ and }\{f_{l,n}(\omega)\}_{l\in Y} \textrm{ is }(s,E_n)\textrm{ separated}\}.$$ %Now suppose we were given $c,s>0$ and $n\in \N$, we let
	%$$B(c,s,n):=\left\{\omega:T(\omega,s,n)>c\cdot R_n\right\}.$$
	We also let $$R(\omega,s,n):=\{(l,l')\in \{1,\ldots,R_n\}^2: l\neq l',\, (f_{l,n}(\omega)+s\cdot E_n)\cap (f_{l',n}(\omega)+s\cdot E_n)\neq \emptyset\}.$$
	The following lemma follows from Lemmas 3.2 and 3.5 from \cite{Bak}. In these lemmas the sets $E_n$ were always taken to be balls, but the proofs still apply with only minor notational changes in the more general setting where the sets $E_n$ are only assumed to be Borel sets.
	\begin{lemma}
		\label{Density lemma}
		Assume that there exists $G:(0,\infty)\to (0,\infty)$ satisfying $\lim_{s\to 0}G(s)=0$, such that for all $n\in\mathbb{N}$ we have $$\int_{\Omega}\frac{\#R(\omega,s,n)}{R_n}\, d\eta\leq G(s).$$ Then we have the following information about the upper density		
		$$\eta\left(\bigcap_{\epsilon>0} \bigcup_{c,s>0}\left\{\omega:\overline{d}(n:T(\omega,s,n)>c\cdot R_n)\geq 1-\epsilon\right\}\right)=\eta(\Omega).$$
	\end{lemma}
	Finally we need the following estimate, a {\it transversality} condition in the self-affine context.
	
	\begin{lemma}
		\label{Transverality estimate}
		Let $\{A_i\}_{i\in \I}$ be a finite set of invertible $d\times d$ matrices satisfying $\|A_i\|<1/2$ for all $i\in \I$. Let $\i,\j\in \I^{\mathbb{N}}$ be such that $\i\neq \j$ and let $R>0$ be arbitrary. Then for any Borel set $E\subset\mathbb{R}^d$ we have 
		\[
		\L_{\#\I\cdot d}\left(\tt\in [-R,R]^{\#\I\cdot d}:\pi_{\tt}(\i)-\pi_{\tt}(\j)\in E\right)\ll_{R} |Det(A_{i_1,\ldots i_{|\i\wedge \j|-1}})^{-1}|\cdot\L_{d}(E). 
		\]	
		
	\end{lemma}Here and throughout, if $|\i\wedge \j|=1$ then $A_{i_1,\ldots i_{|\i\wedge \j|-1}}$ is simply the identity matrix.
	\begin{proof}
		This statement essentially follows from an argument due to Solomyak \cite{Sol2}, which is in turn based upon an argument due to Falconer \cite{Fal2}. We include the details for the sake of completeness.
		
		Let $\i \neq \j\in \I^\N$. We start by observing that  $$\pi_{\tt}(\i)-\pi_{\tt}(\j)=t_{i_1}+\sum_{n=1}^{\infty}A_{i_1\ldots i_n}t_{i_{n+1}}-t_{j_1}-\sum_{n=1}^{\infty}A_{j_1\ldots j_n}t_{j_{n+1}}.$$ Which by the definition of $|\i\wedge \j|$ can be rewritten as 
		\begin{align*}
			\pi_{\tt}(\i)-\pi_{\tt}(\j)&=\sum_{n=|\i\wedge \j|-1}^{\infty}A_{i_1\ldots i_n}t_{i_{n+1}}-\sum_{n=|\i\wedge \j|-1}^{\infty}A_{j_1\ldots j_n}t_{j_{n+1}}\\
			&=A_{i_{1}\ldots i_{|\i\wedge \j|-1}}\left(t_{i_{|\i\wedge \j|}}-t_{j_{|\i\wedge \j|}}+\sum_{n=|\i\wedge \j|}^{\infty}A_{i_{|\i\wedge \j|}\ldots i_n}t_{i_{n+1}}-\sum_{n=|\i\wedge \j|}^{\infty}A_{j_{|\i\wedge \j|}\ldots j_n}t_{j_{n+1}}\right).
		\end{align*}
		If we group terms in these summations according to common values of $i\in \I$ we can express the above as $$A_{i_{1}\ldots i_{|\i\wedge \j|-1}}\left(t_{i_{|\i\wedge \j|}}-t_{j_{|\i\wedge \j|}}+\sum_{i\in \I}\left(\sum_{n\geq |\i\wedge \j|:i_{n+1}=i}A_{i_{|\i\wedge \j|}\ldots i_n}t_{i}-\sum_{n\geq |\i\wedge \j|:j_{n+1}=i}A_{j_{|\i\wedge \j|}\ldots j_n}t_{i}\right)\right).$$ Now for each $i\in \I$ we let $E_i:\mathbb{R}^d \to \mathbb{R}^d$ be the linear map defined by 
		\[E_{i}(t_i)=\sum_{n\geq |\i\wedge \j|:i_{n+1}=i}A_{i_{|\i\wedge \j|}\ldots i_n}t_{i}-\sum_{n\geq |\i\wedge \j|:j_{n+1}=i}A_{j_{|\i\wedge \j|}\ldots j_n}t_{i}.\]
		Solomyak in \cite{Sol2} proved that under the assumptions of this lemma, there exists $C>0$ independent of $\i$ and $\j$ such that either $\|(I+E_{i_{|\i\wedge \j|}})^{-1}\|<C$ or $\|(-I+E_{j_{|\i\wedge \j|}})^{-1}\|<C$. Let us suppose $\|(I+E_{i_{|\i\wedge \j|}})^{-1}\|<C$. The other case is handled similarly.
		
		It follows from the  above that $\pi_{\tt}(\i)-\pi_{\tt}(\j)\in E$ is equivalent to $$t_{i_{|\i\wedge \j|}}+E_{i_{|\i\wedge \j|}}(t_{i_{|\i\wedge \j|}})\in t_{j_{|\i\wedge \j|}}-\sum_{i\in \I\setminus\{i_{|\i\wedge \j|}\}}E_{i}(t_i)+A_{i_{1}\ldots i_{|\i\wedge \j|-1}}^{-1}(E),$$ which in turn is equivalent to 
		\begin{equation}
			\label{t_{i_1} inclusion}
			t_{i_{|\i\wedge \j|}}\in (I+E_{i_{|\i\wedge \j|}})^{-1}\left(t_{j_{|\i\wedge \j|}}-\sum_{i\in \I\setminus\{i_{|\i\wedge \j|}\}}E_{i}(t_i)+A_{i_{1}\ldots i_{|\i\wedge \j|-1}}^{-1}(E)\right).
		\end{equation} Let us now fix a set of vectors $\{t_i\}_{i\in \I\setminus\{i_{|\i\wedge \j|}\}}$ and consider $t_{i_{|\i\wedge\j|}}$ as a variable. Using the fact that $\|(I+E_{i_{|\i\wedge \j|}})^{-1}\|<C,$ and $$\L_{d}\left(t_{j_{|\i\wedge \j|}}-\sum_{i\in \I\setminus\{i_{|\i\wedge \j|}\}}E_{i}(t_i)+A_{i_{1}\ldots i_{|\i\wedge \j|-1}}^{-1}(E)\right)=|Det(A_{i_1,\ldots i_{|\i\wedge \j|-1}})^{-1}|\cdot \L_{d}(E),$$ we see that \eqref{t_{i_1} inclusion} implies that 
		\[
		\mathcal{L}_{d}\left(t_{i_{|\i\wedge \j|}}\in\mathbb{R}^d:\eqref{t_{i_1} inclusion}\textrm{ is satisfied for the fixed} \{t_i\}_{i\in \I\setminus\{i_{|\i\wedge \j|}\}}\right)\ll  |Det(A_{i_1,\ldots i_{|\i\wedge \j|-1}})^{-1}|\cdot \L_{d}(E).
		\]
		Now applying Fubini's theorem, it follows that 
		$$\L_{\# \I\cdot d}\left(\tt\in [-R,R]^{\# \I\cdot d}: \eqref{t_{i_1} inclusion} \textrm{ is satisfied}\right)\ll_{R}  |Det(A_{i_1,\ldots i_{|\i\wedge \j|-1}})^{-1}|\cdot \L_{d}(E).$$

	\end{proof}

	\section{Basic results}
	\label{Basic results}
	Before moving on to the proofs of our main theorems, we explore the ways in which the conclusions of these theorems can fail. The proofs of the following statements serve as a warm up for what is to follow.
	
	\begin{thm}
		\label{Exact overlap}
		Let $\S$ be an iterated function system. Suppose that there exists $\j',\j\in \I^{*}$ such that $\j'\neq \j$ and $S_{\j'}=S_{\j}.$ Then for any $\x\in X^{\N}$ and sequence of Borel sets satisfying $\L_{d}(E_n)\ll \lambda(\A)^{-n},$ the set $W(\S,\x,(E_n))$ has zero Lebesgue measure. Similarly, the set $R(\S,(E_n))$ has zero Lebesgue measure. 
	\end{thm}
	\begin{proof}
		Let $\j', \j\in \I^{*}$ be such that $\j'\neq \j$ and $S_{\j'}=S_{\j}$. By considering $S_{\j'\j}$ and $S_{\j\j'}$ if necessary, we may assume that $\j'$ and $\j$ have the same length. Let us suppose $|\j'|=|\j|=k$ for some $k\in\mathbb{N}$. Then we have by the definition \eqref{eq:lambda} of $\lambda (\A)$ that
		\[
		\gamma:=\sum_{\i\in \I^{k}\setminus\{\j\}}\frac{|Det(A_\i)|}{\lambda(\A)^k}<\sum_{\i\in \I^{k}}\frac{|Det(A_\i)|}{\lambda(\A)^k}=1.\]
		Now let $\x=(x_n)\in X^{\N}$ be arbitrary, and let $(E_n)$ be an arbitrary sequence of Borel sets satisfying $\L_{d}(E_n)\ll \lambda(\A)^{-n}$. We start by proving that $W(S,\x,(E_n))$ has zero Lebesgue measure. Observe that for any $n>k$ we have 
		\begin{align*}
			&\left\{x:(T_{j_n}\circ \cdots \circ T_{j_1})(x)\in x_n+E_n\textrm{ for some }(j_1,\ldots,j_n)\in \I^n\right\}\\
			=&\bigcup_{(j_1,\ldots, j_n)\in \I^{n}}S_{j_1\ldots j_n}\left(x_n+E_n\right)\\
			=&\bigcup_{\i \in (\I^{k}\setminus\{\j\})^{\lfloor n/k\rfloor}}\bigcup_{\i' \in \I^{n-k\cdot \lfloor n/k\rfloor}}S_{\i \i'}\left(x_n+E_n\right).
		\end{align*}
		Now using that $\L_{d}(E_n)\ll \lambda(\A)^{-n}$, we have the following for $n>k$
		\begin{align*}
			&\L_{d}\left(\left\{x:(T_{j_n}\circ \cdots \circ T_{j_1})(x)\in x_n+E_n\textrm{ for some }(j_1,\ldots,j_n)\in \I^n\right\}\right)\\
			\ll & \sum_{\i \in (\I^{k}\setminus\{\j\})^{\lfloor n/k\rfloor}}\sum_{\i' \in \I^{n-k\cdot \lfloor n/k\rfloor}} |Det(A_{\i \i'})|\lambda(\A)^{-n}\\
			= & \sum_{\i \in (\I^{k}\setminus\{\j\})^{\lfloor n/k\rfloor}}|Det(A_\i)|\lambda(\A)^{-k\cdot \lfloor n/k\rfloor}\sum_{\i' \in \I^{n-k\cdot \lfloor n/k\rfloor}}| Det(A_{ \i'}|)\lambda(\A)^{-n+k\cdot \lfloor n/k\rfloor}
		\end{align*}
		The second sum in this product is equal to $1$ by the definition \eqref{eq:lambda} of $\lambda(\A)$. Continuing from here, using the definition of $\gamma$ we have
		\begin{align*}
			\sum_{\i \in (\I^{k}\setminus\{\j\})^{\lfloor n/k\rfloor}}|Det(A_\i)|\lambda(\A)^{-k\cdot \lfloor n/k\rfloor}
			=&\left( \sum_{\i \in (\I^{k}\setminus\{\j\})}|Det(A_\i)|\lambda(\A)^{-k}\right)^{\lfloor n/k\rfloor}\\
			=&\gamma^{\lfloor n/k\rfloor}.
		\end{align*}
		Hence $$\L_{d}\left(\left\{x:(T_{j_n}\circ \cdots \circ T_{j_1})(x)\in x_n+E_n\textrm{ for some }(j_1,\ldots,j_n)\in \I^n\right\}\right)\ll \gamma^{\lfloor n/k\rfloor}$$
		We know that $\gamma\in(0,1)$, hence it follows from the estimate above that
		\begin{align*}
			&\sum_{n=1}^{\infty}\L_{d}\left(\left\{x:(T_{j_n}\circ \cdots \circ T_{j_1})(x)\in x_n+E_n\textrm{ for some }(j_1,\ldots,j_n)\in \I^n\right\}\right)\\
			\ll&\sum_{n=1}^{\infty}\gamma^{\lfloor n/k\rfloor}<\infty.
		\end{align*}
		Applying the Borel Cantelli lemma our result follows.
		
		Now we bring our attention to proving the recurrence result. Again we assume that $(E_n)$ is a sequence of Borel sets satisfying $\L_{d}(E_n)\ll \lambda(\A)^{-n}$. Recall that each $S_i$ is of the form $S_{i}(x)=A_{i}x+t_i$. Therefore $T_{i}(x)=S_{i}^{-1}(x)=A_{i}^{-1}(x-t_i).$ Similarly we have $$T_{\i}(x)=A_{\i}^{-1}x-\sum_{k=1}^{|\i|}A_{i_1\ldots i_{k}}^{-1}t_{i_{k}}$$for any $\i\in \I^*.$ Using this identity, we observe the following for any $\i\in \I^n$:
		\begin{align*}
			T_{\i}(x)-x\in E_n
			\iff &A_{\i}^{-1}(x)-x\in \sum_{k=1}^{n}A_{i_1\ldots i_{k}}^{-1}t_{i_{k}}+E_n \\
			\iff &x\in (A_{\i}^{-1}-I)^{-1}\left(\sum_{k=1}^{n}A_{i_1\ldots i_{k}}^{-1}t_{i_{k}}+E_n\right).
		\end{align*}
		Recall that $(A_{\i}^{-1}-I)^{-1}=\sum_{l=1}^{\infty}A_{\i}^{l}.$ It is useful to think of the operator $\sum_{l=1}^{\infty}A_{\i}^{l}$ as $A_{\i}$ composed with $\sum_{l=1}^{\infty}A_{\i}^{l-1}.$ Because $\|A_i\|<1$ for all $i\in \I$, the operator $\sum_{l=1}^{\infty}A_{\i}^{l-1}$ satisfies $$\L_{d}\left(\sum_{l=1}^{\infty}A_{\i}^{l-1}E\right)\ll \L_{d}(E)$$ for any Borel set $E$. Now applying $A_{\i}$ to $\sum_{l=1}^{\infty}A_{\i}^{l-1}E$ we see that 
		\begin{equation}
			\label{Determinant is bounding}
			\L_{d}\left(\sum_{l=1}^{\infty}A_{\i}^{l}E\right)\ll |Det(A_{\i})|\L_{d}(E)
		\end{equation} for any Borel set $E$. Using \eqref{Determinant is bounding}, together with our assumption $\L_{d}(E_n)\ll \lambda(\A)^{-n}$, it follows that for any $\i \in \I^n$ we have
		\begin{align}
			\label{Recurrence measure}
			\L_{d}\left(\left\{x:T_{\i}(x)-x\in E_n\right\}\right)
			=&\L_{d}\left(\left\{x\in (A_{\i}^{-1}-I)^{-1}\left(\sum_{k=1}^{n}A_{i_1\ldots i_{k}}^{-1}t_{i_{k}}+E_n\right)\right\}\right)\nonumber \\
			=&\L_{d}\left(\left\{x\in \sum_{l=1}^{\infty}A_{\i}^{l}\left(\sum_{k=1}^{n}A_{i_1\ldots i_{k}}^{-1}t_{i_{k}}+E_n\right)\right\}\right)\nonumber \\
			\ll & |Det(A_{\i})|\L_{d}\left(\sum_{k=1}^{n}A_{i_1\ldots i_{k}}^{-1}t_{i_{k}}+E_n\right)\nonumber\\
			\ll & |Det(A_{\i})|\lambda(\A)^{-|n| }.
		\end{align}
		We now observe that for any $n>k$ we have 
		\begin{align*}
			&\left\{x:(T_{j_n}\circ \cdots T_{j_1})(x)-x\in E_n\textrm{ for some }(j_1,\ldots,j_n)\in \I^n\right\}\\
			=&\bigcup_{\i \in (\I^{k}\setminus\{\j\})^{\lfloor n/k\rfloor}}\bigcup_{\i' \in \I^{n-k\cdot \lfloor n/k\rfloor}}\left\{x:T_{\i' \i}(x)-x\in E_n \right\}.
		\end{align*}
		Using this equality and \eqref{Recurrence measure}, the rest of the proof follows from an analogous argument to that used to prove that $W(\S,\x,(E_n))$ has zero Lebesgue measure. 
	\end{proof}
	Theorem \ref{Exact overlap} is straightforward but it exhibits one of the main difficulties in proving our theorems. For many choices of $\{A_i\}_{i\in \I},$ there is a dense set of $\tt\in\mathbb{R}^{\#\I \cdot d}$ such that the corresponding IFS $S_{\tt}$ admits two distinct words $\i,\j\in \I^{*}$ satisfying $S_{\i,\tt}=S_{\j,\tt}$. This means that there is a dense set of exceptions for which the conclusions of Theorems \ref{Main thm} and \ref{exponential theorem} do not hold. This set of exceptions is what makes our analysis challenging. The following statement shows that the presence of two distinct words satisfying $S_{\i}=S_{\j}$ is not the only mechanism preventing positive measure from occurring; it shows that if the cylinder sets are aligned in an extreme manner, then this can lead to our limsup sets having zero Lebesgue measure.
	
	\begin{thm}
		For each $1\leq l\leq d$, let $\S_{l}=\{S_{i}\}_{i\in \I_{l}}$ be an IFS acting on $\mathbb{R}$ satisfying the following properties:
		\begin{itemize}
			\item There exists $\lambda_{l}\in (0,1)$ such that for each $i\in \I_{l}$ we have $S_{i}(x)=\lambda_{l}x+t_{i}$ for some $t_{i}\in \mathbb{R}$.
			\item $[0,1]=\cup_{i\in \I_{l}}S_{i}([0,1])$. 
		\end{itemize}
		Let $\mathfrak{S}$ be the product IFS acting on $\mathbb{R}^d$ given by $$\mathfrak{S}=\{S_{(i_1,\ldots,i_d)}(\textbf{x})=(\lambda_1x_1 +t_{i_1},\ldots, \lambda_{d}x_d+t_{i_d})\}_{(i_1,\ldots,i_d)\in \I_1\times \cdots \times \I_d}.$$ Assume that there exists $1\leq l_1<l_2<d$ such that $\#\I_{l_1}\cdot \lambda_{l_1}<\#\I_{l_2}\cdot \lambda_{l_2}.$ Then for any sequence $\x\in ([0,1]^d)^{\N},$ we have $\L_{d}(W(\mathfrak{S},\x,1))=0.$ Similarly $\L_{d}(R(\mathfrak{S},1)))=0.$ 
	\end{thm}
	
	\begin{proof}
		We begin our proof by remarking that for the product IFS $\mathfrak{S}$ the self-affine set is $[0,1]^d$. We also remark that for this IFS we have $$\lambda(\A)=\prod_{l=1}^{d}\# \I_{l} \cdot \lambda_{l}.$$ Without loss of generality, we can assume that $\#\I_{1}\cdot \lambda_{1}=\min_{1\leq l\leq d}\{\#\I_{l}\cdot \lambda_{l}\}$ and that $\#\I_{1}\cdot \lambda_{1}<\#\I_{2}\cdot \lambda_{2}.$ These statements imply that
		$$\# \I_{1}\cdot \lambda_{1}<\left(\prod_{l=1}^d\# \I_{l}\cdot \lambda_{l}\right)^{1/d},$$ and therefore
		\begin{equation}
			\label{narrow rectangles}
			\frac{\# \I_{1}\cdot \lambda_{1}}{\lambda(\A)^{1/d}}<1.
		\end{equation} Let us now fix a sequence $\x=(x_n)\in ([0,1]^d)^{\N}.$ For any $\i \in (\I_1\times \cdots \times \I_{d})^n$ the first coordinate of $S_{\i}(x_n)$ is uniquely determined by $x_n$ and the $\I_{1}$ components in $\i$. This implies that 
		\begin{align*}
			&\left(\bigcup_{\i\in (\I_1\times \cdots \times \I_{d})^n} S_{\i}\left(B\left(x_n,\frac{1}{\lambda(\A)^{n/d}}\right)\right)\right)\cap [0,1]^d\\
			\subseteq & \left(\bigcup_{\i\in (\I_1\times\cdots\times\I_d)^n}S_{\i}(x_n)+\prod_{1\leq l\leq d}\left[\frac{-\lambda_{l}^n}{\lambda(\A)^{n/d}},\frac{\lambda_{l}^n}{\lambda(\A)^{n/d}}\right]\right)\cap [0,1]^d\\
			\subseteq&\left(\bigcup_{\j\in \I_1^n}S_{\j}(\pi_{1}(x_n))+\left[\frac{-\lambda_{1}^n}{\lambda(\A)^{n/d}},\frac{\lambda_{1}^n}{\lambda(\A)^{n/d}}\right]\right)\times [0,1]^{d-1}.
		\end{align*} 
		Here $\pi_{1}:\mathbb{R}^d\to\mathbb{R}$ is the projection map to the first coordinate. It follows from the above that 
		$$\L_{d}\left(\left(\bigcup_{\i\in (\I_1\times \cdots \times \I_{d})^n} S_{\i}\left(B\left(x_n,\frac{1}{\lambda(\A)^{n/d} }\right)\right)\right)\cap [0,1]^d\right)\leq \frac{2\# \I_{1}^n\cdot \lambda_{1}^n}{\lambda(\A)^{n/d}}.$$
		Using this inequality, together with \eqref{narrow rectangles}, we obtain $$\sum_{n=1}^{\infty}\L_{d}\left(\left(\bigcup_{\i\in (\I_1\times \cdots \times \I_{d})^n} S_{\i}\left(B\left(x_n,\frac{1}{\lambda(\A)^{n/d} }\right)\right)\right)\cap [0,1]^d\right)\ll \sum_{n=1}^{\infty}\frac{\# \I_{1}^n\cdot \lambda_{1}^n}{\lambda(\A)^{n/d}}<\infty.$$ The fact $\L_{d}(W(\mathfrak{S},(\x_n),1))=0$ now follows from the Borel-Cantelli lemma. 
		
		The proof that $\L_{d}(R(\mathfrak{S},1)))=0$ follows by a similar application of the Borel-Cantelli lemma. We omit the proof of this statement. 
		
	\end{proof}
	The following statement shows that in the statement of Theorem \ref{Main thm}, it is absolutely necessary to include a divergence assumption on the function $h$.
	\begin{thm}
		\label{Convergent sum}
		Let $\S$ be an iterated function system and assume that $h:\N\to [0,\infty)$ satisfies $\sum_{n=1}^{\infty}h(n)<\infty$. Then $W(\S,\x,h)$ has zero Lebesgue measure for any $\x\in X^{\N},$ and $R(\S,h)$ has zero Lebesgue measure.
	\end{thm}
	
	\begin{proof}
		Let us fix $h:\N\to [0,\infty)$ satisfying $\sum_{n=1}^{\infty}h(n)<\infty$. We will only prove that $W(\S,\x,h)$ has zero Lebesgue measure for any $\x\in X^{\mathbb{N}}$. The proof that $R(\S,h)$ has zero Lebesgue measure follows using the same arguments used in the proof of Theorem \ref{Exact overlap}. 
		
		Let $\x=(x_n)\in X^{\mathbb{N}}$ be fixed. For any $n\in \N$ we have
		\begin{align*}
			&\L_{d}\left(\left\{x:T_{\i}(x)\in B\left(x_n,\left(\frac{h(n)}{\lambda(\A)^n}\right)^{1/d}\right)\textrm{ for some }\i\in \I^n\right\}\right)\\
			\leq &\sum_{\i\in \I^n}\L_{d}\left(S_{\i} \left(B\left(x_n,\left(\frac{h(n)}{\lambda(\A)^n}\right)^{1/d}\right)\right)\right)\\
			=&h(n)\L_{d}(B(0,1))\cdot \sum_{\i\in \I^n}\frac{|Det(A_\i)|}{\lambda(\A)^n}\\
			=&h(n)\L_{d}(B(0,1)).
		\end{align*}
		Therefore, by our assumption on $h$ we have
		$$\sum_{n=1}^{\infty}\L_{d}\left(x:T_{\i}(x)\in B\left(x_n,\left(\frac{h(n)}{\lambda(\A)^n}\right)^{1/d}\right)\textrm{ for some }\i\in \I^n\right)
		\ll  \sum_{n=1}^{\infty}h(n)<\infty.$$ The result now follows by the Borel Cantelli lemma. 
	\end{proof}

	\section{Proof of Theorem \ref{Main thm}}
	\label{Section 5}
	%\label{Main Theorem section}
	%	Throughout this section we let $R>0$ be an arbitrary but fixed real number. To prove Theorem \ref{Main thm}, it suffices to show that the conclusions of statements 1., 2., and 3. hold for $\L_{\# \I\cdot d}$-almost every $\tt\in [-R,R]^{\# \I\cdot d}.$ In what follows we will let $U=[-R,R]^{\# \I\cdot d}.$
	
	%	Throughout this section the following reformulation of our shrinking target set is useful:
	%{S}_{\tt},(\pi_{\tt}((i_{n,m})_{m=1}^{\infty})_{n=1}^{\infty},\psi)=\bigcap_{l=1}^{\infty}\bigcup_{n\geq l}\bigcup_{\i\in I^n}S_{\i,\tt}\left(B\left(\pi_{\tt}((i_{n,m})_{m=1}^{\infty},\left(\frac{h(n)}{\lambda(\S_{\tt})^n}\right)^{1/d}\right)\right).$$ Notice that under the assumptions of Theorem \ref{Main thm}, for each $\i\in \I^*$ the set $S_{\i,\tt}(B(\pi_{\tt}((i_{n,m})_{m=1}^{\infty},\psi))=\pi_{\tt}(\i(i_{m,n})_{m=1}^{\infty})+E_n,$ where $E_n$ is an ellipse that only depends upon $n$. It also follows from the definition of $\lambda(\A)$ that $$\L_{d}(E_n)=\frac{h(n)}{\#\I^n}\L_{d}(B(0,1)).$$ 
	
	The first step in the proof of Theorem \ref{Main thm} is to obtain information about the distribution of the ellipses $$\left\{S_{\j,\tt}\left(B\left(\pi_{\tt}(\i),\frac{s}{\lambda(\A)^{n/d}}\right)\right)\right\}_{\j\in \I^n}$$ for small values of $s$ and for a typical $\tt$. The following lemma provides that information. It is convenient at this point to introduce some additional notation. Suppose a set of matrices $\A=\{A_i\}$ and a vector $\tt\in \R^{\# \I\cdot d}$ are given, then for each $\i \in \I^{\N}, s>0$ and $n\in \mathbb{N}$ we let $$B_{\tt}(s,\i,n):=B\left(\pi_{\tt}(\i),\frac{s}{\lambda(\A)^{n/d}}\right).$$ The following lemma is based upon an argument due to Benjamini and Solomyak \cite{BenSol}, which is in turn based upon an argument due to Peres and Solomyak \cite{PerSol}.
	
	\begin{lemma}
		\label{Counting pairs}
		Let $\{A_i\}$ be a collection of matrices satisfying the assumptions of Theorem \ref{Main thm}. Then for any $R>0$, $n\in\mathbb{N},$ $s>0$ and $\i\in \I^{\N},$ we have
		$$\int_{[-R,R]^{\#\I \cdot d}} \#\left\{\j,\k\in \I^n:\j\neq \k,\, S_{\j,\tt}\left(B_{\tt}(s,\i,n)\right)\cap S_{\k,\tt}\left(B_{\tt}(s,\i,n)\right)\neq \emptyset\right\}\, d\tt=\mathcal{O}_{R}(\#\I^n\cdot  s^d).$$
	\end{lemma}

	\begin{proof}
		Let $A$ be such that $A_{i}=A$ for all $i\in \I$. Such a matrix exists by our underlying assumptions. We start our proof by remarking that for any $n\in \N$ and $s>0,$ there exists an ellipse $$E_{n}=A^n\left(B\left(0,\frac{s}{\lambda(\A)^{n/d}}\right)\right),$$ such that for any $\i\in \I^{\N}$ and $\j\in \I^n,$ we have 
		\begin{equation}
			\label{introduce ellipse}
			S_{\j,\tt}\left(B_{\tt}(s,\i,n)\right)=\pi_{\tt}(\j \i)+E_n.
		\end{equation}
		We also remark that 
		\begin{equation}
			\label{En measure}
			\L_{d}(E_n)=\frac{|Det(A)|^ns^{d}\L_{d}(B(0,1))}{\lambda(\A)^n}.
		\end{equation}  It follows from \eqref{introduce ellipse} that for two words $\j,\k\in \I^n,$ we have $$S_{\j,\tt}(B_{\tt}(s,\i,n))\cap S_{\k,\tt}(B_{\tt}(s,\i,n))\neq \emptyset$$  if and only if $$\pi_{\tt}(\j\i)-\pi_{\tt}(\k\i)\in E_n-E_n.$$ Because $E_n$ is convex and symmetric around the origin, we have $E_n-E_n=2E_n.$ Therefore $$S_{\j,\tt}(B_{\tt}(s,\i,n))\cap S_{\k,\tt}(B_{\tt}(s,\i,n))\neq \emptyset$$  if and only if
		\begin{equation}
			\label{ellipse inclusion}
			\pi_{\tt}(\j\i)-\pi_{\tt}(\k\i)\in 2E_n.
		\end{equation}  %Also observe that if $|\i\wedge\j|=k$ then $$\pi_{\tt}(\i(i_{m,n})_{m=1}^{\infty})-\pi_{\tt}(\j(i_{m,n})_{m=1}^{\infty})\in 2E_n$$ if and only if $$\pi_{\tt}(i_{k}\ldots i_n(i_{m,n})_{m=1}^{\infty})-\pi_{\tt}(j_{k}\ldots j_m(i_{m,n})_{m=1}^{\infty})\in A^{-k+1}2E_n.$$
		We now rewrite our integral as follows
		\begin{align*}
			&\int_{[-R,R]^{\#\I \cdot d}} \#\left\{\j,\k\in \I^n:\j\neq \k,\, S_{\j,\tt}\left(B_{\tt}(s,\i,n)\right)\cap S_{\k,\tt}\left(B_{\tt}(s,\i,n)\right)\neq \emptyset\right\}\, d\tt\\
			=& \int_{[-R,R]^{\#\I \cdot d}} \sum_{\j\in \I^n}\sum_{k=1}^{n}\sum_{\k\in \I^n:|\j\wedge \k|=k}\chi_{\left\{\tt':S_{\j,\tt'}\left(B_{\tt}(s,\i,n)\right)\cap S_{\k,\tt'}\left(B_{\tt}(s,\i,n)\right)\neq \emptyset\right\}}(\tt)\, d\tt\\
			=& \int_{[-R,R]^{\#\I \cdot d}} \sum_{\j\in \I^n}\sum_{k=1}^{n}\sum_{\k\in \I^n:|\j\wedge \k|=k}\chi_{\left\{\tt':\pi_{\tt'}(\j\i)-\pi_{\tt'}(\k\i)\in 2E_n\right\}}(\tt)\, d\tt.
		\end{align*}In the final line we used the fact that $S_{\j,\tt}(B_{\tt}(s,\i,n))\cap S_{\k,\tt}(B_{\tt}(s,\i,n))\neq \emptyset$ is equivalent to \eqref{ellipse inclusion}.
		
		We remark that when $A_i=A$ for all $i\in \I$, then our formula for $\lambda(\A)$ can be simplified to $$\lambda(\A)=\# \I \cdot |Det(A)|.$$ Now using this identity together with Lemma \ref{Transverality estimate} and \eqref{En measure}, we obtain the following
		\begin{align*}
			&\int_{[-R,R]^{\#\I \cdot d}} \sum_{\j\in \I^n}\sum_{k=1}^{n}\sum_{\k\in \I^n:|\i\wedge \j|=k}\chi_{\left\{\tt':\pi_{\tt'}(\j\i)-\pi_{\tt'}(\k\i)\in 2E_n\right\}}(\tt)\, d\tt \\
			\ll_{R}& \sum_{\j\in \I^n}\sum_{k=1}^{n}\sum_{\k\in \I^n:|\i\wedge \j|=k} |Det(A)|^{-k+1}\frac{s^d}{\#\I^n}\\
			\ll_{R}&s^d\sum_{\i\in \I^n}\sum_{k=1}^{n} \frac{1}{\# \I^{k}|Det(A)|^{k-1}}\\
			\ll_{R}&s^d\sum_{\i\in \I^n}\sum_{k=1}^{n}\frac{1}{\lambda(\A)^k}\\
			\ll_{R}&s^d\sum_{\i\in \I^n}\sum_{k=1}^{\infty}\frac{1}{\lambda(\A)^k}\\
			\ll_{R}&s^d\sum_{\i\in \I^n}1\\
			\ll_{R}&s^d\#\I^n.
		\end{align*}
		In the penultimate line we used our assumption $\lambda(\A)>1$ to conclude that $\sum_{k=1}^{\infty}\lambda(\A)^{-k}=\O(1).$		
	\end{proof}

	\subsection{Proof of Theorem \ref{Main thm}.1}
	\begin{lemma}
		\label{All sequences}
		Let $\{A_i\}_{i\in \I}$ be a finite set of matrices satisfying the assumptions of Theorem \ref{Main thm} and let $R>0$ be arbitrary. For Lebesgue almost every $\tt\in [-R,R]^{\# \I\cdot d}$, for any $\epsilon>0$ there exists $c,s>0$ such that for any sequence $(x_n)\in X_{\tt}^{\N},$ if
		\begin{align*}
			T(\tt,s,n):=\max\left\{\# Y:Y\subset \I^n\,\textrm{ and } \{S_{\i,\tt}(x_n)\}_{\i \in Y} \textrm{ is }\left(s,A^n\left(B\left(0,\frac{1}{\lambda(\A)^{n/d}}\right)\right)\right)\textrm{-separated}\right\}
		\end{align*} 
		then $$\overline{d}\left(\{n:T(\tt,s,n)>c\cdot \#\I^n\}\right)\geq 1-\epsilon.$$
	\end{lemma}
	\begin{proof}
		Let $(i_{n,m}')\in \I^{\N\times \N}$ be arbitrary. Applying Lemma \ref{Density lemma}\footnote{In this lemma take $\Omega=[-R,R]^{\# \I\cdot d}$, $\eta$ to be the Lebesgue measure restricted to $[-R,R]^{\# \I\cdot d},$ and for each $n\in \N$ let $\{f_{l,n}\}_{l=1}^{R_n}=\{f_{\i,n}\}_{\i\in \I^n}$ where $f_{\i,n}(\tt)=S_{\i,\tt}(\pi_{\tt}((i_{n,m}')_m))$.}  and Lemma \ref{Counting pairs}, we know that for Lebesgue almost every $\tt\in [-R,R]^{\# \I\cdot d},$ for any $\epsilon>0$ there exists $c,s>0$ such that if $T'(\tt,s,n)$ is given by 
		\[
		\max\left\{\# Y:Y\subset \I^n\,\textrm{ and } \left\{S_{\i,\tt}(\pi_\tt((i_{n,m}')_{m=1}^{\infty}))\right\}_{\i \in Y} \textrm{ is }\left(s,A^n\left(B\left(0,\frac{1}{\lambda(\A)^{n/d}}\right)\right)\right)\textrm{-separated}\right\}
		\]
		then 
		\[\overline{d}\left(\{n:T'(\tt,s,n)>c\cdot \#\I^n\}\right)\geq 1-\epsilon.\]
		In other words, our desired conclusion holds for the specific choice of sequence $(\pi_{\tt}((i_{n,m}')_m))_n$. To complete our proof, we now need to show that for any $\tt\in[-R,R]^{\# \I\cdot d}$ for which our desired conclusion holds for the specific sequence $(\pi_{\tt}((i_{n,m}')_m))_n$, the same conclusion simultaneously holds for any sequence $(x_n)\in X_{\tt}^{\N}$ for the same choice of parameters. This fact follows from the simple observation that for any $n\in \N$ and $\i\in \I^{n},$ we have 
		\[
		S_{\i,\tt}(x_n)-S_{\i,\tt}(\pi_{\tt}((i_{n,m}')_{m=1}^{\infty}))=A^nx_n-A^n\pi_{\tt}((i_{n,m}')_{m=1}^{\infty}).
		\]
		Crucially the right hand side of the above does not depend upon $\i$. This observation implies that for each $n\in\mathbb{N},$ $\tt\in \R^{\# \I\cdot d}$, and $x_n\in X_{\tt},$ the sets $\{S_{\i,\tt}(x_n)\}_{\i\in \I^n}$ and $\{S_{\i,\tt}(\pi_{\tt}((i_{n,m}')_{m=1}^{\infty}))\}_{\i\in \I^n}$ are translates of each other. Therefore
		\begin{align*}
			&	\max\left\{\# Y:Y\subset \I^n\,\textrm{ and } \{S_{\i,\tt}((i_{n,m}')_{m=1}^{\infty})\}_{\i \in Y} \textrm{ is }\left(s,A^n\left(B\left(0,\frac{1}{\lambda(\A)^{n/d}}\right)\right)\right)\textrm{-separated}\right\}\\
			=&\max\left\{\# Y:Y\subset \I^n\,\textrm{ and } \{S_{\i,\tt}(x_n)\}_{\i \in Y} \textrm{ is }\left(s,A^n\left(B\left(0,\frac{1}{\lambda(\A)^{n/d}}\right)\right)\right)\textrm{-separated}\right\}
		\end{align*} for any $n\in \N$ and $s>0$. Therefore our desired conclusion holding for the specific sequence $(\pi_{\tt}((i_{n,m}')_m))_n$ immediately implies the same conclusion for all $(x_n)\in X_{\tt}^{\N}$ for the same choice of parameters. This completes the proof.
	\end{proof}

	%	Applying Lemma \ref{Density lemma} we know that for $\L_{\#\I\cdot d}$-almost every $\tt\in U$, for any $\epsilon>0$ there exists $c,s>0$ such that if $$T(\tt,s,n):=\max\left\{\# Y:Y\subset \I^n\,\textrm{ and } \{S_{\i,\tt}(B(\pi_{\tt}((i_{n,m})_{m=1}^{\infty} \frac{s}{\lambda(\A)^{n/d}}))\}_{\i \in Y} \textrm{ is }(s,E_n)-separated\right\}$$ then $$\{n:T(\tt,s,n)>c\cdot \#\I^n\}$$ has upper density at strictly large than $(1-\epsilon)$. We will now proof that our result holds for any such $\tt$. In what follows our value of $\tt$ is fixed.
	
	With Lemma \ref{All sequences} we can now prove Theorem \ref{Main thm}.1.
	\begin{proof}[Proof of Theorem \ref{Main thm}.1]
		To prove our statement, it suffices to show that the desired conclusion holds for Lebesgue almost every $\tt\in [-R,R]^{\#\I\cdot d}$ where $R>0$ is arbitrary. In what follows $R$ will be fixed. Now let $\tt$ belong to the full measure set of parameters for which the conclusion of Lemma \ref{All sequences} is satisfied. Let $\x=(x_n)\in X_{\tt}^{\N}$ and $h\in H$ be arbitrary. We now set out to prove that $W(S_{\tt},\x,h)$ has positive Lebesgue measure.
		
		It follows from the definition of $H$ (see \eqref{eq:H}) and Lemma \ref{All sequences}, that there exists $c,s>0$ such that if we let 
		$$T(n)=\max\left\{\# Y:Y\subset \I^n\,\textrm{ and } \{S_{\i,\tt}(x_n)\}_{\i \in Y} \textrm{ is }\left(s,A^n\left(B\left(0,\frac{1}{\lambda(\A)^{n/d}}\right)\right)\right)-\textrm{separated}\right\}$$ for each $n\in \N$, and $$\mathcal{N}=\{n\in \N:T(n)\geq c\cdot \# \I^n\}$$ then $$\sum_{n\in\mathcal{N}}h(n)=\infty.$$
		We now fix such a $c$ and $s$. 
		
		For each $n\in \mathcal{N}$ we let $W_n\subset \I^{n}$ be a set of words satisfying $\# W_n=T(n)$ and $$\{S_{\i,\tt}(x_n)\}_{\i \in W_n}\textrm{ is }\left(s,A^n\left(B\left(0,\frac{1}{\lambda(\A)^{n/d}}\right)\right)\right)-\textrm{separated}.$$ Instead of working directly with Euclidean balls it is more convenient to work with balls defined using the supremum norm. As such, replacing $s$ with a potentially smaller constant if necessary, we may assume that for each $n\in \N$ we have that 
		\begin{equation}
			\label{separated rectangles}
			\{S_{\i,\tt}(x_n)\}_{\i \in W_n}\textrm{ is }\left(s,A^n\left(\prod_{i=1}^{d}\left[\frac{-1}{\lambda(\A)^{n/d}},\frac{1}{\lambda(\A)^{n/d}}\right]\right)\right)-\textrm{separated}.
		\end{equation}

		Let $h':\N\to [0,\infty)$ be defined according to the rule $h'(n)=\min\{h(n),s^d\}.$ It follows immediately from the definition of $h'$ that for $n\in\mathcal{N}$ and $\i\in W_{n},$ we have 
		\begin{equation}
			\label{InclusionA}	S_{\i,\tt}\left(B\left(x_n,\left(\frac{h'(n)}{\lambda(\A)^n}\right)^{1/d}\right)\right)\subseteq S_{\i,\tt}\left(B\left(x_n,\left(\frac{h(n)}{\lambda(\A)^n}\right)^{1/d}\right)\right).
		\end{equation}
		Moreover, by the definition of $W_n,$ we also know that for  distinct $\i$ and $\j$ in $W_n$ we have \begin{equation}
			\label{Disjoint}S_{\i,\tt}\left(B\left(x_n,\left(\frac{h'(n)}{\lambda(\A)^n}\right)^{1/d}\right)\right)\bigcap S_{\j,\tt}\left(B\left(x_n,\left(\frac{h'(n)}{\lambda(\A)^n}\right)^{1/d}\right)\right)=\emptyset.
		\end{equation} It also follows from the definition that  $\sum_{n\in \mathcal{N}}h'(n)=\infty$.

		%	Let $h\in H$ be arbitrary. By the above and the definition of $H$, it follows that there exists $c,s>0$ such that $$\sum_{n:T(\tt,s,n)>c\cdot \#\I^n}h(n)=\infty.$$ In what follows these values of $c,s>0$ are fixed. 
		
		Again motivated by a desire to work with balls defined with respect to the supremum norm rather than the Euclidean norm, we define $\tilde{h}:\N\to [0,\infty)$ according to the rule $\tilde{h}(n)=h'(n)/d$. Let us denote 
		\[
		\mathcal H(n)=\left(\frac{\tilde h(n)}{\lambda(\A)^n}\right)^{\tfrac 1d}.
		\]We observe that  
		\begin{equation}
			\label{InclusionB}x_n+\prod_{i=1}^d\left[-\mathcal H(n),\mathcal H(n)\right]\subseteq B\left(x_n,\left(\frac{h'(n)}{\lambda(\A)^{n}}\right)^{1/d}\right)
		\end{equation} for all $n\in \N$. Crucially, since $\sum_{n\in \mathcal{N}}h'(n)=\infty$ we have 
		\begin{equation}
			\label{Measure divergence}
			\sum_{n\in \mathcal{N}}\tilde{h}(n)=\infty.
		\end{equation}
		For each $n\in \mathbb{N}$ we let 
		\[
		Z_{n}:=\bigcup_{\i\in W_n}S_{\i,\tt}\left(x_n+\prod_{i=1}^d\left[-\mathcal H(n), \mathcal H(n)\right]\right).
		\]
		It follows from \eqref{Disjoint} and \eqref{InclusionB} that the union defining $Z_n$ is disjoint. To prove our result, we will study the set $$Z_{\infty}:=\bigcap_{l=1}^{\infty}\bigcup_{n\in \mathcal{N}:n\geq l}Z_n.$$ It follows from \eqref{InclusionA} and \eqref{InclusionB} that $$Z_{\infty}\subseteq W(\S_{\tt},(x_n)_{n=1}^{\infty},h).$$ Therefore to prove our result, it suffices to show that $Z_{\infty}$ has positive Lebesgue measure. To do this we will use Lemma \ref{Quasi-independence on average}. We start by proving that the divergence assumption of Lemma \ref{Quasi-independence on average} is satisfied. The union defining $Z_n$ is disjoint and hence for each $n\in \mathcal{N}$ we have 
		\begin{equation}
			\label{Zn measure}
			\L_{d}(Z_n)=\sum_{\i\in W_n}\frac{|Det(A)|^n\tilde{h}(n)\L_{d}(B(0,1))}{\lambda(\A)^n}\asymp \frac{\# \I^n\cdot |Det(A)|^n\tilde{h}(n)\L_{d}(B(0,1))}{\lambda(\A)^n}\asymp \tilde{h}(n).
		\end{equation}
		The first $\asymp$ relation follows as $\#W_{n}\asymp \# \I^n$, and the second uses that $\lambda(\A)=\# \I\cdot |Det(\A)|$ when $A_i=A$ for all $i\in \I$.

		It follows now from \eqref{Measure divergence} and \eqref{Zn measure} that  $\sum_{n\in\mathcal{N}}\L_{d}(Z_n)=\infty.$ Hence we satisfy the first criterion of Lemma \ref{Quasi-independence on average}. It remains now to obtain meaningful bounds for $\L_{d}(Z_n\cap Z_{n'})$ for distinct $n,n'\in \mathcal{N}$. This we do below.

		We begin by remarking that since $A$ is a positive diagonal matrix, there exists $\lambda_1,\lambda_2,\ldots,\lambda_d\in (0,1)$ such that
		$$A=\begin{bmatrix}
			\lambda_1 &  &  & \\
			& \lambda_2 & &\\
			&  & \ddots & \\
			&  &  & \lambda_{d}
		\end{bmatrix}.$$
		For each $n\in \mathcal{N}$ and $\i\in W_n,$ we have $$S_{\i,\tt}\left(x_n+\prod_{i=1}^d\left[-\mathcal H(n),\mathcal H(n)\right]\right)=S_{\i,\tt}(x_n)+E_n$$ where $$E_{n}=\prod_{i=1}^{d}\left[-\lambda_i^n\mathcal H(n),\lambda_i^n\mathcal H(n)\right].$$ We also let $$E_{n}'=\prod_{i=1}^{d}\left[-\frac{s\lambda_i^n}{\lambda(\A)^{n/d}},\frac{s\lambda_i^n}{\lambda(\A)^{n/d}}\right].$$ Importantly we have $E_n\subset E_{n}'$ for all $n\in\N$. It also follows from \eqref{separated rectangles} that for $n\in \mathcal{N}$ and distinct $\i,\j\in W_n,$ we have
		$$\left(S_{\i,\tt}(x_n)+E_n'\right)\cap \left( S_{\j,\tt}(x_n)+E_{n}'\right)=\emptyset.$$
		
		We now set out to bound $\L_{d}(Z_n\cap Z_{n'})$ for $n'>n$. We split our analysis into two cases. Without loss of generality we may assume that $\lambda_{1}=\max_{1\leq i\leq d}\{\lambda_i\}$.
		\medskip

		\noindent\textbf{Case 1. $n<n'\leq n+\lfloor \frac{1/d \log \tilde{h}(n)-\log s}{\log \lambda_1/\lambda(\A)^{1/d}}\rfloor.$}	
		
		%	We now set out to bound $\sum_{n,m\in\mathcal{N}\cap\{1,\ldots,Q\}}\L_{d}(Z_n\cap Z_m),$ we begin by rewriting this expression as follows
		%	\begin{equation}
			%		\label{Measure sums}
			%		\sum_{n,m\in\mathcal{N}\cap\{1,\ldots,Q\}}\L_{d}(Z_n\cap Z_m)=\sum_{n\in\mathcal{N}\cap \{1,\ldots,Q\}}\L_{d}(Z_n)+2\sum_{n\in\mathcal{N}\cap\{1,\ldots,Q-1\}}\sum_{m\in\mathcal{N}\cap\{n+1,\ldots,Q\}}\L_{d}(Z_n\cap Z_m)
			%	\end{equation}
		If $n'$ satisfies $n<n'\leq n+\lfloor \frac{1/d \log \tilde{h}(n)-\log s}{\log \lambda_1/\lambda(\A)^{1/d}}\rfloor$ then \begin{equation}
			\label{thick rectangle}
			\frac{\lambda_{1}^{n'}s}{\lambda(\A)^{n'/d}}\geq \lambda_{1}^n\mathcal H(n).
		\end{equation} Put more informally, \eqref{thick rectangle} means that the rectangle $E_{n'}'$ is wider in the first coordinate than $E_{n}.$ 
		
		%Similarly, if if $m\in \{n+\lfloor \frac{1/d \log h(n)-\log s}{\log \lambda_1/\lambda(A)^{1/d}}\rfloor+1,\ldots, Q\}$ then 
		%	\begin{equation}
			%		\label{small rectangles}
			%		\frac{\lambda_{i}^{m}s}{\lambda(\A)^{m/d}}\leq \lambda_{i}^n\left(\frac{h(n)}{\lambda(\A)^n}\right)^{1/d}
			%	\end{equation}
		%	for all $1\leq i\leq d$. Which informally means that $E_{m}'$ can be placed inside of $E_n.$ We begin by focusing on the first set of parameters. 
		
		If $n<n'\leq n+\lfloor \frac{1/d \log \tilde{h}(n)-\log s}{\log \lambda_1/\lambda(\A)^{1/d}}\rfloor$ so \eqref{thick rectangle} holds, then for each $\i\in W_n$ if $\j\in W_{n'}$ is such that $$\left(S_{\j,\tt}(x_{n'})+E_{n'}'\right)\cap \left(S_{\i,\tt}(x_n)+E_n\right)\neq \emptyset$$ then 
		\begin{align*}
			S_{\j,\tt}(x_{n'})+E_{n'}'\subseteq\left(S_{\i,\tt}(x_n)+\prod_{i=1}^{d}\left[-\lambda_i^n \mathcal H(n)-\frac{2\lambda_{i}^{n'}s}{\lambda(\A)^{n'/d}},\lambda_i^n\mathcal H(n)+\frac{2\lambda_{i}^{n'}s}{\lambda(\A)^{n'/d}}\right]\right)
		\end{align*}Now using \eqref{thick rectangle} we see that if $\j\in W_{n'}$ is such that $$\left(S_{\j,\tt}(x_{n'})+E_{n'}'\right)\cap \left(S_{\i,\tt}(x_n)+E_n\right)\neq \emptyset,$$  then
		\begin{align*}
			&S_{\j,\tt}(x_{n'})+E_{n'}'\\
			\subseteq & \left(S_{\i,\tt}(x_n)+\left(\left[-\frac{3\lambda_{1}^{n'}s}{\lambda(\A)^{n'/d}},\frac{3\lambda_{1}^{n'}s}{\lambda(\A)^{n'/d}}\right]\times\prod_{i=2}^{d}\left[-\lambda_i^n\mathcal H(n)-\frac{2\lambda_{i}^{n'}s}{\lambda(\A)^{n'/d}},\lambda_i^n\mathcal H(n)+\frac{2\lambda_{i}^{n'}s}{\lambda(\A)^{n'/d}}\right]\right)\right)
		\end{align*} Now using the fact that the elements of the set $\{S_{\j,\tt}(x_{n'})+E_{n'}'\}_{\j\in W_{n'}}$ are disjoint, it follows from a volume argument that
		\begin{align}
			\label{Counting rectangles}
			&\#\left\{\j\in W_{n'}:\left(S_{\j,\tt}(x_{n'})+E_{n'}'\right)\cap \left(S_{\i,\tt}(x_n)+E_n\right)\neq \emptyset\right\}\nonumber\\
			\ll & \frac{\left(\frac{3\lambda_{1}^{n'}s}{\lambda(\A)^{n'/d}}\times \prod_{i=2}^{d}\left(\lambda_i^n\mathcal H(n)+\frac{2\lambda_{i}^{n'}s}{\lambda(\A)^{n'/d}}\right)\right)}{\prod_{i=1}^{d}\frac{2\lambda_{i}^{n'}s}{\lambda(\A)^{n'/d}}}\nonumber\\
			\ll &\prod_{i=2}^d\left (\lambda_i^n\mathcal H(n)\left(\frac{\lambda_i^{n'} s}{\lambda(\A)^{n'/d}}\right)^{-1}+1\right)
		\end{align}
		Using \eqref{Counting rectangles}, we now see that 
		\begin{align*}
			\L_{d}(Z_n\cap Z_{n'})=&\sum_{\i\in W_n}\L_{d}\left((S_{\i,\tt}(x_n)+E_n)\cap Z_{n'}\right)\\
			\leq &\sum_{\i \in W_n} \# \left\{\j\in W_{n'}:\left(S_{\j,\tt}(x_{n'})+E_{n'}'\right)\cap \left(S_{\i,\tt}(x_n)+E_n\right)\neq \emptyset\right\}\L_{d}(E_{n'})\\
			\ll &\# W_n\times \frac{\tilde{h}(n')|Det(A)|^{n'}}{\lambda(\A)^{n'}}\times\prod_{i=2}^d\left (\lambda_i^n\mathcal H(n)\left(\frac{\lambda_i^{n'} s}{\lambda(\A)^{{n'}/d}}\right)^{-1}+1\right).
		\end{align*}
		To continue, note that $\#W_n\asymp (\#\I)^n.$ Using this fact together with the identity $\lambda(\A)=\# \I\cdot  |Det(A)|$, we see that the above satisfies
		\begin{align*}
			&\# W_n\times \frac{\tilde{h}(n')|Det(A)|^{n'}}{\lambda(\A)^{n'}}\times\prod_{i=2}^d\left (\lambda_i^n\mathcal H(n)\left(\frac{\lambda_i^{n'} s}{\lambda(\A)^{{n'}/d}}\right)^{-1}+1\right)\\
			\ll &  \frac{\tilde{h}(n')}{\# \I^{n'-n}}\times\prod_{i=2}^d\left (\lambda_i^n\mathcal H(n)\left(\frac{\lambda_i^{n'} s}{\lambda(\A)^{n'/d}}\right)^{-1}+1\right)\\
			=&\frac{\tilde{h}(n')}{\# \I^{n'-n}}\left(1+\sum_{J\subset \{2,\ldots,d\}, J\neq\emptyset}\prod_{i\in J}\lambda_i^n\mathcal H(n)\left(\frac{\lambda_i^{n'} s}{\lambda(\A)^{n'/d}}\right)^{-1}\right)\\
			\ll&\frac{\tilde{h}(n')}{\# \I^{n'-n}}\left(1+\sum_{J\subset \{2,\ldots,d\}, J\neq\emptyset}\tilde{h}(n)^{\# J/d}\lambda(\A)^{(n'-n)\# J/d}\prod_{i\in J}\lambda_i^{n-n'}\right)\\
			\leq  &\frac{\tilde{h}(n')}{\# \I^{n'-n}}+\tilde{h}(n')\tilde{h}(n)^{1/d}\sum_{J\subset \{2,\ldots,d\}, J\neq\emptyset}\frac{\lambda(\A)^{(n'-n)\# J/d}}{\#\I^{n'-n}}\prod_{i\in J}\lambda_i^{n-n'},
		\end{align*}
		where in the penultimate line we have used the substitution $\mathcal H(n)=\left(\frac{\tilde{h}(n)}{\lambda(\A)^n}\right)^{1/d}$. Now making use of the identity $\lambda(\A)=\# \I\cdot  |Det(A)|$ again, as well as the fact that $|Det A|=\prod_{i=1}^d\lambda_i$, we obtain
		\begin{align*}
			&\frac{\tilde{h}(n')}{\# \I^{n'-n}}+\tilde{h}(n')\tilde{h}(n)^{1/d}\sum_{J\subset \{2,\ldots,d\}, J\neq\emptyset}\frac{\lambda(\A)^{(n'-n)\# J/d}}{\#\I^{n'-n}}\prod_{i\in J}\lambda_i^{n-n'}\\
			= & \frac{\tilde{h}(n')}{\# \I^{n'-n}}+\tilde{h}(n')\tilde{h}(n)^{1/d}\sum_{J\subset \{2,\ldots,d\}, J\neq\emptyset}\frac{\# \I^{(n'-n)\# J/d}|Det(A)|^{(n'-n)\# J/d}}{\#\I^{n'-n}}\prod_{i\in J}\lambda_i^{n-n'}\\
			= & \frac{\tilde{h}(n')}{\# \I^{n'-n}}+\tilde{h}(n')\tilde{h}(n)^{1/d}\sum_{J\subset \{2,\ldots,d\}, J\neq\emptyset}\# \I^{(n'-n) (\#J-d)/d}|Det(A)|^{(n'-n)\# J/d}\prod_{i\in J}\lambda_i^{n-n'}\\
			= & \frac{\tilde{h}(n')}{\# \I^{n'-n}}+\tilde{h}(n')\tilde{h}(n)^{1/d}\sum_{J\subset \{2,\ldots,d\}, J\neq\emptyset}\# \I^{(n'-n) (\#J-d)/d}\prod_{i=1}^d\lambda_i^{(n'-n)\# J/d}\prod_{i\in J}\lambda_i^{n-n'}\\
			\leq & \frac{\tilde{h}(n')}{\# \I^{n'-n}}+\tilde{h}(n')\tilde{h}(n)^{1/d}\sum_{J\subset \{2,\ldots,d\}, J\neq\emptyset}\# \I^{(n'-n) (\#J-d)/d}\prod_{i=1}^d\lambda_i^{(n'-n)\# J/d}\prod_{i=1}^d\lambda_i^{n-n'}\\
			=& \frac{\tilde{h}(n')}{\# \I^{n'-n}}+\tilde{h}(n')\tilde{h}(n)^{1/d}\sum_{J\subset \{2,\ldots,d\}, J\neq\emptyset}\# \I^{(n'-n) (\#J-d)/d}\prod_{i=1}^d\lambda_i^{(n'-n)(\# J-d)/d}\\
			=& \frac{\tilde{h}(n')}{\# \I^{n'-n}}+\tilde{h}(n')\tilde{h}(n)^{1/d}\sum_{J\subset \{2,\ldots,d\}, J\neq\emptyset}(\# \I \cdot |Det(A)|)^{(n'-n) (\#J-d)/d}\\
			=& \frac{\tilde{h}(n')}{\# \I^{n'-n}}+\tilde{h}(n')\tilde{h}(n)^{1/d}\sum_{J\subset \{2,\ldots,d\}, J\neq\emptyset}\lambda(\A)^{(n'-n) (\#J-d)/d}\\
			\leq& \frac{\tilde{h}(n')}{\# \I^{n'-n}}+\tilde{h}(n')\tilde{h}(n)^{1/d}\sum_{J\subset \{2,\ldots,d\}, J\neq\emptyset}\lambda(\A)^{-(n'-n)/d}\\
			\ll & \frac{\tilde{h}(n')}{\# \I^{n'-n}}+\tilde{h}(n')\tilde{h}(n)^{1/d}\lambda(\A)^{-(n'-n)/d}.
		\end{align*} 
		In the penultimate inequality we used that $\lambda(\A)^{(n'-n) (\#J-d)/d}\leq \lambda(\A)^{-(n'-n)/d}$ for any $J\subset \{2,\ldots,d\}$ such that $J\neq\emptyset$. In the final inequality we used that $\sum_{J\subset \{2,\ldots,d\}, J\neq\emptyset}1=\O(1).$ Summarising the above, we have shown that if $n<n'\leq n+\lfloor \frac{1/d \log h(n)-\log s}{\log \lambda_1/\lambda(\A)^{1/d}}\rfloor,$ then
		\begin{equation}
			\label{bound 1}
			\L_{d}(Z_n\cap Z_{n'})\ll \frac{\tilde{h}(n')}{\# \I^{n'-n}}+\tilde{h}(n')\tilde{h}(n)^{1/d}\lambda(\A)^{-(n'-n)/d}.
		\end{equation}
		\medskip
		
		\noindent \textbf{Case 2. $n'>n+\lfloor \frac{1/d \log \tilde{h}(n)-\log s}{\log \lambda_1/\lambda(\A)^{1/d}}\rfloor$.}

		We now concentrate on bounding $\L_{d}(Z_n\cap Z_n')$ for $n'>n+\lfloor \frac{1/d \log \tilde{h}(n)-\log s}{\log \lambda_1/\lambda(\A)^{1/d}}\rfloor.$ If $n'>n+\lfloor \frac{1/d \log \tilde{h}(n)-\log s}{\log \lambda_1/\lambda(\A)^{1/d}}\rfloor,$ then recalling the definition of $\mathcal H(n)$ and using that $\lambda_{1}=\max_{1\leq i\leq d}\{\lambda_{i}\}$ we have
		\begin{equation}
			\label{small rectangles}
			\frac{\lambda_{i}^{n'}s}{\lambda(\A)^{n'/d}}\leq  \lambda_{i}^n\left(\frac{\tilde{h}(n)}{\lambda(\A)^n}\right)^{1/d}\leq \lambda_i^n\mathcal H(n)
		\end{equation}for all $1\leq i\leq d$. Now by a similar argument to that given in Case 1, we have that if $\i\in W_n$ and
		$\j\in W_{n'}$ are such that $$\left(S_{\j,\tt}(x_{n'})+E_{n'}'\right)\cap \left(S_{\i,\tt}(x_n)+E_n\right)\neq \emptyset$$ then by \eqref{small rectangles} we have
		$$S_{\j,\tt}(x_{n'})+E_{n'}'\subseteq  \left(S_{\i,\tt}(x_n)+\prod_{i=1}^{d}\left[-3\lambda_i^n\mathcal H(n),3\lambda_i^n\mathcal H(n)\right]\right).
		$$ Therefore by a volume argument, we have that for any $\i\in W_{n}$
		\begin{align*}
			\#\left\{\j\in W_{n'}:\left(S_{\j,\tt}(x_{n'})+E_{n'}'\right)\cap \left(S_{\i,\tt}(x_n)+E_n\right)\neq \emptyset\right\}\ll & \frac{\prod_{i=1}^{d}\lambda_i^n\mathcal H(n)}{\prod_{i=1}^{d}\frac{\lambda_i^{n'}}{\lambda(\A)^{n'/d}}}. 
		\end{align*}
		Recalling the definition of $\mathcal H(n)$, we can continue from this estimate and obtain
		\begin{align*}		
			\#\left\{\j\in W_{n'}:\left(S_{\j,\tt}(x_{n'})+E_{n'}'\right)\cap \left(S_{\i,\tt}(x_n)+E_n\right)\neq \emptyset\right\}&\ll \tilde{h}(n) \frac{\prod_{i=1}^d\lambda_i^{n-n'}}{\lambda(\A)^{n-n'}}\\
			&\leq \tilde{h}(n)\frac{|Det(A)|^{n-n'}}{\# \I^{n-n'}|Det(A)|^{n-n'}}\\
			&= \frac{\tilde{h}(n)}{\#\I^{n-n'}}. 
		\end{align*}
		Using this bound, it follows that for $n'>n+\lfloor \frac{1/d \log h(n)-\log s}{\log \lambda_1/\lambda(\A)^{1/d}}\rfloor$ we have 
		\begin{align}
			\label{bound2}
			\L_{d}(Z_n\cap Z_{n'})
			\ll &\sum_{\i\in W_{n}}\#\left\{\j\in W_{n'}:\left(S_{\j,\tt}(x_{n'})+E_{n'}'\right)\cap \left(S_{\i,\tt}(x_n)+E_n\right)\neq \emptyset\right\}\cdot \L_{d}(E_{n'})\nonumber\\
			\ll &\sum_{\i\in W_{n}}\frac{\tilde{h}(n)}{\#\I^{n-n'}} \frac{|Det(A)|^{n'}\tilde{h}(n')}{\lambda(\A)^{n'}}\nonumber\\
			\leq & \#\I^{n}\frac{\tilde{h}(n)}{\#\I^{n-n'}} \frac{\tilde{h}(n')}{\# \I^{n'}}\nonumber \\
			=&\tilde{h}(n)\tilde{h}(n')
			%\# W_n\cdot h(n)\# \I^{m-n} \frac{h(m)}{\I^m}
			%\leq \#\I^n\cdot h(n)\# \I^{m-n} \frac{h(m)}{\I^m}
			%=h(n)h(m).
		\end{align}
		\medskip

		We are now ready to apply the estimates from Cases 1 and 2. Combining \eqref{bound 1} with \eqref{bound2}, we see that for any $n'>n$ we have 
		\begin{equation}
			\label{Universal bound}
			\L_{d}(Z_n\cap Z_n')\ll \frac{\tilde{h}(n')}{\# \I^{n'-n}}+\tilde{h}(n')\tilde{h}(n)^{1/d}\lambda(\A)^{-(n'-n)/d}+ \tilde{h}(n)\tilde{h}(n').
		\end{equation}
		Fix $Q\in \N$. Applying \eqref{Universal bound} we have
		\begin{align*}
			&\sum_{n,n'\in\mathcal{N}\cap\{1,\ldots,Q\}}\L_{d}(Z_n\cap Z_n')\\
			\ll & \sum_{n\in\mathcal{N}\cap \{1,\ldots,Q\}}\L_{d}(Z_n)+\underbrace{\sum_{n\in\mathcal{N}\cap\{1,\ldots,Q-1\}}\sum_{n'\in\mathcal{N}\cap\{n+1,\ldots,Q\}} \frac{\tilde{h}(n')}{\# \I^{n'-n}}}_{(A)}\\
			&+\underbrace{\sum_{n\in\mathcal{N}\cap\{1,\ldots,Q-1\}}\sum_{n'\in\mathcal{N}\cap\{n+1,\ldots,Q\}}\tilde{h}(n')\tilde{h}(n)^{1/d}\lambda(\A)^{-(n'-n)/d}}_{(B)}+ \underbrace{\sum_{n\in\mathcal{N}\cap\{1,\ldots,Q-1\}}\sum_{n'\in\mathcal{N}\cap\{n+1,\ldots,Q\}}\tilde{h}(n)\tilde{h}(n')}_{(C)}.
		\end{align*}
		We now bound terms (A), (B), and (C) individually. For term (A) we have 
		\begin{align}
			\label{Abound}
			\sum_{n\in\mathcal{N}\cap\{1,\ldots,Q-1\}}\sum_{n'\in\mathcal{N}\cap\{n+1,\ldots,Q\}} \frac{\tilde{h}(n')}{\# \I^{n'-n}}&=\sum_{n'\in \mathcal{N}\cap \{2,\ldots, Q\}}\sum_{n\in\mathcal{N}\cap \{1,\ldots, n'-1\}}\frac{\tilde{h}(n')}{\# \I^{n'-n}}\nonumber\\
			&=\sum_{n'\in \mathcal{N}\cap \{2,\ldots, Q\}}\tilde{h}(n')\sum_{n\in\mathcal{N}\cap \{1,\ldots, n'-1\}}\frac{1}{\# \I^{n'-n}}\nonumber\\
			&\ll \sum_{n'\in \mathcal{N}\cap \{2,\ldots, Q\}}\tilde{h}(n')\nonumber\\
			&\ll \sum_{n\in \mathcal{N}\cap \{1,\ldots, Q\}}\L_{d}(Z_n).
		\end{align}
		In the penultimate line we used that $\# \I > 1$ and therefore $\sum_{n\in\mathcal{N}\cap \{1,\ldots, n'-1\}}\frac{1}{\# \I^{n'-n}}=\mathcal{O}(1)$. In the final line we used \eqref{Zn measure}. Now focusing on term (B) we have 
		\begin{align}
			\label{Bbound}
			&\sum_{n\in\mathcal{N}\cap\{1,\ldots,Q-1\}}\sum_{n'\in\mathcal{N}\cap\{n+1,\ldots,Q\}}\tilde{h}(n')\tilde{h}(n)^{1/d}\lambda(\A)^{-(n'-n)/d}\\
			\ll& \sum_{n\in\mathcal{N}\cap\{1,\ldots,Q-1\}}\sum_{n'\in\mathcal{N}\cap\{n+1,\ldots,Q\}}\tilde{h}(n')\lambda(\A)^{-(n'-n)/d}\nonumber\\
			=&\sum_{n'\in \mathcal{N}\cap \{2,\ldots, Q\}}\sum_{n\in\mathcal{N}\cap \{1,\ldots, n'-1\}}\tilde{h}(n')\lambda(\A)^{-(n'-n)/d}\nonumber\\
			=&\sum_{n'\in \mathcal{N}\cap \{2,\ldots, Q\}}\tilde{h}(n')\sum_{n\in\mathcal{N}\cap \{1,\ldots, n'-1\}}\lambda(\A)^{-(n'-n)/d}\nonumber\\
			\ll& \sum_{n'\in \mathcal{N}\cap \{2,\ldots, Q\}}\tilde{h}(n')\nonumber\\
			\ll& \sum_{n\in \mathcal{N}\cap \{1,\ldots, Q\}}\L_{d}(Z_n).
		\end{align} In the penultimate line we used that $\lambda(\A)>1$ so $\sum_{n\in\mathcal{N}\cap \{1,\ldots, n'-1\}}\lambda(\A)^{-(n'-n)/d}=\O(1)$. In the final line we used \eqref{Zn measure}. Finally for term (C), using \eqref{Zn measure} we have
		\begin{equation}
			\label{Cbound}
			\sum_{n\in\mathcal{N}\cap\{1,\ldots,Q-1\}}\sum_{n'\in\mathcal{N}\cap\{n+1,\ldots,Q\}}h(n)h(n')\ll \left(\sum_{n\in\mathcal{N}\cap \{1,\ldots, Q\}}\L_{d}(Z_n)\right)^2.
		\end{equation}
		Using \eqref{Abound}, \eqref{Bbound}, and \eqref{Cbound}, we see that
		$$\sum_{n,m\in\mathcal{N}\cap\{1,\ldots,Q\}}\L_{d}(Z_n\cap Z_m)
		\ll  \sum_{n\in\mathcal{N}\cap \{1,\ldots,Q\}}\L_{d}(Z_n)+\left(\sum_{n\in\mathcal{N}\cap \{1,\ldots, Q\}}\L_{d}(Z_n)\right)^2.$$ Using this bound and Lemma \ref{Quasi-independence on average}, we obtain
		\begin{align*}
			\L_{d}\left(Z_{\infty}\right)&\geq \limsup_{Q\to\infty}\frac{\left(\sum_{n\in \mathcal{N}\cap \{1,\ldots,Q\}}\L_{d}(Z_n)\right)^2}{\sum_{n,n'\in\mathcal{N}\cap\{1,\ldots,Q\}}\L_{d}(Z_n\cap Z_n')}\\
			&\gg \limsup_{Q\to\infty}\frac{\left(\sum_{n\in \mathcal{N}\cap \{1,\ldots,Q\}}\L_{d}(Z_n)\right)^2}{\sum_{n\in\mathcal{N}\cap \{1,\ldots,Q\}}\L_{d}(Z_n)+\left(\sum_{n\in\mathcal{N}\cap \{1,\ldots, Q\}}\L_{d}(Z_n)\right)^2}\\
			&= 1.
		\end{align*}
		Therefore $\L_{d}(Z_{\infty})>0.$ This completes our proof.

	\end{proof}

	\subsection{Proofs of Theorem \ref{Main thm}.2 and Theorem \ref{Main thm}.3}
	%	It follows from Theorem \ref{Main thm}.1 that for Lebesgue almost every $\tt\in \mathbb{R}^{\# \I\cdot d}$, for any $h\in H$ the set $W(\mathcal{S}_{\tt},(\pi_{\tt}((i_{m})_{m=1}^{\infty})_{n=1}^{\infty},\psi)$ has positive Lebesgue measure. For the rest of this proof we fix a $\tt$ for which this holds. We now fix a $h\in H$ that is also decaying regularly. Our goal is to now show that almost every $x\in X_{\tt}$ is contained in $W(\mathcal{S}_{\tt},(\pi_{\tt}((i_{m})_{m=1}^{\infty})_{n=1}^{\infty},\psi).$ 
	
	We now bring our attention to the proofs of statements 2. and 3. from Theorem \ref{Main thm}. We begin by proving a number of technical statements. Our first step is the following variant of the well known 3r covering lemma.
	
	\begin{lemma}
		\label{5r covering lemma}
		Let $\{E_n=\prod_{i=1}^{d}[-\delta_{i,n},\delta_{i,n}]\}_{n=1}^{\infty}$ be a sequence of rectangles. Assume that $(\delta_{i,n})_{n=1}^{\infty}$ is decreasing for each $1\leq i\leq d$. Then for any sequence of points $(y_n)_{n=1}^{\infty}\subset \mathbb{R}^d$ we have that there exists $M\subset \N$ satisfying:
		\begin{enumerate}
			\item $\bigcup_{n=1}^{\infty}(y_{n}+E_n)\subset \bigcup_{m\in M}(y_m+3E_m)$
			\item $(y_m+E_m)\cap (y_{m'}+E_{m'})=\emptyset$ for distinct $m,m'\in M.$
		\end{enumerate}
	\end{lemma}
	\begin{proof}
		We define $M$ inductively. Let $M_1=\{1\}.$ Then for any $n\geq 1$ such that $(y_n+E_n)\cap (y_1+E_1)\neq \emptyset$ we have $(y_n+E_n)\subseteq (y_1+3E_1).$ This follows from the assumption that $(\delta_{i,n})$ is decreasing for each $i$. Now let $k\in\mathbb{N}$. Suppose that $M_k=\{m_i\}_{i=1}^k$ has been constructed and satisfies the following: for any $n\in \N$ such that $(y_n+E_n)\cap \bigcup_{i=1}^{k}(y_{m_i}+E_{m_i})\neq \emptyset$ we have $(y_n+E_n)\subseteq \bigcup_{i=1}^{k}(y_{m_i}+3E_{m_i}),$ and for any distinct $m,m'\in M_k$ we have $(y_m+E_m)\cap (y_{m'}+E_{m'})=\emptyset$. If for all $n\in \N$ we have $(y_n+E_n)\cap \bigcup_{i=1}^{k}(y_{m_i}+E_{m_i})\neq \emptyset$ we let $M=M_{k}$. If this is not the case we let 
		\[
		m_{k+1}:=\inf \left\{n\in \N:(y_{n}+E_{n})\cap \bigcup_{i=1}^{k}(y_{m_i}+E_{m_i})=\emptyset\right\}.
		\]
		We then define $M_{k+1}:=M_{k}\cup \{m_{k+1}\}.$ By definition, $y_{m_{k+1}}+E_{m_{k+1}}$ is disjoint from all the $y_{m_i}+E_{m_i}$ with $i\le k$. Now suppose $n\geq 1$ is such that $(y_n+E_n)\cap \bigcup_{i=1}^{k+1}(y_{m_i}+E_{m_i})\neq \emptyset.$ If $n$ satisfies $(y_n+E_n)\cap \bigcup_{i=1}^{k}(y_{m_i}+E_{m_i})\neq \emptyset$ then $(y_n+E_n)\subseteq \bigcup_{i=1}^{k}(y_{m_i}+3E_{m_i})$ by our inductive hypothesis. If $n$ satisfies $(y_n+E_n)\cap \bigcup_{i=1}^{k}(y_{m_i}+E_{m_i})= \emptyset$ and $(y_n+E_n)\cap (y_{m_{k+1}}+E_{m_{k+1}})\neq \emptyset$, then $n\geq m_{k+1}$ by the definition of $m_{k+1}.$ Now using the fact that each $(\delta_{i,n})$ is decreasing we have that $(y_n+E_n)\subset (y_{m_{k+1}}+3E_{m_{k+1}}).$ It is clear therefore that if $n$ is such that $(y_n+E_n)\cap \bigcup_{i=1}^{k+1}(y_{m_i}+E_{m_i})\neq \emptyset$ then $(y_n+E_n)\subseteq \bigcup_{i=1}^{k+1}(y_{m_i}+3E_{m_i}).$
		
		Repeating these steps we see that either this process eventually terminates and we can define $M=M_k$ for some $k\in \N$, or this process continues indefinitely and we can define $M=\cup_{k=1}^{\infty} M_k.$ In either case it is clear that $M$ satisfies both $1.$ and $2.$

	\end{proof}
	The following lemma demonstrates that under appropriate conditions, the Lebesgue measure of a shrinking target set defined by balls does not change if we multiply the radii of these balls by an arbitrarily small positive constant.
	
	\begin{lemma}
		\label{arbitrary constant lemma}
		Let $\S=\{S_{i}\}_{i\in \I}$ be an IFS, $(x_n)\in X^{\N}$ and $h:\N\to [0,\infty)$ be decreasing. Assume that there exists a positive diagonal matrix $A$ such that the matrix parts of each $S_i$ equals $A$. Then 
		$$\L_{d}\left(\bigcap_{0<c\leq 1}W(\mathcal{S},(x_n)_{n=1}^{\infty},c\cdot h)\right)=\L_{d}(W(\mathcal{S},(x_n)_{n=1}^{\infty},h)).$$
	\end{lemma}
	\begin{proof}
		We may assume that $\L_{d}(W(\mathcal{S},(x_n)_{n=1}^{\infty},h))>0$. Otherwise our result follows trivially. For the purposes of obtaining a contradiction, suppose our desired equality does not hold. In that case there exists $0<c\leq 1$ such that $$\L_{d}(W(\mathcal{S},(x_n)_{n=1}^{\infty},h)\setminus W(\mathcal{S},(x_n)_{n=1}^{\infty},c\cdot h))>0.$$ This in turn implies that there exists $N\in\mathbb{N}$ such that 
		\[
		\L_{d}\left(W(\mathcal{S},(x_n)_{n=1}^{\infty},h)\cap \left\{x: x\notin S_{\i}\left(B\left(x_{|\i|},\left(\frac{c \cdot h(|\i|)}{\lambda(\A)^{|\i|}}\right)^{1/d}\right)\right)\,\forall \i\in \bigcup_{m\geq N}\I^m\right\}\right)>0.
		\]
		Our goal now is to contradict this inequality by showing that $$W(\mathcal{S},(x_n)_{n=1}^{\infty},h)\cap \left\{x: x\notin S_{\i}\left(B\left(x_{|\i|},\left(\frac{c \cdot h(|\i|)}{\lambda(\A)^{|\i|}}\right)^{1/d}\right)\right)\,\forall \i\in \bigcup_{m\geq N}\I^m\right\}$$ has no density points. It will then follow from Theorem \ref{Lebesgue density theorem} that this set has zero Lebesgue measure, and as such we will have the desired contradiction. Suppose that $z$ is a density point for this set. Note that $z$ is then automatically a density point for $W(\mathcal{S},(x_n)_{n=1}^{\infty},h)$.  As such there exists $r'>0$ such that \begin{equation}
			\label{halfdensity}\L_d(W(\mathcal{S},(x_n)_{n=1}^{\infty},h)\cap B(z,r))>\L_d(B(z,r))/2
		\end{equation} for all $0<r<r'$. 
		
		Now, for each $\i \in \I^*$, denote 
		\[
		Q(\i)=\prod_{i=1}^{d}\left[-\left(\frac{ h(|\i|)}{\lambda(\A)^{|\i|}}\right)^{1/d},\left(\frac{ h(|\i|)}{\lambda(\A)^{|\i|}}\right)^{1/d}\right].
		\]
		Let $r<r'$ be arbitrary and let $N'> N$ be sufficiently large that $$Diam\left(S_{\i}\left(x_{|\i|}+Q(\i)\right)\right)<r\textrm{ for all }\i \in \bigcup_{m=N'}^{\infty}\I^m.$$
		Let $$\Omega:=\left\{\i\in \bigcup_{m=N'}^\infty\I^m:S_{\i}\left(x_{|\i|}+Q(\i)\right)\cap B(z,r)\neq\emptyset\right\}.$$ By the definition of $\Omega$ and $N'$ we have  
		\begin{equation}
			\label{inclusions}
			W(\mathcal{S},(x_n)_{n=1}^{\infty},h)\cap B(z,r)\subseteq  \bigcup_{\i\in \Omega}S_{\i}\left(x_{|\i|}+Q(\i)\right)\subseteq B(z,2r).
		\end{equation}
		Now it follows from \eqref{halfdensity} and the first inclusion in \eqref{inclusions} that \begin{equation}
			\label{covering measure}
			\L_d\left(\bigcup_{\i\in \Omega}S_{\i}\left(x_{|\i|}+Q(\i)\right)\right)\geq \L_d(B(z,r))/2.
		\end{equation}
		At this point we want to apply Lemma \ref{5r covering lemma}. However, before we can apply this lemma we need to check that we satisfy the relevant hypothesis. Because each of the contractions $S_i$ share the same matrix part, and this matrix is a positive diagonal matrix, there exists $\lambda_1,\ldots,\lambda_d$ such that for each $\i \in \I^{*}$ we have $$S_{\i}\left(x_{|\i|}+Q(\i)\right)=S_{\i}(x_{|\i|})+\prod_{i=1}^{d}\left[-\lambda_i^{|\i|}\left(\frac{ h(|\i|)}{\lambda(\A)^{|\i|}}\right)^{1/d},\lambda_{i}^{|\i|}\left(\frac{ h(|\i|)}{\lambda(\A)^{|\i|}}\right)^{1/d}\right].$$ Now using the fact that $h$ is decreasing, we see that the set of rectangles $$\left\{\prod_{i=1}^{d}\left[-\lambda_i^{|\i|}\left(\frac{ h(|\i|)}{\lambda(\A)^{|\i|}}\right)^{1/d},\lambda_{i}^{|\i|}\left(\frac{ h(|\i|)}{\lambda(\A)^{|\i|}}\right)^{1/d}\right]\right\}_{\i\in \Omega}$$ can be enumerated so that the hypothesis of Lemma \ref{5r covering lemma} is satisfied; that is, the side lengths of the rectangles are all simultaneously decreasing. As such we can assert that there exists $\Omega'\subset\Omega,$ such that for $\i,\i'\in \Omega$ satisfying $\i\neq \i'$ we have 
		\begin{equation}
			\label{property1}S_{\i}\left(x_{|\i|}+Q(\i)\right)\cap S_{\i'}\left(x_{|\i'|}+Q(\i')\right)=\emptyset
		\end{equation} and
		\begin{equation}
			\label{property2}
			\bigcup_{\i\in \Omega}S_{\i}\left(x_{|\i|}+Q(\i)\right)
			\subseteq \bigcup_{\i\in \Omega'} S_{\i}\left(x_{|\i|}+3\cdot Q(\i)\right).
		\end{equation}
		Using \eqref{covering measure}, \eqref{property1} and \eqref{property2} we observe
		\begin{align}
			\label{final part}
			\L&\left(\bigcup_{\i\in \Omega'}S_{\i}\left(x_{|\i|}+Q(\i)\right)\right)
			=\sum_{\i\in \Omega'}\L_{d}\left(S_{\i}\left(x_{|\i|}+Q(\i)\right)\right)
			=\frac{1}{3^d}\sum_{\i\in \Omega'}\L_{d}\left(S_{\i}\left(x_{|\i|}+3\cdot Q(\i)\right)\right)\nonumber\\
			\geq&\frac{1}{3^d}\L_{d}\left(\bigcup_{i\in \Omega'} S_{\i}\left(x_{|\i|}+3\cdot Q(\i)\right)\right)
			\geq \frac{1}{3^d}\L_{d}\left(\bigcup_{i\in \Omega} S_{\i}\left(x_{|\i|}+ Q(\i)\right)\right)\nonumber\\
			\geq& \frac{\L_{d}(B(z,r))}{2\cdot 3^d}
		\end{align}
		For any $\i\in \Omega'$ we have 
		$$S_{\i}\left(B\left(x_{|\i|},\left(\frac{ c\cdot h(|\i|)}{\lambda(\A)^{|\i|}}\right)^{1/d}\right)\right)\subseteq S_{\i}\left(x_{|\i|}+Q(\i)\right)$$ and $$\L_{d}\left(S_{\i}\left(B\left(x_{|\i|},\left(\frac{ c\cdot h(|\i|)}{\lambda(\A)^{|\i|}}\right)^{1/d}\right)\right)\right)\geq \frac{c}{d^{d/2}} \L_{d}\left(S_{\i}\left(x_{|\i|}+Q(\i)\right)\right).$$
		Now using \eqref{property1}, \eqref{final part} and the above we have 
		\begin{align}
			\label{measure estimate}
			\L_{d}&\left(\bigcup_{\i\in \Omega'}S_{\i}\left(B\left(x_{|\i|},\left(\frac{ c\cdot h(|\i|)}{\lambda(\A)^{|\i|}}\right)^{1/d}\right)\right)\right)\nonumber 
			=\sum_{\i\in \Omega'}\L_d\left(S_{\i}\left(B\left(x_{|\i|},\left(\frac{ c\cdot h(|\i|)}{\lambda(\A)^{|\i|}}\right)^{1/d}\right)\right)\right)\nonumber\\
			\geq& \frac{c}{d^{d/2}}\sum_{\i\in \Omega'} \L_{d}\left(S_{\i}\left(x_{|\i|}+Q(\i)\right)\right)\nonumber
			\geq  \frac{c}{d^{d/2}}	\L\left(\bigcup_{\i\in \Omega'}S_{\i}\left(x_{|\i|}+Q(\i)\right)\right)\nonumber\\
			\geq&   \frac{c}{d^{d/2}} \cdot \frac{\L_{d}(B(z,r))}{2\cdot 3^d}.
		\end{align} Now using \eqref{inclusions} and \eqref{measure estimate}, and remembering that $N'>N$, we have 
		\begin{align*}
			&\L_d\left(W(\mathcal{S},(x_n)_{n=1}^{\infty},h)\cap \left\{x: x\notin S_{\i}\left(B\left(x_{|\i|},\left(\frac{c \cdot h(|\i|)}{\lambda(\A)^{|\i|}}\right)^{1/d}\right)\right)\,\forall \i\in \bigcup_{m\geq N}\I^m\right\}\cap B(z,2r)\right)\\
			\leq &\L_d\left( \left\{x: x\notin S_{\i}\left(B\left(x_{|\i|},\left(\frac{c \cdot h(|\i|)}{\lambda(\A)^{|\i|}}\right)^{1/d}\right)\right)\,\forall \i\in \cup_{m\geq N}\I^m\right\}\cap B(z,2r)\right)\\
			\leq &\L_d(B(z,2r))-\L_{d}\left(\bigcup_{i\in \Omega'}S_{\i}\left(B\left(x_{|\i|},\left(\frac{ c\cdot h(|\i|)}{\lambda(\A)^{|\i|}}\right)^{1/d}\right)\right)\right)\\
			\leq & \L(B(z,2r))- \frac{c}{d^{d/2}} \cdot \frac{\L_{d}(B(z,r))}{2\cdot 3^d}\\
			\leq & \L_d(B(z,2r))\left(1-\frac{c}{2\cdot 3^d\cdot 2^d\cdot d^{d/2}}\right).
		\end{align*}Since $r<r'$ was arbitrary, this means that $z$ is not a density point for $$W(\mathcal{S},(x_n)_{n=1}^{\infty},h)\cap \left\{x: x\notin S_{\i}\left(B\left(x_{|\i|},\left(\frac{c \cdot h(|\i|)}{\lambda(\A)^{|\i|}}\right)^{1/d}\right)\right)\,\forall \i\in \bigcup_{m\geq N}\I^m\right\}.$$ Hence we have obtained the desired contradiction. This completes our proof. 
	\end{proof}
	The following lemma shows that under appropriate conditions, if a subset of a self-affine set has positive measure, then almost every element of the self-affine set is contained in some image of this set.
	\begin{lemma}
		\label{Full measure images}
		Let $\S=\{S_i\}_{i\in \I}$ be an IFS with self-affine set $X$. Assume that $\L_{d}(X)>0$ and that $X$ is differentiation regular. Then for any Borel set $W\subset X$ satisfying $\L_{d}(W)>0$ we have
		$$\L_{d}\left(\bigcup_{\i\in \I^*}S_{\i}(W)\right)=\L_{d}(X).$$
	\end{lemma}

	\begin{proof}
		Let $\mathcal{V}$ and $\eta$ be as in the definition of differentiation regular in Section \ref{sec:statements}. Now let $x\in X$ be arbitrary. Let $\{V_{j}(x)\}$ be a sequence in $\mathcal{V}$ whose existence is guaranteed by the definition of differentiation regular. Let us fix a $V_{j}(x)$ and let $\i'$ be the corresponding word in the definition of differentiation regular satisfying $$S_{\i'}(X)\subset V_{j}(x)\textrm{ and }\L_{d}(S_{\i'}(X))\geq \eta \L_{d}(V_{j}(x)).$$ Then we have the following:
		\begin{align*}
			\L\left( V_{j}(x)\cap \bigcup_{\i\in \I^*}S_{\i}(W) \right)\geq \L\left(S_{\i'}(W)\right)&=Det(A_{\i'})\cdot \L\left(W\right)\\
			&=\frac{Det(A_{\i'})\L(X)\cdot \L\left(W\right)}{\L(X)}\\
			&=\frac{\L(S_{\i'}(X))\cdot \L\left(W\right)}{\L(X)}\\
			&=\frac{\eta\L(V_{j}(x))\cdot \L\left(W\right)}{\L(X)}.\\
		\end{align*}
		It follows therefore that for any $x\in X$ we have 
		$$\lim_{V\to x,x\in V\in\mathcal{V}}\frac{\L(V\cap \bigcup_{\i\in \I^*}S_{\i}\left(W\right))}{\L(V)}    \geq \frac{\eta\L\left(W\right)}{\L(X)}>0.$$ Therefore, by the definition of a density basis, it follows that Lebesgue almost every $x\in X$ is contained in $\bigcup_{\i\in \I^*}S_{\i}\left(W\right).$ This completes our proof.
		
	\end{proof}
	
	We are now in a position to prove Theorem \ref{Main thm}.2. 
	
	\begin{proof}[Proof of Theorem \ref{Main thm}.2]
		Let $\tt$ belong to the full measure set for which the conclusion of Theorem \ref{Main thm}.1 is true. We will now show that this $\tt$ also satisfies the conclusion of Theorem \ref{Main thm}.2, and in doing so complete our proof. Let us now fix $x\in X_{\tt}$ and $h\in H$ that is decaying regularly and decreasing. Because of our assumptions on $\tt,$ we know that $\L_{d}(W(S_{\tt},x,h))>0$.	Since $A$ is a diagonal matrix we know that $X$ is differentiation regular by Lemma \ref{Shmerkin lemma}. Therefore, by Lemma \ref{arbitrary constant lemma}, and Lemma \ref{Full measure images}, we have that Lebesgue almost every element of $ X_{\tt}$ belongs to 
		$$B:=\bigcup_{i\in \I^*}S_{\i,\tt}\left(\bigcap_{0<c\leq 1}W(S_{\tt},x,c\cdot h)\right).$$ Therefore to prove the result it suffices to show that if $y\in B$ then $y\in W(S_{\tt},x,h).$ With this goal in mind, we fix $y\in B$ arbitrarily. By definition, there exists $\j\in \I^{*}$ and $z\in \bigcap_{0<c\leq 1}W(S_{\tt},x,c\cdot h)$ such that $y=S_{\j,\tt}(z).$ Now let $\gamma>0$ be such that $$\frac{h(n+1)}{h(n)}\geq \gamma $$ for all $n\in \mathbb{N}.$ Such a $\gamma$ exists because of our assumption that $h$ is decaying regularly. We let 
		$$c=\frac{\gamma^{|\j|}}{\lambda(\A)^{|\j|}}.$$ Suppose $\k\in \I^*$ is such that
		\begin{equation}
			\label{z hit target}
			T_{\k,\tt}(z)\in B\left(x,\left( \frac{c\cdot h(|\k|)}{\lambda(\A)^{|\k|}}\right)^{1/d}\right).
		\end{equation}
		Then by the definition of $c,$ we have
		\begin{equation}
			\label{y hit target}
			(T_{\k,\tt}\circ T_{\overline{\j},\tt})(y)\in B\left(x,\left(\frac{c\cdot h(|\k|)}{\lambda(\A)^{|\k|}}\right)^{1/d}\right)\subseteq B\left(x,\left(\frac{h(|\k|+|\overline{\j}|)}{\lambda(\A)^{|\k|+|\overline{\j}|}}\right)^{1/d}\right).
		\end{equation}
		Since $z\in \bigcap_{0<c\leq 1}W(S_{\tt},x,c\cdot h)$ there are infinitely many $\k$ such that \eqref{z hit target} is satisfied. It follows that there are infinitely many $\k$ such that \eqref{y hit target} is satisfied. Therefore $y\in W(\mathcal{S}_{\tt},x,h).$ This completes the proof.
		
	\end{proof}
	
	The proof of Theorem \ref{Main thm}.3 is similar to the proof of Theorem \ref{Main thm}.2. We include the details for completion.
	
	\begin{proof}[Proof of Theorem \ref{Main thm}.3]
		Let $U\subset \mathbb{R}^{\#\I\cdot d}$ and $\i\in\I^{\N}$ be as in the statement of Theorem \ref{Main thm}.3. Duplicating the arguments given in the proof of Theorem \ref{Main thm}.2, we can show that for Lebesgue almost every $\tt\in U$, for any $h\in H$ that is decaying regularly and decreasing, we have that Lebesgue almost every $x\in X_{\tt}$ is contained in $$B:=\bigcup_{i\in \I^*}S_{\i,\tt}\left(\bigcap_{0<c\leq 1} W(\mathcal{S}_{\tt},\pi_{\tt}(\i),c\cdot h)\right).$$
		Now let us fix a $\tt\in U$ belonging to the full measure set for which this conclusion is true. By our underlying assumptions, we may also assume that this $\tt$ satisfies $\pi_\tt(\i)\in int(X_\tt).$ We now show that this $\tt$ also satisfies the conclusions of Theorem \ref{Main thm}.3.
		
		Let us now fix $h\in H$ that is decaying regularly and decreasing. Since each element of $S_{\tt}$ maps sets of Lebesgue measure zero to sets of Lebesgue measure zero, our fixed parameter $\tt$ also satisfies
		\begin{equation}
			\label{no preimage}
			\L_{d}\left(B \setminus  \bigcup_{\i\in \I^{*}}S_{\i,\tt}(X_{\tt}\setminus B)\right)=\L_{d}(X_{\tt}).
		\end{equation}
		Now using the fact $\pi_\tt(\i)\in int(X_\tt),$ we see that we can replace $h$ with a sufficiently small bounded function if necessary, so that without loss of generality for all $n\in \N$ we have 
		\begin{equation}
			\label{Ball inclusion}
			B\left(\pi_{\tt}(\i),\left(\frac{h(n)}{\lambda(\A)^{n}}\right)^{1/d}\right)\in int(X_{\tt}).
		\end{equation} We will now show that any element of $$C:=B \setminus  \bigcup_{\i\in \I^{*}}S_{\i,\tt}(X_{\tt}\setminus B)$$ is contained in $$\left\{x: \exists \j\in\I^{\N} \textrm{ such that }(T_{j_{n},\tt}\circ \cdots \circ T_{j_1,\tt})(x)\in B\left(\pi_{\tt}(\i),\left(\frac{h(n)}{ \lambda(\A)^{n}}\right)^{1/d}\right)\textrm{ for i.m. } n\in \N\right\}.$$ Which by \eqref{no preimage} will complete our proof. 
		
		Let us fix $x\in C$. If $x\in C$ then $x\in B$, therefore we can use the argument given in the proof of Theorem \ref{Main thm}.2 to show that there exists $\i_1\in \I^{*}$ such that $$T_{\i_1,\tt}(x)\in B\left(\pi_{\tt}(\i),\left(\frac{h(|\i_1|)}{\lambda(\A)^{|\i_1|}}\right)^{1/d}\right).$$By \eqref{Ball inclusion} we know that $T_{\i_1,\tt}(x)\in int(X_{\tt}).$ Combining this with the fact $x\in C$ and therefore not in $T_{\i_{1},\tt}^{-1}(X_{\tt}\setminus B)$, it follows that $T_{\i_1,\tt}(x)\in B$. Therefore there exists $\j_1$ and $y$ such that $$S_{\j_1,\tt}(y)=T_{\i_1,\tt}(x)\textrm{ and }y\in \bigcap_{0<c\leq 1}W(\mathcal{S}_{\tt},\pi_{\tt}(\i),c\cdot h).$$ Now let $\gamma>0$ be such that $$\frac{h(n+1)}{h(n)}\geq \gamma$$ for all $n\in \N;$ just as in the proof of Theorem \ref{Main thm}.2. If we let $$c=\frac{\gamma^{|\i_1|+|\j_1|}}{{\lambda(\A)^{|\i_1|+|\j_1|}}},$$ then it follows from the argument given in the proof of Theorem \ref{Main thm}.2 that if $$T_{\mathbf{k},\tt}(y)\in  B\left(\pi_{\tt}(\i),\left(\frac{c\cdot h(|\mathbf{k})|}{\lambda(\A)^{|\mathbf{k}|}}\right)^{1/d}\right)$$ then $$T_{\mathbf{k}\overline{\j_1} \i_1,\tt}(x)\in B\left(\pi_{\tt}(\i),\left(\frac{h(|\mathbf{k}\j_1 \i_1|)}{\lambda(\A)^{|\mathbf{k}\j_1 \i_1|}}\right)^{1/d}\right).$$
		We let $\i_2=\mathbf{k}\overline{\j_1}.$ Using \eqref{Ball inclusion} and the fact $x\in C$, we may conclude that $T_{\i_2 \i_1,\tt}(x)\in B.$ As such we can repeat the above argument to assert that there exists a word $\i_{3}\in \I^*$ such that $$T_{\i_3\i_2 \i_1,\tt}(x)\in B\left(\pi_{\tt}(\i),\left(\frac{h(|\i_3\i_2 \i_1|)}{\lambda(\A)^{|\i_3\i_2 \i_1|}}\right)^{1/d}\right)\textrm{ and }T_{\i_3\i_2 \i_1,\tt}(x)\in B.$$ It is clear that this process can be continued indefinitely, and as such we can define a sequence of words $(\i_{p})_{p=1}^{\N}$ such that for all $p\in \mathbb{N}$ we have $$T_{\i_p\ldots \i_1,\tt}(x)\in B\left(\pi_{\tt}(\i),\left(\frac{h(|\i_p\ldots \i_1|)}{\lambda(\A)^{|\i_p\ldots \i_1|}}\right)^{1/d}\right)\textrm{ and }T_{\i_p\ldots \i_1,\tt}(x)\in B.$$ Our results now follows by taking our desired sequence to be the infinite concatenation of the words $\{\i_p\}_{p=1}^{\infty}$, i.e. $\j=\i_1\i_2\i_3\ldots$. .
	\end{proof}
	
	\subsection{Proof of Theorem \ref{Main thm}.4}
	
	The proof of Theorem \ref{Main thm}.4 is similar to the proof of Theorem \ref{Main thm}.1. For this reason we only include an outline. The key technical result which allows us to recast our recurrence set statements into a framework that resembles that used to study shrinking targets sets is the following. 
	
	\begin{lemma}
		\label{recurrence to shrinking}
		%	Let $\S=\{S_i\}_{i\in \I}$ be an IFS such that $\|A_i\|<1/2$ for all $i\in \I$. Then there exists $c>0$ depending only on $\max_{i\in \I}\|A_i\|$, such that for any $\i\in \I^{*}$ if $x\in \pi(\overline{\i}^{\infty})+A_{\overline{\i}}(B(0,c r))$ then $T_{\i}(x)\in x+B(0,r)$. 
		Let $\S=\{S_i\}_{i\in \I}$ be an IFS and $\i\in \I^{*}.$ Then for any Borel set $E\subset\mathbb{R}^d$ we have $x\in \pi(\overline{\i}^{\infty})+\sum_{k=1}^{\infty}A_{\overline{\i}}^k(E)$ if and only if  $T_{\i}(x)\in x+E$. Moreover, if we assume that $\|A_i\|<1/2$ for all $i\in \I,$ then there exists $c,C>0$ depending only on $\max_{i\in \I}\|A_i\|$ such that 
		\begin{equation}\label{eq:detsum}
			\frac{|Det(A_{\overline{\i}})|}{C}\leq \left|Det\left(\sum_{k=1}^{\infty}A_{\overline{\i}}^k\right)\right|\leq C\left|Det(A_{\overline{\i}})\right|
		\end{equation} 
		and 
		\begin{equation}\label{eq:contained}
			A_{\overline{\i}}(B(0,cr))\subseteq \sum_{k=1}^{\infty}A_{\overline{\i}}^k(B(0,r))
		\end{equation}
		for all $\i\in \I^*$ and $r>0$.
	\end{lemma} 	
	\begin{proof}
		Let $\i=(i_1,\ldots,i_n)$. We begin by observing the following equivalences:
		\begin{align*}
			T_{i_1\ldots i_n}(x)-x\in E&\iff  (A_{i_1}^{-1}\circ \cdots A_{i_n}^{-1})(x)-\sum_{j=1}^{n}(A_{i_1}^{-1}\circ \cdots A_{i_j}^{-1})t_{i_j}-x\in E\\
			&\iff (A_{i_1}^{-1}\circ \cdots A_{i_n}^{-1})(x)-x\in \sum_{j=1}^{n}(A_{i_1}^{-1}\circ \cdots A_{i_j}^{-1})t_{i_j}+E\\
			&\iff (A_{i_1}^{-1}\circ \cdots A_{i_n}^{-1}-I)(x)\in \sum_{j=1}^{n}(A_{i_1}^{-1}\circ \cdots A_{i_j}^{-1})t_{i_j}+E.
		\end{align*}
		Now using that $(A_{i_1}^{-1}\circ \cdots A_{i_n}^{-1}-I)^{-1}=\sum_{k=1}^{\infty}(A_{i_n}\circ \cdots \circ A_{i_1})^k$ we observe
		\begin{align*}
			T_{i_1\ldots i_n}(x)-x\in E&\iff x\in \sum_{k=1}^{\infty}(A_{i_n}\circ \cdots \circ A_{i_1})^k\left(\sum_{j=1}^{n}(A_{i_1}^{-1}\circ \cdots A_{i_j}^{-1})t_{i_j}+E\right)\\
			&\iff x\in \sum_{j=1}^{n}\sum_{k=1}^{\infty}(A_{i_n}\circ \cdots \circ A_{i_1})^{k-1}(A_{i_n}\circ \cdots \circ A_{i_{j+1}})t_{i_j}+\sum_{k=1}^{\infty}(A_{i_n}\circ \cdots \circ A_{i_1})^k E\\
			&\iff x\in \pi(\overline{\i}^{\infty})+\sum_{k=1}^{\infty}A_{\overline{\i}}^k E.
		\end{align*} This completes the proof of the first claim in the statement. 
		
		We now focus on the second part of our lemma. We assume that $\|A_{i}\|<1/2$ for all $i\in \I$. Let $\lambda:=\max_{i\in \I}\|A_{i}\|$. By definition $\lambda\in(0,1/2)$. It is useful at this point to think of the linear map $x\to \sum_{k=1}^{\infty}A_{\overline{\i}}^kx$ as the composition of $x\to \sum_{k=1}^{\infty}A_{\overline{\i}}^{k-1}x$  with $x\to A_{\overline{\i}}x$. For any $x\in \mathbb{C}^d$ we have
		\begin{equation}
			\label{norm lower bound}
			\left\|\sum_{k=1}^{\infty}A_{\overline{\i}}^{k-1} x\right\|\geq \|x\|-\left\|\sum_{k=1}^{\infty}A_{\overline{\i}}^{k}x\right\|\geq \|x\|-\sum_{k=1}^{\infty}\lambda^k\|x\|\geq \|x\|\left(1-\frac{\lambda}{1-\lambda}\right)
		\end{equation}
		and
		\begin{equation}
			\label{norm upper bound}
			\left\|\sum_{k=1}^{\infty}A_{\overline{\i}}^{k-1} x\right\|\leq \|x\|+\left\|\sum_{k=1}^{\infty}A_{\overline{\i}}^{k}x\right\|\leq \|x\|+\sum_{k=1}^{\infty}\lambda^k\|x\|\leq \|x\|\left(1+\frac{\lambda}{1-\lambda}\right)
		\end{equation}
		Let $c=\left(1-\frac{\lambda}{1-\lambda}\right)$ and $c'=\left(1+\frac{\lambda}{1-\lambda}\right)$. Here $c>0$ since $\lambda\in(0,1/2)$. Equation \eqref{norm lower bound} implies that $$B(0,cr)\subseteq \sum_{k=1}^{\infty}A_{\overline{\i}}^{k-1}B(0,r)$$ for all $r>0$. This in turn implies \eqref{eq:contained}. 
		
		To see why \eqref{eq:detsum} is true, notice that  \eqref{norm lower bound} and \eqref{norm upper bound} imply that 
		$$c\|x\|\leq 	\left\|\sum_{k=1}^{\infty}A_{\overline{\i}}^{k-1} x\right\|\leq c'\|x\|$$ for all $x\in \mathbb{C}^d$. Therefore the absolute value of every eigenvalue of $\sum_{k=1}^{\infty}A_{\overline{\i}}^{k-1}$ is bounded above by $c'$ and below by $c$. Now using the fact that the determinant of $\sum_{k=1}^{\infty}A_{\overline{\i}}^{k-1}$ is the product of its eigenvalues, we assert that there exists $C>0$ depending only upon $\max_{i\in \I}\|A_{i}\|$ such that $$\frac{1}{C}\leq \left|Det\left(\sum_{k=1}^{\infty}A_{\overline{\i}}^{k-1}\right)\right|\leq C.$$ Using the fact that the determinant is multiplicative, we can now multiply through by $|Det (A_{\overline{\i}})|$ in the above and conclude \eqref{eq:detsum}. This completes our proof.
	\end{proof}
	Now given an IFS $\S$ satisfying $\|A_i\|<1/2$ for all $i\in \I$,  it follows from Lemma \ref{recurrence to shrinking} that there exists $c>0$ such that for any function $h:\N\to [0,\infty)$ we have $$\bigcap_{m=1}^{\infty}\bigcup_{n=m}^{\infty}\bigcup_{\i\in \I^{n}}\left(\pi(\i^{\infty})+A_{\i}\left(B\left(0,c\cdot \left(\frac{h(n)}{\lambda(\A)^n}\right)^{1/d}\right)\right)\right)\subset R(\S,h).$$ So to prove Theorem \ref{Main thm}.4, it is sufficient to prove a positive measure result for the sets on the left hand side of the above inclusion. The sets $$\bigcap_{m=1}^{\infty}\bigcup_{n=m}^{\infty}\bigcup_{\i\in \I^{n}}\left(\pi(\i^{\infty})+A_{\i}\left(B\left(0,c\cdot \left(\frac{h(n)}{\lambda(\A)^n}\right)^{1/d}\right)\right)\right)$$ are amenable to the same methods we used to prove Theorem \ref{Main thm}.1. In particular, suppose we are given $\{A_i\}$ a finite set of matrices satisfying the assumptions of Theorem \ref{Main thm}.1, then for each $\tt\in \mathbb{R}^{\#\I\cdot d}$, $\i\in \I^{*}$ and $s>0$ we let $$B_{\tt}(\i,s):=\pi_{\tt}(\i^{\infty})+A^n B\left(0,\frac{s}{\lambda(\A)^{n/d}}\right).$$ Lemma \ref{Transverality estimate} can be used in a similar way to prove the following analogue of Lemma \ref{Counting pairs}, the proof of which we omit.
	
	\begin{lemma}
		\label{Counting pairs recurrence}
		Let $\{A_i\}_{i\in \I}$ be a collection of matrices satisfying the assumptions of Theorem \ref{Main thm}. Then for any $R>0$, $n\in \N$, $s>0$ we have 
		$$	\int_{[-R,R]^{\# \I\cdot d}}\#\left\{\j,\k\in \I^n:\j\neq \k,\, B_{\tt}(\j,s)\cap B_{\tt}(\k,s)\neq \emptyset \right\}\, d\tt=\O_{R}(\# \I^n\cdot s^d).$$
	\end{lemma}
	Once equipped with Lemma \ref{Counting pairs recurrence}, the proof of Theorem \ref{Main thm}.4 follows the same argument as Theorem \ref{Main thm}.1.
	\section{Proof of Theorem \ref{exponential theorem}}
	\label{Section 6}
	Suppose one were given a sequence $(i_{n,m})\in \I^{\N\times \N}$ and a sequence of Borel sets $(E_n)$ satisfying the hypothesis of Theorem \ref{exponential theorem}.1, the key to proving Theorem \ref{exponential theorem}.1 is to understand for small $s>0$ and for arbitrary $n\in \N$ the Lebesgue measure of the set 
	$$\bigcup_{\i\in \I^n}S_{\i,\tt}\left(\pi_{\tt}((i_{n,m})_{m=1}^{\infty})+s\cdot E_n\right)$$ for a typical $\tt$. To obtain meaningful bounds we need to understand the measure of the intersection of two sets in this union for a typical $\tt$. This is the content of Lemma \ref{typical measure}. Using this lemma we can then obtain for arbitrary $R>0$ and $n\in \N$ a useful expression for $$\int_{[-R,R]^{\# \I\cdot d}}\bigcup_{\i\in \I^n}S_{\i,\tt}\left(\pi_{\tt}((i_{n,m})_{m=1}^{\infty})+E_n\right)\, d\tt.$$ This is the content of Proposition \ref{intersection prop}. This latter expression is what allows us to prove our result.
	\begin{lemma}
		\label{typical measure}
		Let $\{A_i\}_{i\in \I}$ be a finite set of matrices satisfying $\|A_{i}\|<1/2$ for all $i\in \I$. Let $(i_{n,m})_{n,m}\in \I^{\N\times \N}$ and $(E_n)$ be a sequence of Borel sets satisfying the assumptions of Theorem \ref{exponential theorem}.1. Then for any $n\in \N$ and $s>0$, if $\i,\j\in \I^n$ are distinct words such that $|\i\wedge \j|=k$, we have  
		\begin{align*}
			&\int_{[-R,R]^{\# \I\cdot d}}\L_{d}\left(S_{\i,\tt}\left(\pi_{\tt}((i_{n,m})_{m=1}^{\infty})+s\cdot E_n\right)\cap S_{\j,\tt}\left(\pi_{\tt}((i_{n,m})_{m=1}^{\infty})+s\cdot E_n\right)\right)\, d\tt\\
			\ll_{R} & \frac{|Det(A_{i_{k}}\ldots A_{i_n})| \cdot |Det(A_{\j})|s^{2d}}{\lambda(\A)^{2n}}.
		\end{align*}
	\end{lemma}
	\begin{proof}
		Let us begin our proof by fixing $(i_{n,m})_{n,m}$ and $(E_n)$ satisfying the assumptions of Theorem \ref{exponential theorem}.1. Fix $n\in\mathbb{N}$ and $s>0$. For any $\j\in \I^n$ we have  $$S_{\j,\tt}\left(\pi_{\tt}((i_{n,m})_{m=1}^{\infty})+s\cdot E_n\right)=\pi_{\tt}(\j(i_{n,m})_{m=1}^{\infty})+A_{\j}(s\cdot E_n).$$ and 
		\begin{equation*}
					\L_{d}\left(A_{\j}(s\cdot E_n)\right)=\frac{|Det(A_{\j})|s^d}{\lambda(\A)^{n}}.
		\end{equation*}
		Let $r_n>0$ and $C>0$ be as in the statement of Theorem \ref{exponential theorem}. So in particular we have 
		\begin{equation}
			\label{dilated measure}
			\L_{d}\left([-r_n,r_n]^d\pm E_n\right)\leq C\L_{d}(E_n).
		\end{equation} We let $r_n^*>0$ be a sufficiently small real number\footnote{We can simply take $r_n^*$ to be any number sufficiently small so that $[-r_n^*,r_n^*]^d\subset s\cdot A_{\i}[-r_n,r_n]^d$ for all $\i\in \I^n.$} satisfying $$ [-r_n^*,r_n^*]^d- A_{\i}(s\cdot E_n)\subseteq A_{\i}\left(s\cdot\left([-r_n,r_n]^d- E_n\right)\right)$$ and 
		$$ [-r_n^*,r_n^*]^d+ A_{\i}(s\cdot E_n)\subseteq A_{\i}\left(s\cdot\left([-r_n,r_n]^d+ E_n\right)\right)$$for all $\i\in \I^{n}$.
		
		Let us now fix $\i,\j\in \I^n$ such that $|\i\wedge \j|=k$. For $\p=(p_1, \dots, p_d)\in \mathbb Z^d$, denote
		\[
		G(\p)=\prod_{i=1}^{d}[p_ir_n^*,(p_i+1)r_n^*]
		\]
		and then let 
		$$V_{\j}=\left\{\p\in\mathbb{Z}^d:G(\p)\cap A_{\j}(s\cdot E_n)\neq\emptyset\right\}.$$  It follows from the properties of $r_n^*$ listed above, the fact $\L_{d}(E_n)=\lambda(\A)^{-n}$, and \eqref{dilated measure} that  
		\begin{equation}
			\label{inclusion 1}
			 \L_{d}\left(A_{\j}\left(s\cdot \left([-r_n,r_n]^d+E_n\right)\right)\right)\ll \frac{Det(A_{\j})s^d}{\lambda(\A)^n}
		\end{equation}and
		\begin{equation*}
			\label{inclusion 2}
			\bigcup_{\p\in V_{\j}}G(\p)\subseteq [-r_n^*,r_n^*]^d+A_{\j}(s\cdot E_n)\subseteq A_{\j}\left(s\cdot\left([-r_n,r_n]^d+ E_n\right)\right).
		\end{equation*}
		Using the above inclusions and \eqref{inclusion 1} we have
		\begin{equation}
			\label{measure bound}
			\sum_{\p \in V_{\j}}(r_{n}^*)^d\ll \frac{Det(A_\j)s^d}{\lambda(\A)^n}.
		\end{equation}
		Now using the definition of $V_{\j}$ we have 
		\begin{align}
			\label{V_j measure bound}
			&\int_{[-R,R]^{\# \I\cdot d}}\L_{d}\left(S_{\i,\tt}\left(\pi_{\tt}((i_{n,m})_{m=1}^{\infty})+s\cdot E_n\right)\cap S_{\j,\tt}\left(\pi_{\tt}((i_{n,m})_{m=1}^{\infty})+s\cdot E_n\right)\right)\, d\tt\nonumber\\
			\leq &\sum_{\p\in V_j}\int_{[-R,R]^{\# \I\cdot d}}\L_{d}\left(S_{\i,\tt}\left(\pi_{\tt}((i_{n,m})_{m=1}^{\infty})+s\cdot E_n\right)\cap \left(\pi_{\tt}(\j(i_{n,m})_{m=1}^{\infty})+G(\p)\right)\right)\, d\tt\nonumber\\
			\leq &\sum_{\p\in V_\j}\L_{\# \I\cdot d}\left(\tt\in [-R,R]^{\# \I\cdot d}:S_{\i,\tt}\left(\pi_{\tt}((i_{n,m})_{m=1}^{\infty})+s\cdot E_n\right)\cap \left(\pi_{\tt}(\j(i_{n,m})_{m=1}^{\infty})+G(\p)\right)\neq\emptyset\right)\nonumber\\
			&\times (r_n^*)^d.
			%				\leq & \sum_{\p\in V_\j}(r_n^*)^d\L_{\# \I\cdot d}\left(\tt\in [-R,R]^{\# \I\cdot d}:S_{\i,\tt}\left(\pi_{\tt}((i_{n,m})_{m=1}^{\infty})+s\cdot E_n\right)\cap \left(\pi_{\tt}(\j(i_{n,m})_{m=1}^{\infty}+\prod_{i=1}^{d}[p_ir_n^*,(p_i+1)r_n^*)\neq\emptyset\right).
		\end{align}Now notice that if $\tt$ is such that for some $\p\in V_\j$ we have 
		$$S_{\i,\tt}\left(\pi_{\tt}((i_{n,m})_{m=1}^{\infty})
		+s\cdot E_n\right)\cap \left(\pi_{\tt}(\j(i_{n,m})_{m=1}
		^{\infty})+G(\p)
		\right)\neq\emptyset,$$ then by our choice of $r_n^*$
		\begin{align*}
			\pi_{\tt}(\i(i_{n,m})_{m=1}^{\infty})-\pi_{\tt}(\j(i_{n,m})_{m=1}^{\infty})&\in G(\p)-A_{\i}(s\cdot E_n)\\
			&\subseteq r_n^*\cdot\p+[0,r_n^*]^d-A_{\i}(s\cdot E_n)\\
			&\subseteq r_n^*\cdot\p +A_{\i}\left(s\cdot\left([-r_n,r_n]^d- E_n\right)\right).
		\end{align*} Using this observation, Lemma \ref{Transverality estimate} and \eqref{dilated measure}, we have the following for each $\p\in V_{\j}$
		\begin{align*}
			&\L_{\# \I\cdot d}\left(\tt\in [-R,R]^{\# \I\cdot d}:S_{\i,\tt}\left(\pi_{\tt}((i_{n,m})_{m=1}^{\infty})+s\cdot E_n\right)\cap \left(\pi_{\tt}(\j(i_{n,m})_{m=1}^{\infty}+G(\p)\right)\neq\emptyset\right)\\
			\leq &\L_{\# \I\cdot d}\left(\tt\in [-R,R]^{\# \I\cdot d}:\pi_{\tt}(\i(i_{n,m})_{m=1}^{\infty})-\pi_{\tt}(\j(i_{n,m})_{m=1}^{\infty})\in r_n^*\cdot\p+A_{\i}\left(s\cdot\left([-r_n,r_n]^d- E_n\right)\right) \right)\\
			\ll_{R}& |Det(A_{i_1,\ldots,i_{k-1}})^{-1}|\L_{d}\left(A_{\i}\left(s\cdot\left([-r_n,r_n]^d- E_n\right)\right)\right)\\
			\ll_{R}& |Det(A_{i_1,\ldots,i_{k-1}})^{-1}|\frac{|Det(A_{\i})|s^d}{\lambda(\A)^n}\\
			=&\frac{|Det(A_{i_{k}}\ldots A_{i_n})|s^d}{\lambda(\A)^n}.
		\end{align*}
		Using this bound together with \eqref{measure bound} and \eqref{V_j measure bound}, we then have 
		\begin{align*}
			&\int_{[-R,R]^{\# \I\cdot d}}\L_{d}\left(S_{\i,\tt}\left(\pi_{\tt}((i_{n,m})_{m=1}^{\infty})+s\cdot E_n\right)\cap S_{\j,\tt}\left(\pi_{\tt}((i_{n,m})_{m=1}^{\infty})+s\cdot E_n\right)\right)\, d\tt\\
			\leq &\sum_{\p\in V_j}\L_{\# \I\cdot d}\left(\tt\in [-R,R]^{\# \I\cdot d}:S_{\i,\tt}\left(\pi_{\tt}((i_{n,m})_{m=1}^{\infty})+s\cdot E_n\right)\cap \left(\pi_{\tt}(\j(i_{n,m})_{m=1}^{\infty}+G(\p)\right)\neq\emptyset\right)\\
			&\times (r_n^*)^d\\
			% \leq &\sum_{V_\j}r_n^d\L_{\# \I\cdot d}\left(\tt\in [-R,R]^{\# \I\cdot d}:S_{\i,\tt}\left(\pi_{\tt}((i_{n,m})_{m=1}^{\infty})+s\cdot E_n\right)\cap \left(\pi_{\tt}(\j(i_{n,m})_{m=1}^{\infty}+\prod_{i=1}^{d}[p_ir_n,(p_i+1)r_n]\right)\neq\emptyset\right)\\
			&\ll_{R}\sum_{V_\j}(r_{n}^*)^d\frac{|Det(A_{i_{k}}\ldots A_{i_n})|s^d}{\lambda(\A)^n}\\
			&\ll_{R} \frac{|Det(A_{i_{k}}\ldots A_{i_n})|\cdot |Det(A_{\j})|s^{2d}}{\lambda(\A)^{2n}}.
		\end{align*}This completes the proof.

	\end{proof}

	\begin{prop}
		\label{intersection prop}
		Let $\{A_i\}$ be a finite set of matrices satisfying the assumptions of Theorem \ref{exponential theorem}. Let $(i_{n,m})_{n,m}\in \I^{\N\times \N}$ and $(E_n)$ be a sequence of Borel sets satisfying the assumptions of Theorem \ref{exponential theorem}.1. Then for any $n\in \N$ we have
		$$\int_{[-R,R]^{\# \I\cdot d}}\mathcal{L}_{d}\left(\bigcup_{\i\in \I^n}S_{\i,\tt}\left(\pi_{\tt}((i_{n,m})_{m=1}^{\infty})+s\cdot E_n\right)\right)\, d\tt=s^d\mathcal{L}_{d}([-R,R]^{\# \I \cdot d})+\mathcal{O}_{R}(s^{2d})$$
		and $$\mathcal{L}_{d}\left(\bigcup_{\i\in \I^n}S_{\i,\tt}\left(\pi_{\tt}((i_{n,m})_{m=1}^{\infty})+s\cdot E_n\right)\right)\leq s^{d}$$
		for all $\tt\in \R^{\#\I\cdot d}$. 
	\end{prop}
	
	\begin{proof}
		The second statement follows from the following computation:
		\begin{align*}
			\mathcal{L}_{d}\left(\bigcup_{\i\in \I^n}S_{\i,\tt}\left(\pi_{\tt}((i_{n,m})_{m=1}^{\infty})+s\cdot E_n\right)\right)&\leq \sum_{\i\in \I^n}\mathcal{L}_{d}\left(S_{\i,\tt}\left(\pi_{\tt}((i_{n,m})_{m=1}^{\infty})+s\cdot E_n\right)\right)\\
			&=\sum_{\i\in \I^n}\frac{s^d}{\lambda(\A)^n}|Det(A_\i)|\\
			&=\frac{s^d}{\lambda(\A)^n}\left(\sum_{i\in \I}|Det(A_i)|\right)^n\\
			&=s^d.
		\end{align*}
		We now move on to the first statement. We start with the following inequality which is an immediate application of Lemma \ref{Bonferroni}
		\begin{align*}
			&\int_{[-R,R]^{\# \I\cdot d}}\mathcal{L}_d\left(\bigcup_{\i\in I^n}S_{\i,\tt}\left(\pi_{\tt}((i_{n,m})_{m=1}^{\infty})+s\cdot E_n\right)\right)\, d\tt\\
			\geq &\int_{[-R,R]^{\#\I\cdot d}}\sum_{\i\in \I^n}\mathcal{L}_{d}\left(S_{\i,\tt}\left(\pi_{\tt}((i_{n,m})_{m=1}^{\infty})+s\cdot E_n\right)\right)\, d\tt\\
			&-\int_{[-R,R]^{\# \I\cdot d}}\sum_{\i,\j\in \I^n:\i\neq \j}\L_{d}\left(S_{\i,\tt}\left(\pi_{\tt}((i_{n,m})_{m=1}^{\infty})+s\cdot E_n\right)\cap S_{\j,\tt}\left(\pi_{\tt}((i_{n,m})_{m=1}^{\infty})+s\cdot E_n\right)\right)\, d\tt.
		\end{align*}It follows from the argument given above in the derivation of statement $2$ that $$\sum_{\i\in \I^n}\mathcal{L}_{d}\left(S_{\i,\tt}\left(\pi_{\tt}((i_{n,m})_{m=1}^{\infty})+s\cdot E_n\right)\right)=s^d$$ for all $\tt\in \R^{\# \I \cdot d}$. Therefore we have 
		$$\int_{[-R,R]^{\# \I\cdot d}}\sum_{\i\in \I^n}\mathcal{L}_{d}\left(S_{\i,\tt}\left(\pi_{\tt}((i_{n,m})_{m=1}^{\infty})+s\cdot E_n\right)\right)=s^d\L_{\# \I\cdot d}([-R,R]^d).$$ It remains to bound the second term. Applying Lemma \ref{typical measure} we observe the following:
		\begin{align*}
			&\int_{[-R,R]^{\# \I\cdot d}}\sum_{\i,\j\in \I^n:\i\neq \j}\L_{d}\left(S_{\i,\tt}\left(\pi_{\tt}((i_{n,m})_{m=1}^{\infty})+s\cdot E_n\right)\cap S_{\j,\tt}\left(\pi_{\tt}((i_{n,m})_{m=1}^{\infty})+s\cdot E_n\right)\right)\, d\tt\\
			=&\int_{[-R,R]^{\# \I\cdot d}}\sum_{\j\in \I^n}\sum_{k=1}^n\sum_{\i\in \I^n:|\i\wedge \j|=k}\L_{d}\left(S_{\i,\tt}\left(\pi_{\tt}((i_{n,m})_{m=1}^{\infty})+s\cdot E_n\right)\cap S_{\j,\tt}\left(\pi_{\tt}((i_{n,m})_{m=1}^{\infty})+s\cdot E_n\right)\right)\, d\tt\\
			\ll_{R} & \sum_{\j\in \I^n}\sum_{k=1}^n\sum_{\i\in \I^n:|\i\wedge \j|=k}\frac{|Det(A_{i_{k}}\ldots A_{i_n})|\cdot |Det(A_{\j})|s^{2d}}{\lambda(\A)^{2n}}\\
			=&s^{2d}\sum_{\j\in \I^{n}}\frac{|Det(A_{\j})|}{\lambda(\A)^n}\sum_{k=1}^{n}\sum_{\i\in \I^{n-k+1}}\frac{|Det(A_{\i})|}{\lambda(\A)^n}\\
			=&s^{2d}\sum_{\j\in \I^{n}}\frac{|Det(A_{\j})|}{\lambda(\A)^n}\sum_{k=1}^{n}\frac{\lambda(\A)^{n-k+1}}{\lambda(\A)^n}\\
			\ll&s^{2d}\sum_{\j\in \I^{n}}\frac{|Det(A_{\j})|}{\lambda(\A)^n}.
		\end{align*}
		In the penultimate line we used that $\lambda(\A)>1$ so $\sum_{k=1}^{n}\frac{\lambda(\A)^{n-k+1}}{\lambda(\A)^n}=\O(1).$ Continuing from here, 	
		\begin{align*}							
			s^{2d}\sum_{\j\in \I^{n}}\frac{|Det(A_{\j})|}{\lambda(\A)^n}
			=s^{2d}\frac{\left(\sum_{i\in \I}|Det(A_i)|\right)^n}{\lambda(\A)^n}
			=s^{2d}\frac{\lambda(\A)^n}{\lambda(\A)^n}=s^{2d}.
		\end{align*}
		This completes our proof.
		
	\end{proof}
	Equipped with Proposition \ref{intersection prop} we can now prove Theorem \ref{exponential theorem}.1.
	\begin{proof}[Proof of Theorem \ref{exponential theorem}.1]
		Let $\{A_i\}$ be a finite set of matrices satisfying the assumptions of Theorem \ref{exponential theorem}. Also let $(i_{n,m})_{(n,m)}\in \I^{\N\times \N}$ and $(E_n)$ be as in the statement of this theorem. Now let $R>0$ and $\epsilon>0$ be arbitrary. To complete our proof, it suffices to show that the conclusions of Theorem \ref{exponential theorem}.1 hold for Lebesgue almost every $\tt\in [-R,R]^{\# \I\cdot d}$ outside of a set of measure $\epsilon$. For each $n\in\mathbb{N}$ and $s>0$ we consider the set $$A_{n}(s):=\left\{\tt\in [-R,R]^{\# \I\cdot d}: \mathcal{L}_{d}\left(\bigcup_{\i\in \I^n}S_{\i,\tt}\left(\pi_{\tt}((i_{n,m})_{m=1}^{\infty})+s\cdot E_n\right)\right)\geq s^d(1-\epsilon)\right\}.$$
		Applying statement $2$ from Proposition \ref{intersection prop} we have 
		\begin{align*}
			&\int_{[-R,R]^{\#\I\cdot d}} \mathcal{L}\left(\bigcup_{\i\in \I^n}S_{\i,\tt}\left(\pi_{\tt}((i_{n,m})_{m=1}^{\infty})+s\cdot E_n\right)\right)\, d\tt\\
			\leq & \mathcal{L}_{{\# \I\cdot d}}(A_{n}(s))s^d+\mathcal{L}_{{\# \I\cdot d}}([-R,R]^{\# \I\cdot d}\setminus A_{n}(s)s^d(1-\epsilon).
		\end{align*} Now applying statement $1$ from Proposition \ref{intersection prop} we have  
		\begin{align*}
			\mathcal{L}_{{\# \I\cdot d}}(A_{n}(s))s^d+\mathcal{L}_{{\# \I\cdot d}}([-R,R]^{\# \I\cdot d}\setminus A_{n}(s))s^d(1-\epsilon)\geq s^{d}\mathcal{L}_{{\# \I\cdot d}}([-R,R]^{\# \I\cdot d})+\mathcal{O}_{R}(s^{2d}).
		\end{align*}
		Cancelling terms from either side, this inequality yields
		$$\mathcal{L}_{{\# \I\cdot d}}([-R,R]^{\# \I\cdot d}\setminus A_{n}(s))s^d\epsilon=\mathcal{O}_{R}(s^{2d}).$$ It follows from this equation that we can choose $s\in(0,1)$ sufficiently small in a manner that depends upon $\epsilon$ and $R$ such that $$\mathcal{L}_{{\# \I\cdot d}}([-R,R]^{\# \I\cdot d}\setminus A_{n}(s))<\epsilon,$$ or equivalently $$\mathcal{L}_{{\# \I\cdot d}}(A_{n}(s))>\mathcal{L}_{\# \I\cdot d}([-R,R]^{\#\I\cdot d})-\epsilon.$$ In what follows we will assume that we have chosen such an $s$ and we will denote it by $s^*$.
		
		Let $$A_{\infty}(s^*)=\left\{\tt\in [-R,R]^{\# \I\cdot d}:\tt \in A_{n}(s^*)\textrm{ for i.m. }n\in \N \right\}.$$ By the continuity of the Lebesgue measure from above, we have $$\mathcal{L}_{\# \I\cdot d}(A_{\infty}(s^*))=\mathcal{L}_{\# \I\cdot d}\left(\bigcap_{m=1}^{\infty}\bigcup_{n=m}^{\infty}A_{n}(s^*)\right)=\lim_{m\to\infty}\mathcal{L}_{\# \I\cdot d}\left(\bigcup_{n=m}^{\infty}A_{n}(s^*)\right)\geq \mathcal{L}_{\# \I\cdot d}([-R,R]^{\#\I\cdot d})-\epsilon.$$ Using the above inequality, we see that to prove our result it suffices to show that the desired conclusions hold for any $\tt\in A_{\infty}(s^*)$. This we do below.

		Now using the assumption $s\cdot E_n\subset E_n$ for all $s\in(0,1),$ we have the following for any $\tt \in A_{\infty}(s^*)$ 
		\begin{align*}
			\limsup_{n\to\infty}\L_{d}\left(\bigcup_{\i\in \I^n}S_{\i,\tt}\left(\pi_{\tt}((i_{n,m})_{m=1}^{\infty})+E_n\right)\right)&\geq 	\limsup_{n\to\infty}\L_{d}\left(\bigcup_{\i\in \I^n}S_{\i,\tt}\left(\pi_{\tt}((i_{n,m})_{m=1}^{\infty})+s^*\cdot E_n\right)\right)\\
			&\geq (s^*)^d(1-\epsilon)>0 .
		\end{align*} As such, the first conclusion of Theorem \ref{exponential theorem}.1 follows once we observe that $$\bigcup_{\i\in \I^n}S_{\i,\tt}(\pi_{\tt}((i_{n,m})_{m=1}^{\infty})+E_n)$$ coincides with
		$$\left\{x:\exists (i_1,\ldots,i_n)\in \I^n \textrm{ such that } (T_{i_n,\tt}\circ \cdots \circ T_{i_1,\tt})(x)\in \pi_{\tt}((i_{n,m})_{m=1}^{\infty})+E_n\right\}$$ for any $\tt\in \R^{\# \I\cdot d}$.

		We will now prove that the second conclusion of Theorem \ref{exponential theorem}.1 holds for any $\tt\in A_{\infty}(s^*)$. We want to use continuity of the Lebesgue measure again, however to do this we must know that 
		\begin{equation}
			\label{finite m}
			\L_{d}\left(\bigcup_{n=m}^{\infty}\bigcup_{\i\in\I^n}S_{\i,\tt}\left(\pi_{\tt}((i_{n,m})_{m=1}^{\infty})+ E_n\right)\right)<\infty
		\end{equation}
		for some $m\in \mathbb{N}$. At this point we use our assumption that there exists $Q>0$ such that $E_n\subset[-Q,Q]^d$ for all $n\in \N$. This implies that  $$\bigcup_{n=1}^{\infty}\bigcup_{\i\in\I^n}S_{\i,\tt}\left(\pi_{\tt}((i_{n,m})_{m=1}^{\infty})+ E_n\right)$$ belongs to some bounded domain in $\mathbb{R}^d$ and so \eqref{finite m} holds for all $m$\footnote{This is the only point in our proof where we use the existence of $Q>0$ satisfying $E_n\subset[-Q,Q]^d$ for all $n\in \N$.}.
		
		Now freely using the continuity of the Lebesgue measure from above, we see that the following holds for any $\tt\in A_{\infty}(s^*)$ 
		\begin{align*}
			W(\mathcal{S}_{\tt},(\pi_{\tt}((i_{n,m})_{m=1}^{\infty})_{n=1}^{\infty},(E_n))	=&\mathcal{L}_{d}\left(\bigcap_{m=1}^{\infty}\bigcup_{n=m}^{\infty}\bigcup_{\i\in\I^n}S_{\i,\tt}\left(\pi_{\tt}((i_{n,m})_{m=1}^{\infty})+E_n\right)\right)\\
			=&\lim_{m\to\infty}\mathcal{L}_{d}\left(\bigcup_{n=m}^{\infty}\bigcup_{\i\in \I^n}S_{\i,\tt}\left(\pi_{\tt}((i_{n,m})_{m=1}^{\infty})+ E_n\right)\right)\\
			\geq& (s^*)^d(1-\epsilon)>0.				
		\end{align*} In the final line we used 
		$$	\limsup_{n\to\infty}\L_{d}\left(\bigcup_{\i\in \I^n}S_{\i,\tt}\left(\pi_{\tt}((i_{n,m})_{m=1}^{\infty})+E_n\right)\right)\geq (s^*)^d(1-\epsilon)$$ for all $\tt\in A_{\infty}(s^*)$. This completes our proof.

	\end{proof}
	
	\begin{remark}
		In the above we showed that for any $\epsilon$, there is a set $A_\infty(s^*)$ with a complement of measure at most $\epsilon$, such that for all $\tt\in A_\infty(s^*)$, the measure of the set $W(\mathcal{S}_{\tt},(\pi_{\tt}((i_{n,m})_{m=1}^{\infty})_{n=1}^{\infty},(E_n))$ is bounded below by a positive constant. It should be noted, however, that this constant depends on $A_\infty(s^*)$ and in particular, the larger we insist the measure of $A_\infty(s^*)$, the smaller the constant. In particular, this result is far away from a $0$-full measure type result.  
	\end{remark}
	\begin{remark}
		In Theorem \ref{exponential theorem} we only consider sequences of sets satisfying $\L_{d}(E_n)=\lambda(\A)^{-n}$ for all $n$. It is natural to ask whether one can prove a positive measure result for sequences of sets satisfying $\L_{d}(E_n)=h(n)\lambda(\A)^{-n}$ for all $n$ where $h$ satisfies $\sum_{n=1}^{\infty}h(n)=\infty,$ like in Theorem \ref{Main thm}. The issue is that we are unable to extract from Proposition \ref{intersection prop} the existence of a large well-separated set. In our proof we apply this proposition to show that $$\mathcal{L}_{d}\left(\bigcup_{\i\in \I^n}S_{\i,\tt}\left(\pi_{\tt}((i_{n,m})_{m=1}^{\infty})+s\cdot E_n\right)\right)$$ is roughly $s^{d}$ for a large set of $\tt$. However, it is possible that this is true whilst every element in this union intersects another one of the sets.
	\end{remark}
	\subsection{Proof of Theorem \ref{exponential theorem}.2}
	
	\begin{proof}[Proof of Theorem \ref{exponential theorem}.2]
		Let $\{A_i\}$ be a finite set of matrices satisfying the assumptions of Theorem \ref{exponential theorem} and $\i\in \I^{\N}$. Using Theorem \ref{exponential theorem}.1 for $$(E_n)_{n=1}^{\infty}=\left(B\left(0,\frac{1}{(\L_{d}(B(0,1))\cdot\lambda(\A)^{n})^{1/d}}\right)\right)_{n=1}^{\infty}$$ together with our assumption that $X_{\tt}$ is differentiation regular for Lebesgue almost every $\tt\in \R^{\# \I\cdot d}$, we know that for Lebesgue almost every $\tt\in \R^{\# \I\cdot d}$ the set $W(S_{\tt},\pi_{\tt}(\i),\L_{d}(B(0,1))^{-1/d})$ has positive Lebesgue measure and $X_{\tt}$ is differentiation regular. In what follows we fix a $\tt$ satisfying these two properties. Applying Lemma \ref{Full measure images}, we know that Lebesgue almost every $x\in X_{\tt}$ belongs to $$\bigcup_{\i\in \I^*}S_{\i,\tt}(W(S_{\tt},\pi_{\tt}(\i),\L_{d}(B(0,1))^{-1/d})).$$ To complete our proof, we will now show that if $$x\in \bigcup_{\i\in \I^*}S_{\i,\tt}(W(S_{\tt},\pi_{\tt}(\i),\L_{d}(B(0,1))^{-1/d}))$$ then $$x\in \bigcup_{C>0}W(S_{\tt},\pi_{\tt}(\i),C).$$ Let $y\in W(S_{\tt},\pi_{\tt}(\i),\L_{d}(B(0,1))^{-1/d})$ and $\j\in \I^*$ be such that $S_{\j,\tt}(y)=x$. If $\k\in \I^{*}$ is such that \begin{equation}
			\label{k inclusion}
			T_{\k,\tt}(y)\in B\left(\pi_{\tt}(\i),\frac{1}{(\L_{d}(B(0,1))\lambda(\A)^{|\k|})^{1/d}}\right)
		\end{equation}then $$(T_{\k,\tt}\circ T_{\overline{\j},\tt})(x)\in  B\left(\pi_{\tt}(\i),\frac{1}{(\L_{d}(B(0,1))\lambda(\A)^{|\k|})^{1/d}}\right).$$ Taking $C=\lambda(\A)^{|\j|/d}\cdot \L_{d}(B(0,1))^{-1/d},$ we see that the above implies $$(T_{\k}\circ T_{\overline{\j},\tt})(x)\in  B\left(\pi_{\tt}(\i),\frac{C}{\lambda(\A)^{(|\k|+|\j|)/d}}\right).$$ Because by definition there exists infinitely many $\k\in \I^*$ such that \eqref{k inclusion} holds, it follows that $x\in W(S_{\tt},\pi_{\tt}(\i),C)$ and our result follows.
	\end{proof}	
	
	\subsection{Proof of Theorem \ref{exponential theorem}.3}
	The proof of Theorem \ref{exponential theorem}.3 is similar to the proof of Theorem \ref{exponential theorem}.1. As such we only include an outline. The first step is to use Lemma \ref{recurrence to shrinking}. This lemma allows us to assert that for each $\tt\in \R^{\# \I\cdot d}$ we have $$\bigcap_{m=1}^{\infty}\bigcup_{n=m}^{\infty}\bigcup_{\i\in\I^{\infty}}\left(\pi_{\tt}(\i^{\infty})+\sum_{k=1}^{\infty}A_{\i}^k(E_n)\right)\subset R(\S_{\tt},(E_n))$$ for any sequence of Borel sets $(E_n)_{n=1}^{\infty}.$ Now using Lemma \ref{Transverality estimate} and Lemma \ref{recurrence to shrinking} we can prove the following analogue of Lemma \ref{typical measure}.
	\begin{lemma}
		\label{typical measure analogue}
		Let $\{A_i\}_{i\in \I}$ be a finite set of matrices satisfying $\|A_{i}\|<1/2$ for all $i\in \I$. Let $(E_n)$ be a sequence of Borel sets satisfying the assumptions of Theorem \ref{exponential theorem}.3. Then for any $n\in \N$ and $s>0$, if $\i,\j\in \I^n$ are distinct words such that $|\i\wedge \j|=k$, we have  
		\begin{align*}
			&\int_{[-R,R]^{\# \I\cdot d}}\L_{d}\left(\left(\pi_{\tt}(\i^{\infty})+\sum_{k=1}^{\infty}A_{\i}^k(s\cdot E_n)\right)\cap \left(\pi_{\tt}(\j^{\infty})+\sum_{k=1}^{\infty}A_{\j}^k(s\cdot E_n)\right)\right)\, d\tt\\
			\ll_{R} & \frac{|Det(A_{i_{k}}\ldots A_{i_n})|\cdot |Det(A_{\j})|s^{2d}}{\lambda(\A)^{2n}}.
		\end{align*}
	\end{lemma}
	Combining Lemma \ref{Bonferroni} and Lemma \ref{typical measure analogue} allows us to prove the following statement. 
	\begin{prop}
		\label{recurrence intersection prop}
		Let $\{A_i\}_{i\in \I}$ be a finite set of matrices satisfying $\|A_{i}\|<1/2$ for all $i\in \I$. Let $(E_n)$ be a sequence of Borel sets satisfying the assumptions of Theorem \ref{exponential theorem}.3. Then for any $n\in \N$ and $R>0$ we have
		$$\int_{[-R,R]^{\# \I\cdot d}} \L_{d}\left(\bigcup_{\i\in \I^{n}}\left(\pi_{\tt}(\i^{\infty})+\sum_{k=1}^{\infty}A_{\i}^k(s\cdot E_n)\right)\right)=\frac{s^{d}\cdot\sum_{\i\in \I^n} |Det(\sum_{k=1}^{\infty}A_{\i}^k)|\L_{d}([-R,R]^{\#\I\cdot d})}{\lambda(\A)^n}+\O_{R}(s^{2d})$$ and $$\L_{d}\left(\bigcup_{\i\in \I^{n}}\left(\pi_{\tt}(\i^{\infty})+\sum_{k=1}^{\infty}A_{\i}^k(s\cdot E_n)\right)\right)\leq \frac{s^{d}\cdot \sum_{\i\in \I^n}|Det(\sum_{k=1}^{\infty}A_{\i}^k)|}{\lambda(\A)^n}$$
		for all $\tt\in \R^{\#\I\cdot d}.$
	\end{prop}
	Once we are equipped with Proposition \ref{recurrence intersection prop}, the proof of Theorem \ref{exponential theorem}.3 follows the same argument as Theorem \ref{exponential theorem}.1. The choice of $s^*$ in the proof is a priori dependant on the term
	\[
	\frac{\sum_{\i\in \I^n} Det(\sum_{k=1}^{\infty}A_{\i}^k)}{\lambda(\A)^n}
	\]
	appearing in Lemma \ref{recurrence intersection prop}. However, by \eqref{eq:detsum} proved above, this number is essentially a constant with respect to $n$, and so does not introduce extra difficulty when compared to the proof of Theorem \ref{exponential theorem}.1. 
	\section{Proof of Theorem \ref{thm:dim}}
	\label{Section 7}
	
	To prove Theorem \ref{thm:dim}, we will apply the Mass Transference Principle of Wang and Wu \cite[Theorem 3.1]{WanWu}. Rather than just consider iterated function systems, they find lower bounds for the Hausdorff dimension of limsup sets defined by a general system of rectangles of side lengths $\rho_N^{a_i+t_i}$. Loosely speaking, Wang and Wu show that when one has appropriate measure-theoretic knowledge about the limsup set for rectangles with side lengths $\rho_N^{a_i},$ then this can be used to obtain a lower bound for the Hausdorff dimension of the shrunk limsup set defined using the side lengths $\rho_N^{a_i+t_i}.$ Here $\rho_N$ is a sequence shrinking to $0$ with $N$, and $a_1,\dots, a_d, t_1, \dots, t_d\ge 0$ determine the shape of the rectangles.

	Our first step towards proving Theorem \ref{thm:dim} is to establish that for a Lebesgue typical $\tt$ a suitable local ubiquity property is satisfied. See \cite[Definition 3]{WanWu} for the definition of local ubiquity.
	
	\begin{prop}
		\label{ubiquity prop}
		Let $\{A_i\}_{i\in \I}$ be a finite set of matrices satisfying the assumptions of Theorem \ref{thm:dim} and $\j\in \I^{\N}$. Let $U\subset \mathbb{R}^{\# \I\cdot d}$ be as in the statement of Theorem \ref{thm:dim}. Then for Lebesgue almost every $\tt\in U$, there exists a constant $c>0$ such that for any $\epsilon>0$ and ball $B$ contained in $X_{\tt}$ we have 
		$$\limsup_{N\to\infty}\L_{d}\left(\bigcup_{\i\in \I^N}S_{\i,\tt}\left(\pi_{\tt}(\j)+\left[\frac{-1}{\lambda(\A)^{(1-\epsilon)N/d}},\frac{1}{\lambda(\A)^{(1-\epsilon)N/d}}\right]^d\right)\cap B\right)\geq c\cdot \L_{d}(B).$$
	\end{prop}
	\begin{proof}
		We begin our proof by fixing a $\tt\in U$ belonging to the full measure set for which the conclusion of Theorem \ref{exponential theorem}.1 holds with $(E_n)$ a sequence of balls with Lebesgue measure $\lambda(\A)^{-n}$ and the targets centred at $\pi_{\tt}(\j)$. By our assumptions, we may also assume that $\tt$ is such that $X_{\tt}$ has non-empty interior. 
		
		We now  fix a ball $B$ contained in $X_{\tt}$. We will show that there exists $c>0,$ that does not depend upon our choice of $B$, such that 
		\begin{equation}
			\label{MTP WTS}
			\limsup_{N\to\infty}\L_{d}\left(\bigcup_{\i\in \I^n}S_{\i,\tt}\left(B\left(\pi_{\tt}(\j),\left(\frac{\log N}{\lambda(\A)^{N}}\right)^{1/d}\right)\right)\cap B\right)\geq c\cdot \L_{d}(B)
		\end{equation}
		%\footnote{To apply this mass transference principle, it may be sensible to change our targets from balls to squares. This makes the rectangle to rectangle scaling clearer.} 
		Since for any $\epsilon>0$ we have $$B\left(\pi_{\tt}(\j),\left(\frac{\log N}{\lambda(\A)^{N}}\right)^{1/d}\right)\subseteq \pi_{\tt}(\j)+\left[\frac{-1}{\lambda(\A)^{(1-\epsilon)N/d}},\frac{1}{\lambda(\A)^{(1-\epsilon)N/d}}\right]^d$$ for $N$ sufficiently large, we see that \eqref{MTP WTS} implies our proposition. We now focus on proving \eqref{MTP WTS}.
		
		Let $C$ be a large cube containing $X_{\tt}$ and $L\in \mathbb{N}$ be sufficiently large that $$Diam(S_{\i}(C))\leq Diam(B)/3$$
		for all $\i \in \I^{L}$. Let $$\Omega_{B}:=\left\{\i\in \I^{L}:S_{\i}(C)\cap B/2\neq\emptyset \right\}.$$ Then we have 
		$$B/2\subseteq \bigcup_{\i \in \Omega_{B}}S_{\i}(C) \subseteq B.$$ Applying Lemma \ref{5r covering lemma} to the rectangles $\{S_{\i}(C):\i\in \Omega_{B'}\},$ we see that there exists $\widetilde{\Omega_{B'}}\subset \Omega_{B'}$ satisfying the following properties:
		\begin{enumerate}
			\item We have $$S_{\i}(C)\cap S_{\j}(C)=\emptyset $$ for distinct $\i,\j\in \widetilde{\Omega_{B'}}.$
			\item $$\bigcup_{\i \in \Omega_{B'}}S_{\i}(C)\subseteq \bigcup_{\i \in \widetilde{\Omega_{B'}}}3\cdot S_{\i}(C).$$
		\end{enumerate}
		By our assumptions on $\tt,$ we know that there exists $c>0$ such that 
		\begin{equation}
			\label{measure infinitely often}
			\L_{d}\left(\bigcup_{\i\in \I^N}S_{\i}\left(\pi_{\tt}(\j)+E_{N} \right)\right)\geq c
		\end{equation} for infinitely many $N,$ where $E_N$ is the ball centred at the origin satisfying $\L_{d}(E_N)=\lambda(\A)^{-N}$. Using properties $1.$ and $2$. above, we see that if $N$ is such that \eqref{measure infinitely often} is satisfied then we have:
		\begin{align*}
			\L_{d}\left(\bigcup_{\i\in \I^{N+L}}S_{\i}\left(\pi_{\tt}(\j)+E_{N} \right)\cap B\right)
			\geq &\L_{d}\left(\bigcup_{\i'\in\widetilde{\Omega_{B}},\, \i''\in \I^{N}}S_{\i'\i''}\left(\pi_{\tt}(\j)+E_{N}  \right)\right)\\
			=&\L_{d}\left(\bigcup_{\i'\in\widetilde{\Omega_{B}}}S_{\i'}\left(\bigcup_{\i''\in \I^{N}}S_{\i''}\left(\pi_{\tt}(\j)+E_{N} \right) \right)\right)\\
			\geq & c\cdot \sum_{\i'\in\widetilde{\Omega_{B}}}Det(A)^{L}\\
			=&\frac{c}{3^d\cdot\L_{d}(C)}\sum_{\i'\in\widetilde{\Omega_{B}}}Det(A)^{L}\cdot 3^d\cdot\L_{d}(C)\\
			=&\frac{c}{3^d\cdot\L_{d}(C)} \sum_{\i'\in\widetilde{\Omega_{B}}}\L_{d}(3\cdot S_{\i'}(C))\\
			\geq&\frac{c}{3^d\cdot\L_{d}(C)}\L_{d}\left(\bigcup_{\i \in \widetilde{\Omega_{B}}}3\cdot S_{\i}(C)\right)\\
			\geq &\frac{c}{3^d\cdot\L_{d}(C)}\L_{d}\left(\bigcup_{\i \in \Omega_{B}}S_{\i}(C)\right)\\
			\geq &\frac{c}{3^d\cdot\L_{d}(C)}\L_{d}(B/2)\gg \frac{c}{3^d\cdot\L_{d}(C)}\L_{d}(B).
		\end{align*}
		Summarising, we have shown that 
		$$\limsup_{N\to\infty}\L_{d}\left(\bigcup_{\i\in \I^{N+L}}S_{\i}\left(\pi_{\tt}(\j)+E_{N} \right)\cap B \right)\gg\L_{d}(B').$$ Equation \eqref{MTP WTS} now follows once we observe that for $N$ sufficiently large, if $$x\in S_{\i}\left(\pi_{\tt}(\j)+E_{N} \right)$$ for some $\i\in \I^{N+L},$ then 
		$$x\in S_{\i}\left(B\left(\pi_{\tt}(\j),\left(\frac{\log (N+L)}{\lambda(\A)^{N+L}}\right)^{1/d}\right)\right).$$ 
	\end{proof}
	Equipped with Proposition \ref{ubiquity prop} we can now prove Theorem \ref{thm:dim}.
	
	\begin{proof}[Proof of Theorem \ref{thm:dim}]
		Fix $\j\in \I^{\N}$. Let $\tt\in U$ belong to the full measure set for which the conclusion of Proposition \ref{ubiquity prop} holds and for which $X_{\tt}$ has non-empty interior. Let $B$ be an arbitrary ball with centre in $X_{\tt}$. Replacing $B$ with a ball contained within $B$ if necessary, we can assume without loss of generality that $B\subset X_{\tt}$. This follows from our assumption on $\tt$ that ensures $X_{\tt}$ has non-empty interior. We fix $s>1$. We now set out to obtain a lower bound for the Hausdorff dimension of $W_{s}(S_{\tt},\pi_{\tt}(\j))\cap B.$  Let $\epsilon>0$ be arbitrary. Instead of directly bounding the Hausdorff dimension of  $W_{s}(S_{\tt},\pi_{\tt}(\j))\cap B$ from below, we will bound the Hausdorff dimension of 
		$$\tilde{W}(\epsilon,B):=B\cap \bigcap_{m=1}^{\infty}\bigcup_{n=m}^{\infty}\bigcup_{\i\in \I^n}S_{\i,\tt}\left(\pi_{\t}(\j)+\left[-\frac{1}{\lambda(\A)^{(s+\epsilon)n/d}},\frac{1}{\lambda(\A)^{(s+\epsilon)n/d}}\right]^d\right)$$
		from below. Crucially $$\tilde{W}(\epsilon,B)\subset W_{s}(S_{\tt},\pi_{\tt}(\j))\cap B$$ for any $\epsilon$. Therefore a lower bound for the Hausdorff dimension of $\tilde{W}(\epsilon,B)$ is also a lower bound for the Hausdorff dimension of $W_{s}(S_{\tt},\pi_{\tt}(\j))\cap B$. We emphasise that $\tilde{W}(\epsilon,B)$ is a limsup set formed of rectangles with centres contained in $\{S_{\i,\tt}(\pi_{\tt}(\j))\}_{\i\in \cup_{n=1}^{\infty}\I^n}$ and side lengths $$\frac{\lambda_{i}^{N}}{\lambda(\A)^{(s+\epsilon)N/d}}$$ for each $1\leq i\leq d$.
		
		Proposition \ref{ubiquity prop} tells us that the sequence of rectangles with centres $\{S_{\i,\tt}(\pi_{\tt}(\j))\}_{\i\in \cup_{n=1}^{\infty}\I^n}$ and side lengths $$\frac{\lambda_{i}^{N}}{\lambda(\A)^{(1-\epsilon)N/d}}$$ for each $1\leq i\leq d$ satisfies the local ubiquity condition as in \cite[Definition 3.1]{WanWu}\footnote{Using the language of \cite{WanWu}, here we are taking $J_{n}=\I^n$ for each $n\in \N$.}. Using the language of Wang and Wu, if we take $\rho_{N}=\lambda(\A)^{-N},$ $$a_{i}(\epsilon)=\frac{\log \lambda_i} {\log\lambda(\A)^{-1}}+\frac{1-\epsilon}{d}\qquad \textrm{ and }\qquad  t_{i}(\epsilon)=\frac{s-1+2\epsilon}{d}$$ then we have 
		$$\rho_{N}^{a_i(\epsilon)}=\frac{\lambda_{i}^{N}}{\lambda(\A)^{(1-\epsilon)N/d}}\qquad \textrm{ and }\rho_{N}^{a_i(\epsilon)+t_i(\epsilon)}=\frac{\lambda_{i}^{N}}{\lambda(\A)^{(s+\epsilon)N/d}}$$ for each $1\leq i \leq d$. Equipped with this notation, we can now directly apply Theorem 3.1 from \cite{WanWu} to obtain
		\begin{align*}
			&\dim_{H}(W_{s}(S_{\tt},\pi_{\tt}(\j))\cap B)\\
			&\geq \dim_{H}(\tilde{W}(\epsilon,B))\\
			&\geq \min_{p\in \mathcal{P(\epsilon)}}\left\{\# K_{1}(\epsilon) + \# K_{2}(\epsilon)+\frac{\sum_{i\in K_{3}(\epsilon)}a_i-\sum_{k\in K_{2}}t_{i}}{p}\right\}\\
			&=\min_{p\in \mathcal{P}(\epsilon)}\left\{\# K_{1}(\epsilon) + \# K_{2}(\epsilon)\left(1-\frac{s-1+2\epsilon}{dp}\right)+\frac{\# K_{3}(\epsilon)(1-\epsilon)}{dp}+\sum_{i\in K_{3}(\epsilon)}\frac{\log \lambda_i} {p\log\lambda(\A)^{-1}}\right\}
		\end{align*}
		where $P(\epsilon)=\{a_i(\epsilon),a_i(\epsilon)+t_i(\epsilon):1\leq i\leq d\}$ and $$K_{1}(\epsilon)=\{i:a_i\geq p\},\, K_{2}(\epsilon)=\{i:a_i+t_i\leq p\},\, K_{3}(\epsilon)=\{1,\ldots,d\}\setminus (K_{1}(\epsilon)\cup K_{2}(\epsilon)).$$
		It is a simple, albeit tedious calculation, to check that \begin{align*}
			&\lim_{\epsilon \to 0}\min_{p\in \mathcal{P}(\epsilon)}\left\{\# K_{1}(\epsilon) + \# K_{2}(\epsilon)\left(1-\frac{s-1+2\epsilon}{dp}\right)+ \frac{\#K_{3}(\epsilon)(1-\epsilon)}{dp}+\sum_{i\in K_{3}(\epsilon)}\frac{\log \lambda_i} {p\log\lambda(\A)^{-1}}\right\}\\
			=&\min_{p\in P}\left\{\# K_{1} + \# K_{2}\left(1-\frac{s-1}{dp}\right)+\frac{\# K_{3}}{dp}+\sum_{i\in K_{3}}\frac{\log \lambda_i} {p\log\lambda(\A)^{-1}}\right\}
		\end{align*}
		where $$a_{i}=\frac{\log \lambda_i} {\log\lambda(\A)^{-1}}+\frac{1}{d}$$%\qquad \textrm{ and }\qquad  t_{i}=\frac{s-1}{d},$$
		$P=\{a_i,a_i+\frac{s-1}{d}:1\leq i\leq d\}$ and $$K_{1}=\{i:a_i\geq p\},\, K_{2}=\left\{i:a_i+\frac{s-1}{d}\leq p\right\},\, K_{3}=\{1,\ldots,d\}\setminus (K_{1}\cup K_{2}).$$ Therefore we have $$\dim_{H}(W_{s}(S_{\tt},\pi_{\tt}(\j))\cap B)\geq \min_{p\in P}\left\{\# K_{1} + \# K_{2}\left(1-\frac{s-1}{dp}\right)+\frac{\# K_{3}}{dp}+\sum_{i\in K_{3}}\frac{\log \lambda_i} {p\log\lambda(\A)^{-1}}\right\}$$ and our proof is complete.
	\end{proof}
	\section{Proofs of Theorems \ref{Garsia thm} and \ref{GarsiaMTP}}
	\label{Garsia section}
	The key technical result which allows us to prove Theorem \ref{Garsia thm} is the following lemma. Before stating it, we introduce some useful notation. Given $\lambda\in(1/2,1),$ we let $\pi_{\lambda}:\{0,1\}^{\N}\to [0,\frac{1}{1-\lambda}]$ be given by $$\pi_{\lambda}((a_i))=\sum_{i=1}^{\infty}a_i\lambda^{i-1}.$$
	\begin{lemma}
		\label{Garsia separation}
		Let $\lambda\in(1/2,1)$ be such that $\lambda^{-1}$ is a Garsia number. Then there exists $C>0$ such that for any $\a,\a'\in\{0,1\}^n$ satisfying $\a\neq \a',$ we have that $$\left|\pi_{\lambda}(\a^{\infty})-\pi_{\lambda}(\a'^{\infty})\right|\geq \frac{C}{2^n}.$$
	\end{lemma}
	Lemma \ref{Garsia separation} is well known for differences of the form $\pi_{\lambda}(\a0^{\infty})-\pi_{\lambda}(\a'0^{\infty})$ (see \cite{Gar}). Our proof of Lemma \ref{Garsia separation} is a minor adaptation but we include the details for completion.
	
	\begin{proof}
		Let us fix $\lambda\in(1/2,1)$ the reciprocal of a Garsia number, and $\a,\a'\in\{0,1\}^n$ two distinct sequences. We observe that 
		\begin{align*}
			\left|\pi_{\lambda}(\a^{\infty})-\pi_{\lambda}(\a'^{\infty})\right|&=\left|\sum_{k=0}^{\infty}\lambda^{kn}\left(\sum_{i=1}^{n}a_i\lambda^{i-1}\right)-\sum_{k=0}^{\infty}\lambda^{kn}\left(\sum_{i=1}^{n}a_i'\lambda^{i-1}\right)\right|\\
			&=\frac{1}{1-\lambda^n}\left|\sum_{i=1}^{n}a_i\lambda^{i-1}-\sum_{i=1}^{n}a_i'\lambda^{i-1}\right|.
		\end{align*}
		Therefore to prove our lemma, we need to show that there exists $C>0$ which does not depend upon $\a$ or $\a',$ such that\footnote{This is how the separation property satisfied by Garsia numbers is usually formulated.} 
		\begin{equation}
			\label{NTS}
			\left|\sum_{i=1}^{n}a_i\lambda^{i-1}-\sum_{i=1}^{n}a_i'\lambda^{i-1}\right|\geq \frac{C}{2^n}.
		\end{equation}  Let $\beta=\lambda^{-1}$ and let $\gamma_{1},\ldots,\gamma_{k}$ denote the Galois conjugates of $\beta$. Since $\beta$ has norm $\pm 2$ it is not a zero of any non-trivial polynomials with coefficients in $\{-1,0,1\}$. Therefore $\sum_{i=1}^{n}a_i\beta^{n-i+1}-\sum_{i=1}^{n}a_i'\beta^{n-i+1}\neq 0$. Moreover we also have $\sum_{i=1}^{n}a_i\gamma_j^{n-i+1}-\sum_{i=1}^{n}a_i'\gamma_j^{n-i+1}\neq 0$ for each Galois conjugate of $\beta$. Since $\beta$ is an algebraic integer we have
		\begin{align*}
			1&\leq \left|\sum_{i=1}^{n}a_i\beta^{n-i+1}-\sum_{i=1}^{n}a_i'\beta^{n-i+1}\right| \cdot \prod_{j=1}^{k}\left| \sum_{i=1}^{n}a_i\gamma_j^{n-i+1}-\sum_{i=1}^{n}a_i'\gamma_j^{n-i+1}\right|\\
			&\ll  \left|\sum_{i=1}^{n}a_i\beta^{n-i+1}-\sum_{i=1}^{n}a_i'\beta^{n-i+1}\right|\cdot \prod_{j=1}^{k}\left|\gamma_{j}^n\right|\\
			&=\left|\sum_{i=1}^{n}a_i\lambda^{i-1}-\sum_{i=1}^{n}a_i'\lambda^{i-1}\right|\cdot |\beta|^{n}\cdot \prod_{j=1}^{k}\left|\gamma_{j}^n\right|\\
			&=\left|\sum_{i=1}^{n}a_i\lambda^{i-1}-\sum_{i=1}^{n}a_i'\lambda^{i-1}\right|\cdot 2^n.
		\end{align*}
		The last line follows from the fact $\beta$ has norm $\pm 2$ so $|\beta\cdot \prod_{j=1}^{k}\gamma_j|=2.$ We have shown that $$\left|\sum_{i=1}^{n}a_i\lambda^{i-1}-\sum_{i=1}^{n}a_i'\lambda^{i-1}\right|\cdot 2^n\gg 1.$$ Our result now follow upon dividing both sides by $2^n$.
	\end{proof}
	The following lemma immediately follows from Lemma \ref{recurrence to shrinking}. It is essentially the first part of this lemma rewritten for our current purposes. 
	
	\begin{lemma}
		\label{Recurrence interpretation}
		Let $\lambda\in(1/2,1)$ and $\a\in \{0,1\}^n$. Then $x\in B\left(\pi_{\lambda}(\overline{\a}^{\infty}),\frac{\lambda^{|\a|} r}{1-\lambda^{|\a|}}\right)$ if and only if $T_{\a}(x)\in B(x,r)$. 
	\end{lemma}
	%\begin{proof}
	%We observe $$|T_{\a}(x) -x|\leq |T_{\a}(x) -\pi_{\lambda}(\a^{\infty})| + |\pi_{\lambda}(\a^{\infty}) -x|$$ Now using the facts $\pi_{\lambda}(\a^{\infty})$ is fixed by $T_{\a}$ and both $T_{0}$ and $T_1$ scale distance by a factor $\lambda^{-1},$ we have $$|T_{\a}(x) -x|<\lambda^{-n}r+r=(\lambda^{-n}+1)r.$$
	%\end{proof}
	
	We are almost ready to proceed with our proof of Theorem \ref{Garsia thm}. However, before we do, we need to recall two results.
	\begin{thm}[Garsia \cite{Gar}]
		\label{Garsia's thm}
		If $\lambda$ is the reciprocal of a Garsia number then the Bernoulli convolution $\mu_{\lambda}$ is absolutely continuous.
	\end{thm}
	\begin{thm}[Mauldin and Simon \cite{MauSim}]
		\label{Equivalent measures}
		If $\mu_{\lambda}$ is absolutely continuous then $\mu_{\lambda}$ is equivalent to $\L_{1}|_{[0,\frac{1}{1-\lambda}]}.$
	\end{thm}
	
	\begin{proof}[Proof of Theorem \ref{Garsia thm}]
		Since $\lambda$ is the reciprocal of a Garsia number, we know by Theorem \ref{Garsia's thm} that $\mu_{\lambda}$ is absolutely continuous with respect to the Lebesgue measure. Therefore it admits a Radon-Nikodym derivative $d_{\lambda}$. By definition we have that $d_{\lambda}(x)>0$ for $\mu_{\lambda}$ almost every $x$. Now using Theorem \ref{Equivalent measures}, we in fact have that $d_{\lambda}(x)>0$ for Lebesgue almost every $x\in [0,\frac{1}{1-\lambda}]$. By Lebesgue's differentiation theorem \cite{Mat} we know that 
		\begin{equation}
			\label{Lebesgue differentiation thm}
			\lim_{r\to 0}\frac{\mu_{\lambda}(B(x,r))}{2r}=d_{\lambda}(x)
		\end{equation} for Lebesgue almost every $x$. Let us now fix an arbitrary $x'$ satisfying 
		\begin{equation}
			\label{Full measure satisfied}
			d_{\lambda}(x')>0\textrm{ and }\eqref{Lebesgue differentiation thm}.
		\end{equation} Given any function $h:\mathbb{N}\to [0,\infty)$ satisfying $\sum_{n=1}^{\infty}h(n)=\infty$, we will show that there exists $C>0$ such that for any $r$ sufficiently small we have 
		\begin{equation}
			\label{Want to show}
			\frac{\L_{1}(R(S_{\lambda},h)\cap B(x',  r))}{2r}\geq C.
		\end{equation} Since almost every $x\in [0,\frac{1}{1-\lambda}]$ satisfies \eqref{Full measure satisfied}, it will follow from \eqref{Want to show} that the set of density points for $R(S_{\lambda},h)^c \cap [0,\frac{1}{1-\lambda}]$ has zero Lebesgue measure. By Theorem \ref{Lebesgue density theorem} it will then follow that Lebesgue almost every $x\in [0,\frac{1}{1-\lambda}]$ is contained in $R(S_{\lambda},h)$. As such to prove our result it is sufficient to prove that \eqref{Want to show} holds.

		We know that $x'$ satisfies \eqref{Lebesgue differentiation thm} and $d_{\lambda}(x')>0$. Therefore there exists $r^*>0$ such that for all $r\in (0,r^*)$ we have 
		\begin{equation}
			\label{density bounds}
			d_{\lambda}(x')r\leq \mu_{\lambda}(B(x',r))\leq 4d_{\lambda}(x')r.
		\end{equation}
		In what remains of our proof we will assume that $r\in (0,r^*)$ is fixed. Now we note that $\mu_{\lambda}$ is the weak star limit of the sequence of measures $(\mu_{\lambda,n})_{n=1}^{\infty}$ where for each $n\in \N$ we let $$\mu_{\lambda,n}:=\frac{1}{2^n}\sum_{\a\in \{0,1\}^n}\delta_{\pi_{\lambda}(\a^{\infty})}.$$ Using \eqref{density bounds} together with the definition of $\mu_{\lambda,n}$, we see that there exists $N\in\mathbb{N}$ such that if we let $$\Omega_{r,n}:=\{\a\in\{0,1\}^n:\pi_{\lambda}(\a^{\infty})\in B(x',r)\}$$ then \begin{equation}
			\label{density count}
			\frac{2^nd_{\lambda}(x')r}{2}\leq \# \Omega_{r,n}\leq 8\cdot 2^nd_{\lambda}(x')r
		\end{equation} for all $n\geq N$. 
		
		Let $h:\N\to[0,\infty)$ be an arbitrary function satisfying $\sum_{n=1}^{\infty}h(n)=\infty$.  For each $n\geq N$ we let $$E_{n}=\bigcup_{\a\in \Omega_{r,n}}B\left(\pi_{\lambda}(\a^{\infty}),\frac{h(n)}{2^{n+1}}\right).$$ Replacing $h$ with a smaller function is necessary, it follows from Lemma \ref{Garsia separation} that without loss of generality we can assume that the union defining $E_{n}$ is always disjoint. We then let $$E_{\infty}=\bigcap_{m=N}^{\infty}\bigcup_{m=n}^{\infty} E_{n}.$$ We may also assume without loss of generality that $h$ is bounded. As such, it follows from Lemma \ref{Recurrence interpretation} that $E_{\infty}\subset R(S_{\lambda},h)\cap [x'-r,x'+r]$. Therefore to prove that \eqref{Want to show} holds, we need to show that there exists $C>0$ such that 
		\begin{equation}
			\label{WTSGarsia}
			\L_{1}(E_{\infty})\geq C2r.
		\end{equation} To prove this inequality we use Lemma \ref{Quasi-independence on average}. Since the balls in the union defining $E_n$ are disjoint, it follows from \eqref{density count} that  
		\begin{equation}
			\label{Enmeasure}
			\frac{h(n)rd_{\lambda}(x')}{4}\leq \L_{1}(E_n)\leq 4h(n)rd_{\lambda}(x'). 
		\end{equation}Which by our assumption on $h$ implies $\sum_{n=N}^{\infty}\L_{1}(E_n)=\infty$. So we satisfy the assumptions of Lemma \ref{Quasi-independence on average}. It remains to obtains good bounds for $\L_{1}(E_n\cap E_m)$. For a fixed $\a\in \Omega_{r,n}$, it follows from Lemma \ref{Garsia separation} and a volume argument that for $m>n\geq N$ we have $$\#\left\{\b\in \Omega_{r,m}:B\left(\pi_{\lambda}(\a^{\infty}),\frac{h(n)}{2^{n+1}}\right)\cap B\left(\pi_{\lambda}(\b^{\infty}),\frac{h(m)}{2^{m+1}}\right)\neq \emptyset\right\}\ll \frac{h(n)}{2^{n+1}}\cdot 2^m +1.$$ Using this bound and \eqref{density count} we have
		\begin{align*}
			\L_{1}(E_n\cap E_m)=&\sum_{\a\in \Omega_{r,n}}\L_{1}\left(B\left(\pi_{\lambda}(\a^{\infty}),\frac{h(n)}{2^{n+1}}\right)\cap E_m\right)\\
			\leq& \sum_{\a\in \Omega_{r,n}}\#\left\{\b\in \Omega_{r,m}:B\left(\pi_{\lambda}(\a^{\infty}),\frac{h(n)}{2^{n+1}}\right)\cap B\left(\pi_{\lambda}(\b^{\infty}),\frac{h(m)}{2^{m+1}}\right)\neq \emptyset\right\}\cdot \frac{h(m)}{2^{m}}\\
			\ll& \sum_{\a\in \Omega_{r,n}} \frac{h(n)h(m)}{2^n}+\frac{h(m)}{2^m}\\
			=&\frac{\# \Omega_{r,n}h(n)h(m)}{2^n}+\frac{\# \Omega_{r,n}h(m)}{2^m}\\
			\ll &h(n)h(m)d_{\lambda}(x')r+\frac{h(m)d_{\lambda}(x')r}{2^{m-n}}.
		\end{align*}
		Combining the bound above with \eqref{Enmeasure} we have
		\begin{align*}
			\sum_{n,m=N}^{Q}\L_{1}(E_n\cap E_m)&=\sum_{n=N}^{Q}\L_{1}(E_n)+2\sum_{n=N}^{Q-1}\sum_{m=n+1}^{Q}\L_{1}(E_n\cap E_m)\\
			&\ll rd_{\lambda}(x')\sum_{n=N}^{Q}h(n)+\sum_{n=N}^{Q-1}\sum_{m=n+1}^{Q}\left(h(n)h(m)d_{\lambda}(x')r+\frac{h(m)d_{\lambda}(x')r}{2^{m-n}}\right)\\
			&\leq rd_{\lambda}(x')\left(\sum_{n=N}^{Q}h(n)+\left(\sum_{n=N}^{Q}h(n)\right)^2+\sum_{n=N}^{Q-1}\sum_{m=n+1}^{Q}\frac{h(m)}{2^{m-n}}\right)\\
			&\leq rd_{\lambda}(x')\left(\sum_{n=N}^{Q}h(n)+\left(\sum_{n=N}^{Q}h(n)\right)^2+\sum_{m=N+1}^{Q}\sum_{n=N}^{m-1}\frac{h(m)}{2^{m-n}}\right)\\
			&= rd_{\lambda}(x')\left(\sum_{n=N}^{Q}h(n)+\left(\sum_{n=N}^{Q}h(n)\right)^2+\sum_{m=N+1}^{Q}h(m)\sum_{n=N}^{m-1}\frac{1}{2^{m-n}}\right)\\
			&\leq rd_{\lambda}(x')\left(\sum_{n=N}^{Q}h(n)+\left(\sum_{n=N}^{Q}h(n)\right)^2+\sum_{m=N+1}^{Q}h(m)\right)\\
			&\ll rd_{\lambda}(x')\left(\sum_{n=N}^{Q}h(n)+\left(\sum_{n=N}^{Q}h(n)\right)^2\right).
		\end{align*}
		In the penultimate line we used that $\sum_{n=N}^{m-1}\frac{1}{2^{m-n}}\leq 1$ for all $m> N$. Now using the above, Lemma \ref{Quasi-independence on average} and \eqref{Enmeasure}, we have 
		\begin{align*}
			\L_{1}(E_\infty)&\geq \limsup_{Q\to \infty}\frac{(\sum_{n=N}^{Q}\L_{1}(E_n))^2}{\sum_{n,m=N}^{Q}\L_{1}(E_n\cap E_m)}\\
			&\geq \limsup_{Q\to \infty}\frac{(\frac{rd_{\lambda}(x')}{4})^2\left(\sum_{n=N}^{Q}h(n)\right)^2}{rd_{\lambda}(x')\left(\sum_{n=N}^{Q}h(n)+\left(\sum_{n=N}^{Q}h(n)\right)^2\right)}\\
			&\gg rd_{\lambda}(x').
		\end{align*}
		In the final line we used that $\sum_{n=N}^{\infty}h(n)=\infty$, so $$\lim_{Q\to\infty}\frac{\sum_{n=N}^{Q}h(n)}{\left(\sum_{n=N}^{Q}h(n)\right)^2}=0.$$ In summary, we have shown that $\L_{1}(E_{\infty})\gg rd_{\lambda}(x').$ This is exactly the content of \eqref{WTSGarsia} and so our proof is complete. 
	\end{proof}
	We now move on to the proof of Theorem \ref{GarsiaMTP}. The application of the mass transference principle from \cite{BerVel} is standard so we only give brief details. 
	
	\begin{proof}
		The first part of Theorem \ref{GarsiaMTP} follows from Lemma \ref{Recurrence interpretation} and a covering argument. For the second part, notice that Lemma \ref{Recurrence interpretation} implies that for any function $h:\N\to [0,\infty)$ we have
		$$R(S_{\lambda},h):=\bigcap_{m=1}^{\infty}\bigcup_{n=m}^{\infty}\bigcup_{\a\in \{0,1\}^n}B\left(\pi_{\lambda}(\a^{\infty}),\frac{h(n)}{(1-\lambda^n)2^n}\right).$$
		This equality allows us to reinterpret the set $R(S_\lambda,h)$ in terms of a limsup set coming from a sequence of balls. With this reinterpretation, we naturally fall into the framework of \cite{BerVel}. In particular, combining Theorem 3 from \cite{BerVel} with Theorem \ref{Garsia thm} we obtain our result.
	\end{proof}
	\noindent \textbf{Acknowledgements.} This research was supported by an EPSRC New Investigator Award (EP/W003880/1). The authors thank the anonymous referee for their feedback.

	\[
	%K_1=\{k\mid a_k\ge p\}, \quad K_2=\{k\mid a_k+\tfrac sd \le p\}\setminus K_1, \quad K_3=\{1, \dots, d\}\setminus (K_1\cup K_2). 
	\]
	%It is then a direct consequence of Theorem 3.1 in \cite{WanWu} that 
	%\begin{equation}\label{eq:lower}
	%\begin{aligned}
	%\dim W&(\mathcal S_\tt, \pi_\tt(\j), s)\ge \min_{p\in \mathcal P}\left\{ \#K_1+\#K_2+\frac{\sum_{k\in K_3}a_k-\sum_{k\in K_2} t_k}{p}\right\}\\
	%&= \min_{p\in \mathcal P}\left\{ \#K_1+\#K_2+\frac{(\log \lambda(\A)^{-1})^{-1}\sum_{k\in K_3}\log \sigma_k+(\#K_2+\#K_3)\delta-\#K_2\cdot \tfrac sd}{p}\right\}. 
	%\end{aligned}
	%\end{equation}
	%Now, denote $b_k=\log \sigma_k/\log \lambda(\A)^{-1}+\tfrac sd$. Notice that, if $|\delta - \tfrac 1d|$ satisfies 
	%\[
	%|\delta - \tfrac 1d|<\min\left\{ \left|\frac{\log \sigma_i - \log \sigma_k}{\log(\lambda(\A)^{-1})}+\frac {s-1}d \right|\mid i\neq k\right\}, 
	%\]
	%then the sets $K_i, i=1,2,3$ can all be defined using $b_k$ in the place of $a_k$, and $(s-1)/d$ in place of $t_k$: For every $p\in \mathcal P':=\{b_1, \dots, b_d, b_1+\tfrac {s-1}d, \dots, b_d+\tfrac {s-1}d \}$
	%\[
	%K_1=\{k\mid b_k\ge p\}, \quad K_2=\{k\mid b_k+\tfrac {s-1}d \le p\}\setminus K_1, \quad K_3=\{1, \dots, d\}\setminus (K_1\cup K_2). 
	%\]
	%Furthermore, as $\delta<\tfrac 1d$ was arbitrary, we obtain from $\eqref{eq:lower}$: 
	%\[
	%\dim W(\mathcal S_\tt, \pi_\tt(\j), s)\ge \min_{p\in \mathcal P'}\left\{\#K_1+ \#K_2+\frac{(\log \lambda(\A)^{-1})^{-1}\sum_{k\in K_3}\log \sigma_k}{p}+\frac{\#K_3+\#K_2(1-s)}{dp}\right\}. 
	%\]

\end{document}